\newif\ifHAL
\HALtrue

\ifHAL
\documentclass[10pt,a4paper]{article}
\else
\documentclass[10pt]{siamltex}
\fi

\usepackage[colorlinks,citecolor=cyan,linkcolor=magenta]{hyperref}
\usepackage[latin1]{inputenc}
\usepackage{amsmath}
\usepackage{amsfonts}
\usepackage{amssymb}
\usepackage{bm}
\usepackage{subcaption}
\usepackage{graphicx}
\usepackage{xspace}
\usepackage{tikz}
\usepackage{fullpage}

\newcommand{\Rev}[1]{{#1}}

\newcommand{\RR}{\mathbb{R}}      % for Real numbers
\newcommand{\vertiii}[1]{{\|\kern-0.25ex | #1
		| \kern-0.25ex \|}}

\newcommand{\mean}[1]{\{\kern-1.1mm\{#1\}\kern-1.1mm\}}                  % mean value
\newcommand{\ndg}[1]{| \kern -.25mm \|{#1}| \kern -.25mm \|}
\newcommand{\nsdg}[1]{| \kern -.25mm \|{#1}| \kern -.25mm \|_{\rm s}}
\newcommand{\su}{\sum_{K\in \mesh}}

\newcommand{\fes}{\hat{V}_{h}^k}
\newcommand{\fesz}{\hat{V}_{h0}^k}
\newcommand{\fesE}{\hat{V}_{K}^k}
\newcommand{\n}{{\bm n}}
\newcommand{\Fall}{\mathcal{F}_h}

\newcommand{\mesh}{\mathcal{T}_h}

\newcommand{\ncdg}[1]{| \kern -.25mm \|{#1}| \kern -.25mm \|_{\rm DG}}

\renewcommand{\tilde}[1]{\widetilde{#1}}
\renewcommand{\hat}[1]{\widehat{#1}}
\makeatletter
\newcommand*{\rom}[1]{\text{\expandafter\@slowromancap\romannumeral #1@}}
\makeatother

\ifHAL
\usepackage{theorem}
\newtheorem{corollary}{Corollary}[section]
\newtheorem{lemma}[corollary]{Lemma}
\newtheorem{theorem}[corollary]{Theorem}

\newtheorem{remark}[corollary]{Remark}

\newcommand{\qed}{ \vspace{-0.5cm} \hfill $\Box$ }
\newenvironment{proof}[1][Proof.]{\begin{trivlist}
		\item[\hskip \labelsep {\bfseries #1}]}{\end{trivlist}\qed}
\else
\newtheorem{remark}[theorem]{Remark}
\fi

\newcommand{\upi}{^{\mathrm{i}}}
\newcommand{\upb}{^{\mathrm{b}}}
\newcommand{\Fb}{\mathcal{F}_h\upb}
\newcommand{\Fint}{\mathcal{F}_h\upi}
\newcommand{\dK}{\partial K}
\newcommand{\dKi}{\partial K\upi}
\newcommand{\dKb}{\partial K\upb}
\newcommand{\FK}{\mathcal{F}_{\dK}}
\newcommand{\Ihk}{\mathcal{\hat{I}}_h^k}
\newcommand{\FKi}{\mathcal{F}_{\dK\upi}}
\newcommand{\FKb}{\mathcal{F}_{\dK\upb}}
\newcommand{\meshi}{\mesh\upi}
\newcommand{\meshb}{\mesh\upb}
\newcommand{\bG}{{\bm G}}

\begin{document}
\title{Hybrid high-order and weak Galerkin methods for the biharmonic problem}
\author{
	Zhaonan Dong\thanks{
	Inria, 2 rue Simone Iff, 75589 Paris, France,
	and CERMICS, Ecole des Ponts, 77455 Marne-la-Vall\'{e}e 2, France
	{\tt{zhaonan.dong@inria.fr}}.}
	\and
        Alexandre Ern\thanks{
	CERMICS, Ecole des Ponts, 77455 Marne-la-Vall\'{e}e 2, France,
	and  Inria, 2 rue Simone Iff, 75589 Paris, France
	{\tt{alexandre.ern@enpc.fr}}.}
        }
\date{\today}

\maketitle

\begin{abstract}
We devise and analyze two hybrid high-order (HHO) methods for the numerical approximation of the biharmonic problem. The methods support polyhedral meshes, rely on the primal formulation of the problem, and deliver $O(h^{k+1})$ $H^2$-error estimates when using polynomials of order $k\ge0$ to approximate the normal derivative on the mesh (inter)faces. Both HHO methods hinge on a stabilization in the spirit of Lehrenfeld--Sch\"oberl for second-order PDEs. The cell unknowns are polynomials of order $(k+2)$ that can be eliminated locally by means of static condensation. The face unknowns approximating the trace of the solution on the mesh (inter)faces are polynomials of order $(k+1)$ in the first HHO method which is valid in dimension two and uses an original stabilization involving the canonical hybrid finite element, and they are of order $(k+2)$ for the second HHO method which is valid in arbitrary dimension and uses only $L^2$-orthogonal projections in the stabilization. A comparative discussion with the weak Galerkin methods from the literature is provided, highlighting the close connections and the improvements proposed herein. Additionally, we show how the two HHO methods can be combined with a Nitsche-like boundary-penalty technique to weakly enforce the boundary conditions. An originality in the devised Nitsche's technique is to avoid any penalty parameter that must be large enough. Finally, numerical results showcase the efficiency of the proposed methods, and indicate that the HHO methods can generally outperform discontinuous Galerkin methods and even be competitive with $C^0$-interior penalty methods on triangular meshes.
\end{abstract}

\ifHAL
\else
\markboth{Z. DONG, A. ERN}{HHO and WG methods for the biharmonic problem}
\fi

\section{Introduction} \label{Introduction}

Fourth-order PDEs are encountered in the modeling of various physical phenomena,
such as plate bending, thin-plate elasticity,
microelectromechanical systems, and the Cahn--Hilliard phase-field model.
In the present work, we are concerned with the following model problem:
\begin{equation}\label{pde}
\begin{alignedat}{2}
\Delta^2 u  &=f &\qquad&\text{in $\Omega$}, \\
u & = 0  &\qquad&\text{on $\partial \Omega$}, \\
\partial_n u  &= 0  &\qquad&\text{on $\partial \Omega$},
\end{alignedat}
\end{equation}
where $\Omega$ is an open, bounded, polytopal, Lipschitz set in $\mathbb{R}^d$, $d\ge2$, with boundary $\partial \Omega$, the load $f$ is in $L^2(\Omega)$, and $\partial_n$ denotes the normal derivative on $\partial \Omega$. Non-homogeneous boundary conditions
\Rev{and a boundary condition on the second-normal derivative
can be readily incorporated}. 
Instead, considering more singular loads
is nontrivial for the present purpose. \Rev{We also emphasize that the present developments hinge on the weak formulation of \eqref{pde} involving the Hessian.} 

The main goal of this work is to devise and analyze a discretization method
for~\eqref{pde} offering two main features:
(i) it supports polyhedral meshes (the mesh cells
can be polyhedra as such or have a simple shape but contain hanging nodes);
(ii) it hinges on the primal formulation of
the problem, thereby leading to a symmetric positive definite system matrix.
There are already some methods in the literature achieving these goals.
These methods can be loosely classified into three groups, depending on
the dimension of the smallest geometric object to which discrete
unknowns are attached. This criterion is relevant since it
influences the stencil of the method, and it also influences the level
of conformity that can be achieved in the approximation of the
solution. The methods in the first group were developed in the
practically important case where $d=2$. They
attach discrete unknowns to the mesh vertices, edges, and cells and
can achieve $C^1$-conformity. Salient examples are the
$C^1$-conforming virtual element methods (VEM) from \cite{BreMa:13,ChiMa:16} and the
$C^0$-conforming VEM from \cite{ZhChZ:16}. Another example of method in this group
is the nonconforming VEM from
\cite{AnMaV:18,ZhZCM:18} where the approximation is, however, (fully) nonconforming.
The methods in the second group attach discrete
unknowns only to the mesh faces and cells for $d\ge2$.
They are amenable to static
condensation (meaning that the cell unknowns can be eliminated locally
leading to a global problem coupling only the face unknowns),
and they provide a nonconforming approximation to the solution.
The two salient examples are the weak Galerkin (WG) methods from
\cite{MuWaY:14,ZhaZh:15,YeZhZ:20} and the hybrid high-order (HHO) method
from \cite{BoDPGK:18}. Finally, the methods in the third group attach
discrete unknowns only to the mesh cells for $d\ge2$ and belong to the class
of interior penalty
discontinuous Galerkin (IPDG) methods. These are also nonconforming methods, and they
were developed for the model problem~\eqref{pde} in \cite{MozSu:03,SulMo:07,GeoHo:09}.
We mention that on specific meshes composed of simplices or cuboids,
there are variants of the above methods achieving $C^0$-conformity, such as the
$C^0$-WG method from \cite{MuWYZ:14,CheFe:16} and the $C^0$-IPDG
from \cite{EGHLMT:02,BrennerC0}. Furthermore, important examples of
nonconforming finite elements on simplicial meshes are the Morley element
\cite{morley1968,WangXu:06} and the Hsieh--Clough--Tocher (HCT) element
(see, e.g., \cite[Chap.~6]{ciarlet:02}).

In the present work, we focus on HHO methods. HHO methods were introduced in
\cite{DiPEr:15} for locking-free linear elasticity and in \cite{DiPEL:14}
for linear diffusion. The two ingredients of HHO methods are
a local reconstruction operator and a local stabilization operator in each mesh cell.
For second-order PDEs,
the aim of the first operator is to reconstruct locally a
gradient from the cell and the face
unknowns, and the aim of the second operator is to penalize in a least-squares sense
the difference between the trace of the cell unknown and the face unknown on every mesh face.
HHO methods have undergone a vigorous development in the last few years; to cite
a few examples, we mention
Navier--Stokes flows \cite{DiPKr:18}, elastoplastic problems
\cite{AbErPi:19}, Tresca friction problems \cite{ChErPi:20},
spectral problems \cite{CaCDE:19}, and magnetostatics \cite{ChDPL:20}.
HHO methods were embedded into the broad framework of hybridizable dG (HDG)
methods in \cite{CoDPE:16} by reformulating the HHO equations as local balance
equations with equilibrated numerical fluxes. Moreover, HHO methods are closely related to
WG methods, which were also embedded into the HDG framework in \cite[Sec.~6.6]{Cockburn:16}.
The reconstruction operator in the HHO method corresponds to the weak gradient in
WG methods. Hence, HHO and WG methods
differ only in the choice of the discrete unknowns and in the design of stabilization.
Although the close connections between HHO and WG methods should be mutually beneficial,
these connections are, in the authors' opinion, not sufficiently explicit in the
literature, and the title of the present work is also
meant to draw the community's attention on this opportunity.

The design of the stabilization turns out to be a key ingredient so that the method
leads to \emph{optimal} error estimates. By this, we mean, in the case of a
second-order elliptic PDE,
that the method delivers an $O(h^{k+1})$ $H^1$-error estimate, where $k\ge0$ is the
degree of the face unknowns. Notice that this criterion is consistent with
the classical properties of hybridized mixed finite element methods.
The point we want to make here is that optimality cannot be
reached on general meshes
if one uses plain least-squares stabilization, i.e., a more subtle design of the
stabilization is required.
If the cell unknowns are of degree $k$, optimality is
achieved in \cite{DiPEr:15,DiPEL:14} by means of a stabilization that uses
the reconstruction operator (this is the first occurrence of this
idea in the broad framework of HDG methods). Alternatively, if the cell
unknowns are of degree $(k+1)$, one can use the
Lehrenfeld--Sch\"oberl (LS) stabilization \cite{lehsc:16}, as
in \cite{CoDPE:16} for
HHO methods and in \cite{MuWaY:15} for WG methods.
Although the LS stabilization does not use the reconstruction operator,
it is not a plain least-squares stabilization, since
an orthogonal projection is applied to the trace of the cell unknowns.
We mention that it is also
possible to achieve optimality \emph{without stabilization} for second-order PDEs
if one uses Raviart--Thomas functions of degree $k$ to reconstruct the gradient
(see \cite{AbErPi:18,DiPDM:18}). However, optimality is lost if one
reconstructs the gradient in larger polynomial spaces (the convergence rate is
in general $O(h^k)$), since the normal
component of the reconstructed gradient on the mesh faces is too rich to be captured
by the face unknowns.
Another possibility is to enrich the space for the gradient reconstruction by
suitable bubble functions based on the notion of $M$-decomposition devised for HDG methods \cite{CoFuS:17}.

\begin{table}[ht]
\begin{center}
\begin{tabular}{|l|ccc|c|l|}
\hline
unknowns&cell&face&grad&$k$&ref.\\
\hline\hline
WG&$k+2$&$k+2$&$[k+1]^d$&$k\ge0$&\cite{MuWaY:14}\\
&$k+2$&$k+2$&$k+1$&$k\ge0$&\cite{MuWaY:14}\\
&$k+2$&$k+1$&$k+1$&$k\ge0$&\cite{ZhaZh:15}\\
&1&1&$[1]^d$&$k=0$&\cite{YeZhZ:20}\\
\hline\hline
HHO&$k$&$k$&$[k]^d$&$k\ge1$&\cite{BoDPGK:18}\\
\hline
HHO(A)&$k+2$&$k+1$&$k$&$k\ge0$&present ($d=2$)\\
HHO(B)&$k+2$&$k+2$&$k$&$k\ge0$&present ($d\ge2$)\\
\hline
\end{tabular}\end{center}
\caption{Discrete unknowns in HHO and WG methods from the literature and the present work.
In the column `grad', the notation $[\cdot]^d$ means that the full gradient is approximated;
otherwise, only the normal derivative is approximated. For all the methods,
the integer $k$ is fixed by
the fact that the method
delivers an $O(h^{k+1})$ $H^2$-error estimate.}
\label{tab:HHO_WG}
\end{table}

To discretize fourth-order PDEs, HHO and WG methods use cell unknowns that
are meant to approximate the solution in each mesh cell, face unknowns that
are meant to approximate its trace on each mesh (inter)face,
and additional face unknowns that are meant to approximate either its full gradient trace
or only its normal derivative on each mesh (inter)face. The HHO and WG methods
from the literature and the present HHO methods are described in Table~\ref{tab:HHO_WG}
in terms of their discrete unknowns. To put all the methods on the same ground
and allow for a fair comparison, the polynomial degree
$k$ is such that all the methods in the table deliver an $O(h^{k+1})$ $H^2$-error estimate.
Consistently with the terminology adopted above for second-order elliptic PDEs,
the method can be viewed as \emph{optimal} if the order of the face unknowns approximating
the trace of the gradient (or of the normal derivative) is of degree $k$.
As seen from Table~\ref{tab:HHO_WG}, the WG methods from the literature do not
meet this criterion. For instance, the HHO method from \cite{BoDPGK:18}
with $k=1$ converges with one order higher than the WG method from \cite{YeZhZ:20}
while using the same discrete unknowns.
The lack of optimality is related to the use of a plain least-squares stabilization.
Instead, the HHO method from \cite{BoDPGK:18} and the present HHO methods
are optimal, and this is reflected by a more elaborate design of the stabilization.
Notice that for fourth-order PDEs, this also means that the Hessian (and not only the
Laplacian) has to be reconstructed locally.
In \cite{BoDPGK:18}, the stabilization design follows the spirit of \cite{DiPEr:15,DiPEL:14}
in that it uses a Hessian-based deflection reconstruction operator. In the present methods,
the design is performed in the spirit of the LS stabilization. Another difference with
\cite{BoDPGK:18} is that the present methods only introduce face unknowns approximating the
normal derivative of the solution on the mesh (inter)faces (and not the full gradient trace).
As a result, and despite the slight increase in the degree of its
face unknowns approximating the solution trace, the present HHO methods involve less
globally coupled unknowns than in \cite{BoDPGK:18}; see the discussion in
Remarks~\ref{rem:cost_HHO_2D} and~\ref{rem:cost_HHO_3D}. Moreover, we allow here for
$k\ge0$, whereas \cite{BoDPGK:18} requires $k\ge1$. We also mention that
the increase of cell unknowns compared to \cite{BoDPGK:18}
has a moderate impact on computational
costs owing to static condensation. This slight overhead is actually compensated by the
simplification in the stabilization term (see Section \ref{sec:Numerical example}
for further discussion).

Let us briefly summarize the main novelties and results of the present work:
(i) Two novel and computationally effective HHO methods leading to optimal
$O(h^{k+1})$ $H^2$-error estimates with polynomials of order $k\ge0$
to approximate the normal derivative;
(ii) An original design in 2D using, for the first time in HHO methods,
the canonical hybrid finite element in the stabilization;
(iii) The HHO methods do not feature stabilization parameters that must be
large enough (only positive), in contrast with dG and $C^0$-IPDG methods.
(iv) A numerical study showing the attractive performances of the proposed methods, which
in particular can outperform dG methods (except for low polynomial orders and
Voronoi-like meshes where the number of faces is quite large)
and even be competitive with $C^0$-IPDG and HCT methods on simplicial meshes;
(v) A variant of the HHO methods using a Nitsche-type boundary-penalty technique to weakly
enforce the boundary conditions.
We notice in particular that the development of Nitsche's boundary-penalty technique
is instrumental to deal with domains with curved boundary (in the wake of \cite{BurEr:18,BCDE:21}
for elliptic interface problems) and to derive a robust approximation method in the
case of singularly perturbed regimes. 
%These aspects will be reported in a future work.
\Rev{These results are explored in our recent work \cite{DongErn2021singular}}. 
We also emphasize that our Nitsche technique does not need the penalty parameter to be large enough. This is the first time this property is met for fourth-order PDEs, and
to this purpose, we adapt ideas from \cite{L16:Nitsche,BCDE:21} derived for second-order PDEs. \Rev{Heuristically, the reason for circumventing the constraint on having a large enough penalty parameter is that the reconstruction operator in HHO methods avoids the need to introduce an additional consistency term as in the standard Nitsche method.}

As a final remark, we mention that our main error estimates are established
for an exact solution that belongs to the broken
Sobolev space $H^{k+3}(\mesh)$ (where $\mesh$ denotes the underlying mesh)
and to the Sobolev space $H^{2+s}(\Omega)$ with $s>\frac32$. This latter
assumption follows the rather classical paradigm in the analysis of nonconforming methods
and is invoked when bounding the consistency error.
As discussed in Remark~\ref{rem:regularity}, the regularity gap can be lowered to $s>1$
by adapting the techniques developed in \cite{ErnGuer2021}
and \cite[Chap.~40\&41]{Ern_Guermond_FEs_II_2021}
in the context of second-order elliptic PDEs.
We also notice that quasi-optimal error estimates for general loads
in $H^{-2}(\Omega)$ are derived in \cite{VZ2,VZ3}
for the Morley element
and the $C^0$-IPDG method (see also
\cite{carstensen2021lowerorder} for further results in the case of various
lowest-order methods). The techniques in \cite{VZ2,VZ3}
require to modify the right-hand side of
the discrete problem by means of bubble functions and a $C^1$-smoother.
These ideas have been adapted to HHO methods for second-order elliptic PDEs with
loads in $H^{-1}(\Omega)$ in \cite{ErnZa:20}. We expect that the extension
to the biharmonic problem could follow a similar path for $d=2$, whereas for
$d=3$, one difficulty is related, irrespective of the considered discretization
method, to the lack of a well-established and computable $C^1$-smoother of
arbitrary order.

The rest of this work is organized as follows.  We introduce some basic notation, the mesh assumptions, and some analysis tools in Section \ref{sec:Model problem and HHO methods}.  In Section \ref{sec:HHO 2D}, we introduce the HHO method in the 2D setting employing the canonical hybrid finite element to design the stabilization.  In Section \ref{sec:Stability and error analysis}, we present the stability and error analysis of the method introduced in Section \ref{sec:HHO 2D}. In Section \ref{sec:HHO 3D}, we present the second HHO method, this time valid in arbitrary dimension, and we outline the main changes in the stability and error analysis from Section~\ref{sec:Stability and error analysis}. \Rev{Our numerical results indicate that in two dimensions, the first HHO method from Section \ref{sec:HHO 2D} is more effective than the second  method from Section \ref{sec:HHO 3D}}. In Section \ref{sec:HHO-N}, we combine the above HHO methods with Nitsche's boundary-penalty technique. Finally, numerical results showcasing the computational advantages of the proposed HHO methods are presented in Section \ref{sec:Numerical example}.

\section{Model problem and discrete setting}\label{sec:Model problem and HHO methods}
In this section, we introduce some basic notation, the weak formulation of the model problem, and the discrete setting to formulate and analyze the HHO discretization.

\subsection{Basic notation and weak formulation} \label{Model problem}

We use standard notation for the Lebesgue and Sobolev spaces and, in particular, for the fractional-order Sobolev spaces, we consider the Sobolev--Slobodeckij seminorm based on the double integral. For an open, bounded, Lipschitz set $S$ in $\mathbb{R}^d$, $d\in\{1,2,3\}$, we denote by $(v,w)_S$ the $L^2(S)$-inner product, and we employ the same notation when $v$ and $w$ are vector- or matrix-valued fields. We denote by $\nabla w$ the (weak) gradient of $w$ and by $\nabla^2 w$ its (weak) Hessian. Let $\n$ be the unit outward normal vector on the boundary $\partial S$ of $S$. Assuming that the functions $v$ and $w$ are smooth enough, we have the following integration by parts formula:
\begin{equation} \label{eq:ipp1}
(\Delta^2 v,w)_S = (\nabla^2v,\nabla^2w)_S+(\nabla\Delta v,\n w)_{\partial S}-(\nabla^2v\n,\nabla w)_{\partial S}.
\end{equation}
Whenever the context is unambiguous, we denote by $\partial_n$ the (scalar-valued) normal derivative on $\partial S$ and by $\partial_t$ the ($\RR^{d-1}$-valued) tangential derivative.
We also denote by $\partial_{nn} v$ the (scalar-valued) normal-normal second-order derivative and by $\partial_{nt} v$ the ($\RR^{d-1}$-valued) normal-tangential second-order derivative. The integration by parts formula~\eqref{eq:ipp1} can then be rewritten as
\begin{equation} \label{eq:ipp2}
(\Delta^2 v,w)_S = (\nabla^2v,\nabla^2w)_S+(\partial_n\Delta v,w)_{\partial S}-(\partial_{nn}v,\partial_n w)_{\partial S}-(\partial_{nt}v,\partial_t w)_{\partial S}.
\end{equation}
In what follows, the set $S$ is always a polytope so that its boundary can be decomposed into a finite union of planar faces with disjoint interiors. Expressions involving the tangential derivative on $\partial S$ are then implicitly understood to be evaluated as a summation over the faces composing $\partial S$.

Using the above integration by parts formula, the following weak formulation of the model problem \eqref{pde} is classically derived: Find  $u\in H^2_0(\Omega)$ such that
\begin{equation}\label{weak form}
(\nabla^2 u, \nabla^2 v )_\Omega = (f,v)_\Omega, \qquad  \forall v \in H^2_0(\Omega),
\end{equation}
The well-posedness of~\eqref{weak form} is proven, e.g., in~\cite[Section~1.5]{GirRa:86}.

\begin{remark}[Non-homogeneous conditions] \label{rem:non_homo}
Since the domain $\Omega$ is a polytope, its boundary can be split into $\{\partial \Omega_i\}_{i=1}^N$ $(d-1)$-dimensional planar faces with disjoint interiors. Let $g_D $ and $g_N$ be boundary data such that $g_D|_{\partial \Omega_i} \in H^{\frac32} (\partial  \Omega_i)$ and
$g_N|_{\partial \Omega_i} \in H^{\frac12} (\partial   \Omega_i)$ for all $i\in\{1,\ldots,N\}$, as well as $g_D  \in C^0(\partial \Omega)$. Then, one can enforce the non-homogeneous boundary conditions $u=g_D$ and $\partial_nu=g_N$ on all the faces $\partial\Omega_i$; see \cite[Sections~1.5 \&1.6]{Grisvard}.  
\end{remark}

\subsection{Inverse, trace, and Poincar\'e inequalities}\label{Inverse inequality}

Let $\mesh$ be a mesh covering $\Omega$ exactly.
The mesh $\mesh$ can have cells that are disjoint open polytopes in $\RR^d$ (with planar faces),
and hanging nodes are possible. A generic mesh cell is denoted by
$K\in\mesh$, its diameter by $h_K$, and its unit outward normal by $\n_K$.
We assume that the mesh belongs to a shape-regular
mesh sequence $(\mesh)_{h>0}$ in the sense of \cite{DiPEr:15}.
In a nutshell, any mesh $\mesh$ admits a matching simplicial submesh $\mesh'$
such that any cell (or face) of $\mesh'$ is a subset of exactly one cell (or at
most one face) of
$\mesh$. Moreover, there exists a mesh-regularity parameter
$\rho > 0$ such that for all $h>0$, all $K \in \mesh$,
and all $S \in \mesh'$ such that $S \subset K$,
we have $\rho h_S \leq r_S$ and $\rho h_K \leq h_S$,
where $r_S$ denotes the inradius of the simplex $S$.
The mesh faces are collected in the set $\Fall$, which is split as the set
$\Fb$ containing the mesh boundary faces and the set $\Fint$ containing the mesh
interfaces. In this work, we make the mild additional assumption that the
mesh faces are connected; the reason for this is that we will consider an
approximation operator on the mesh faces that is only $H^1$-stable, and not $L^2$-stable,
so that we will need to invoke some polynomial approximation properties directly
on the mesh faces (see \eqref{eq:app_canon_F}).
Let $\n_F$ denote the unit normal vector orienting any mesh face $F\in\Fall$.
The direction of $\n_F$ is arbitrary, but fixed once and for all, for all $F\in \Fint$,
and $\n_F:=\n$ for all $F\in\Fb$.
For any mesh cell $K\in\mesh$, the mesh faces composing its boundary
$\partial K$ are collected in the set $\FK$.
The shape-regularity of the mesh sequence implies that for all $K\in\mesh$ and all
$F\in\FK$ with diameter $h_F$, the length scales $h_K$ and $h_F$ are uniformly
equivalent and that $\#(\FK)$ is uniformly bounded.

Let us recall some important analysis tools. We refer the reader, e.g., to \cite[Sec.~1.4]{DiPietroErn} for the proofs of Lemma~\ref{lemma: Inverse inequality} and Lemma~\ref{lemma: trace inequality} and to \cite{Verf_quasi} for the derivation of the Poincar\'e inequality in $H^2$ from the corresponding inequality in $H^1$. For all $K\in\mesh$ and all $k\ge0$, $\mathbb{P}^k(K)$ denotes the linear space composed of the restriction to $K$ of polynomials of total degree at most $k$.

\begin{lemma}[Discrete inverse and trace inequalities]\label{lemma: Inverse inequality}
Let $\mesh$ belong to a shape-regular mesh sequence and let $k\ge0$.
There are constants $C_{\mathrm{inv}}^{\mathrm{tr}}$ and $C_{\mathrm{inv}}$, only depending on the mesh shape-regularity, the polynomial degree $k$, and the space dimension $d$, such that
for all $v_h\in\mathbb{P}^k(K)$ and all $K\in\mesh$,
\begin{align}\label{trace_inv}
\|{v}_h\|_{\dK} &\leq C_{\mathrm{inv}}^{\mathrm{tr}} h_K^{-\frac12}  \|{v}_h\|_{ K},\\
\label{H1_inv}
\|\nabla {v}_h\|_{K} &\leq C_{\mathrm{inv}} h_K^{-1}  \|{v}_h\|_{ K},
\end{align}
\end{lemma}

\begin{lemma}[Multiplicative trace inequality]\label{lemma: trace inequality}
Let $\mesh$ belong to a shape-regular mesh sequence.
There is a constant $C_{\mathrm{mt}}$, only depending on the mesh shape-regularity and the space dimension $d$, such that for all $v\in H^1(K)$ and all $K\in\mesh$,
\begin{equation}\label{mult_tr}
\|{v} \|_{\dK} \leq C_{\mathrm{mt}} \big(
h_K^{-\frac12}  \|{v}\|_{ K}
+\|{v}\|_{ K}^{\frac12}  \|\nabla {v}\|_{ K}^{\frac12}\big).
\end{equation}
\end{lemma}

\begin{lemma}[Poincar\'e inequality]\label{lemma: Poincare inequality}
Let $\mesh$ belong to a shape-regular mesh sequence.
There is a constant $C_{\mathrm{P}}$, only depending on the mesh shape-regularity and the space dimension $d$, such that for all $v\in H^2(K)^\perp:=\{v\in H^2(K)\;|\; (v,\xi)_K=0, \;\forall \xi \in \mathbb{P}^1(K)\}$ and all $K\in\mesh$,
\begin{equation}\label{P_ineq}
h_K^{-2} \|{v} \|_{K} + h_K^{-1} \|\nabla {v} \|_{K}
\leq C_{\mathrm{P}} \|\nabla^2 {v}\|_{ K}.
\end{equation}
\end{lemma}

\begin{remark}[Discrete inverse inequality on faces]
Similarly to~\eqref{H1_inv} and recalling that the diameter of any face $F\in\FK$ is uniformly equivalent to $h_K$, one can prove that there is a constant  $\tilde{C}_{\mathrm{inv}}$\Rev{,} only depending on the mesh shape-regularity, the polynomial degree $k$, and the space dimension $d$, such that
\begin{equation}\label{inv_tang}
\|\partial_t {v}_h\|_{F} \leq \tilde{C}_{\mathrm{inv}} h_K^{-1}  \|{v}_h\|_{ F},
\end{equation}
for all $v_h\in\mathbb{P}^k(K)$, all $K\in\mesh$, and all $F\in\FK$.
\end{remark}

\begin{remark}[Fractional multiplicative trace inequality]
Let $\mesh$ belong to a shape-regular mesh sequence. Let $s\in (\frac12,1]$.
There is a constant $C_{\mathrm{fmt}}$, only depending on the mesh shape-regularity
and the space dimension $d$, such that for all $v\in H^s(K)$ and all $K\in\mesh$,
\begin{equation}\label{f_mult_tr}
\|{v} \|_{\dK} \leq C_{\mathrm{fmt}} \big(
h_K^{-\frac12}  \|{v}\|_{K}
+ h_K^{s-\frac12}  |v|_{H^s(K)}\big).
\end{equation}
The proof when $K$ is a simplex can be found in \cite[Lem.~7.2]{ErnGuermond:17}.
In the general case,  for every subface of a
face in $\FK$, one carves a subsimplex inside $K$ whose height is uniformly
equivalent to $h_K$. Notice that for $s=1$, \eqref{f_mult_tr} is a simple consequence
of \eqref{mult_tr} and Young's inequality since $|v|_{H^1(K)}=\|\nabla v\|_{K}$.
\end{remark}

\subsection{Polynomial approximation in cells and on faces}

Let $k\ge0$ and let $\Pi^{k+2}_{K}$ be the $L^2$-orthogonal projection onto $\mathbb{P}^{k+2}(K)$. Since the mesh cells can be decomposed into a finite number of subsimplices, the approximation properties of $\Pi^{k+2}_{K}$ can be established by proceeding as in \cite[Lem.~5.4]{ErnGuermond:17}.

\begin{lemma}[Polynomial approximation in $K$]\label{lemma: Polynomial approximation}
Let $\mesh$ belong to a shape-regular mesh sequence.
Let $k\geq 0$.
There is a constant $C_{\mathrm{app}}$, only depending on the mesh shape-regularity, the polynomial degree $k$, and the space dimension $d$, such that
for all $t\in[0,k+3]$, all $m\in\{0,\ldots,\lfloor t\rfloor\}$, all $v\in H^t(K)$, and all $K\in\mesh$,
\begin{equation}\label{pol_app}
|v - \Pi^{k+2}_{K}(v)|_{H^m(K)}
\leq C_{\mathrm{app}} h^{t-m}_K |{v} |_{H^t(K)}.
\end{equation}
\end{lemma}

Another useful property of $\Pi^{k+2}_{K}$ results from the multiplicative trace inequality \eqref{mult_tr} and the Poincar\'e inequality \eqref{P_ineq}. Indeed, we infer that there is a constant $C_\Pi$, only depending on the mesh shape-regularity, the polynomial degree $k$, and the space dimension $d$, such that for all $v\in H^2(K)$ and all $K\in\mesh$,
\begin{equation} \label{eq:approx_faces}
h_K^{-\frac32}\|v-\Pi^{k+2}_{K}(v)\|_{\dK} + h_K^{-\frac12}\|\nabla(v-\Pi^{k+2}_{K}(v))\|_{\dK}
\le C_\Pi \|\nabla^2(v-\Pi^{k+2}_{K}(v))\|_K.
\end{equation}

We will use two operators for the polynomial approximation on the mesh faces. The first one is an $L^2$-orthogonal projection. Specifically, letting $\mathbb{P}^{k}(\FK) := \times_{F\in\FK} \mathbb{P}^{k}(F)$ for all $k\ge0$ and all $K\in\mesh$, we denote by $\Pi^k_{\dK}$ the
$L^2$-orthogonal projection onto $\mathbb{P}^{k}(\FK)$. Notice that $\Pi^k_{\dK}(v)$ can be computed independently for each face $F\in\FK$. The second operator is specific to the 2D setting where the mesh faces are straight segments. On the reference interval $\hat{I}:= (-1,1)$, the canonical hybrid finite element of degree $(k+1)$ has for its degrees of freedom the value at the two endpoints and, for $k\ge1$, the \Rev{integrals on $\hat{I}$ weighted by} a chosen set of basis functions in $\mathbb{P}^{k-1}(\hat{I})$ (see, e.g., \cite[Sec.~6.3.3~\&~7.6]{Ern_Guermond_FEs_I_2021} or \cite[Thm.~3.14]{schwab}). For all $F\in\Fall$, let $J_F^{k+1}:H^1(F)\rightarrow \mathbb{P}^{k+1}(F)$ be the corresponding interpolation operator generated using geometric affine mappings. Then, the two key identities satisfied by $J_F^{k+1}$ are for all $v\in H^1(F)$,
\begin{equation} \label{eq:prop_canon_F}
\big( \partial_t(v-J_F^{k+1}(v)),\xi\big)_F=0,\; \forall \xi\in \mathbb{P}^k(F),
\qquad
(v-J_F^{k+1}(v),\theta)_F=0,\; \forall \theta\in \mathbb{P}^{k-1}(F),
\end{equation}
or, in more compact form,
$\Pi_F^k(\partial_t v)=\partial_t(J_F^{k+1}(v))$ and $\Pi^{k-1}_F\circ J_F^{k+1}=\Pi^{k-1}_F$.
Moreover, $J_{F}^{k+1}$ satisfies the following approximation properties: There is $C_J$, only depending on the mesh shape-regularity, the polynomial degree $k$, and the space dimension $d$, such that
\begin{equation} \label{eq:app_canon_F}
\|v-J_F^{k+1}(v)\|_F \le C_Jh_F\|\partial_tv\|_F, \qquad
\|v-J_F^{k+1}(v)\|_F \le C_Jh_F^2\|\partial_{tt}v\|_F,
\end{equation}
for all $v\in H^1(F)$ and all $v\in H^2(F)$, respectively. Notice that
for $k=0$, $J_F^1$ coincides with the Lagrange interpolate on $F$ based
on its two endpoints.

In what follows, it is convenient to rewrite~\eqref{eq:prop_canon_F} and~\eqref{eq:app_canon_F} on the whole boundary of every mesh cell $K\in\mesh$. Letting $H^l(\FK):=\{v\in L^2(\dK)\;|\; v|_F\in H^l(F),\;\forall F\in\FK\}$ with $l\in\{1,2\}$, $J_{\dK}^{k+1}:H^1(\FK)\rightarrow \mathbb{P}^{k+1}(\FK)$ is defined facewise by setting $J_{\dK}^{k+1}(v)|_F:=J_F^{k+1}(v|_F)$ for all $v\in H^1(\FK)$. Recalling that the tangential derivative is understood to act facewise, we obtain for all $v\in H^1(\FK)$,
\begin{equation} \label{eq:prop_canon_dK}
\big( \partial_t(v-J_{\dK}^{k+1}(v)),\xi\big)_{\dK}=0,\; \forall \xi\in \mathbb{P}^k(\FK),
\qquad
(v-J_{\dK}^{k+1}(v),\theta)_{\dK}=0,\; \forall \theta\in \mathbb{P}^{k-1}(\FK).
\end{equation}
Moreover, there is $\tilde{C}_J$ having the same dependencies as $C_J$ such that
\begin{equation} \label{eq:app_canon_dK}
\|v-J_{\dK}^{k+1}(v)\|_{\dK} \le \tilde{C}_Jh_K\|\partial_tv\|_{\dK}, \qquad
\|v-J_{\dK}^{k+1}(v)\|_{\dK} \le \tilde{C}_Jh_K^2\|\partial_{tt}v\|_{\dK},
\end{equation}
for all $v\in H^1(\FK)$ and all $v\in H^2(\FK)$, respectively, where we used that
$h_F$ and $h_K$ are uniformly equivalent for all $F\in\FK$.

\section{HHO method for the 2D biharmonic problem} \label{sec:HHO 2D}

Let $k\geq 0$ be the polynomial degree. For all $K\in \mesh$, the local HHO space is
\begin{equation}\label{HHO space 2D}
\fesE: =\mathbb{P}^{k+2}(K)
\times
\mathbb{P}^{k+1}(\FK)
\times
\mathbb{P}^{k}(\FK).
\end{equation}
A generic element in  $\fesE$ is denoted $\hat{v}_K := (v_K, v_{\dK}, \gamma_{\dK})$ with $v_K \in \mathbb{P}^{k+2}(K)$, $v_{\dK} \in \mathbb{P}^{k+1}(\FK)$, and  $\gamma_{\dK} \in \mathbb{P}^{k}(\FK)$. The first component of $\hat{v}_K$ aims at representing the solution inside the
mesh cell, the second its trace on the cell boundary, and the third its normal derivative on the cell boundary (along the direction of the outward normal $\n_K$). In what follows, it is implicitly understood that within integrals over $\dK$, the symbol $\partial_n$ means $\n_K{\cdot}\nabla$.

\subsection{Reconstruction and stabilization}

The HHO method is formulated locally by means of a reconstruction and a stabilization operator. The local reconstruction operator $R_K: \fesE \rightarrow \mathbb{P}^{k+2}(K)$ is such that, for all $\hat{v}_K:= (v_K,v_{\dK}, \gamma_{\dK}) \in \fesE$,
$R_K(\hat{v}_K)\in \mathbb{P}^{k+2}(K)$ is determined by solving the following well-posed problem:
\begin{equation}\label{reconstruction}
\begin{aligned}
(\nabla^2 R_K(\hat{v}_K), \nabla^2 w)_K ={}& (\nabla^2  {v}_K, \nabla^2 w)_K
+ (v_K -v_{\dK} , \partial_n \Delta  w)_{\dK}
- (\partial_n v_K - \gamma_{\dK},  \partial_{nn}  w)_{\dK} \\
& 	- (\partial_t (v_K - v_{\dK}),  \partial_{nt}  w)_{\dK},  \quad \forall w \in  \mathbb{P}^{k+2}(K), \\
(R_K(\hat{v}_K),\xi)_K = {}& ( v_K,\xi)_K,
\quad \forall \xi \in \mathbb{P}^{1}(K).
\end{aligned}
\end{equation}
When computing $R_K(\hat{v}_K)$, one actually takes $w \in \mathbb{P}^{k+2}(K)^\perp:=\{w\in \mathbb{P}^{k+2}(K)\;|\; (w,\xi)_K=0,\; \forall \xi\in\mathbb{P}^{1}(K)\}$ since the equation is trivial whenever $w\in \mathbb{P}^{1}(K)$.
Moreover, owing to the integration by parts formula~\eqref{eq:ipp2}, we infer that
\begin{equation}\label{eq:rec_ipp}
(\nabla^2 R_K(\hat{v}_K), \nabla^2 w)_K = (v_K, \Delta^2 w)_K
- (v_{\dK} , \partial_n \Delta  w)_{\dK}
+ (\gamma_{\dK},  \partial_{nn}  w)_{\dK}
+ (\partial_t v_{\dK},  \partial_{nt}  w)_{\dK}.
\end{equation}
This expression shows that in the rightmost term on the right-hand side, we take advantage of the face component $v_{\dK}$ to represent the tangential derivative at the boundary of $K$. Notice also that for $k=0$, the second term on the right-hand side vanishes.

The local stabilization bilinear form $S_{\dK}$ is defined such that, for all $(\hat{v}_K, \hat{w}_K)\in \fesE \times \fesE$ with $\hat{v}_K:= (v_K,v_{\dK}, \gamma_{\dK})$ and
$\hat{w}_K:= (w_K,w_{\dK},\chi_{\dK})$,
\begin{equation}\label{def: stabilization}
\begin{aligned}
S_{\dK}(\hat{v}_K,\hat{w}_K)
:={}& h_K^{-3}
\big( J^{k+1}_{\dK}(v_{\dK}- v_K),J^{k+1}_{\dK}(w_{\dK}- {w}_K)\big)_{\dK} \\
& + h_K^{-1}
\big( \Pi^{k}_{\dK}(\gamma_{\dK}- \partial_n {v}_K),
\Pi^{k}_{\dK}(\chi_{\dK}- \partial_n  {w}_K)
\big)_{\dK}.
\end{aligned}
\end{equation}
Notice the use of the interpolation operator $J^{k+1}_{\dK}$ for the first term on the right-hand side.
The reconstruction and stabilization operators are combined together to build 
\Rev{the} local bilinear form $a_K$ on $\fesE \times \fesE$ such that
\begin{equation} \label{eq:def_aK}
a_K(\hat{v}_K,\hat{w}_K):= (\nabla^2 R_K(\hat{v}_K),\nabla^2 R_K(\hat{w}_K))_K
+S_{\dK}(\hat{v}_K,\hat{w}_K).
\end{equation}

\subsection{The global discrete problem}

We define the global HHO space as
\begin{equation}
\fes : = \mathbb{P}^{k+2}(\mesh) \times \mathbb{P}^{k+1}(\Fall)\times \mathbb{P}^k(\Fall).
\end{equation}
A generic element in $\fes$ is denoted $\hat{v}_h:=(v_{\mesh},v_{\Fall},\gamma_{\Fall})$
with $v_{\mesh}:=(v_K)_{K\in\mesh}$, $v_{\Fall}:=(v_F)_{F\in\Fall}$, and $\gamma_{\Fall}
:=(\gamma_F)_{F\in\Fall}$, where $\gamma_F$ is meant to approximate the normal derivative in the direction of the unit normal vector $\n_F$ orienting $F$. For all $K\in \mesh$, the local components of $\hat{v}_h$ are collected in the triple
$\hat{v}_K:=(v_K,v_{\dK},\gamma_{\dK})\in\fesE$ with $v_{\dK}|_F: = v_F$ and  $\gamma_{\dK}|_F: = (\n_F {\cdot}\n_K)\gamma_{F}$ for all $F \in \FK$. Notice that the way the face components of $v_{\dK}$ are assigned follows the usual way of HHO methods for second-order elliptic PDEs, whereas the definition of the face components of $\gamma_{\dK}$ takes into account the orientation of the faces in $\FK$.
Furthermore, we enforce the homogeneous boundary conditions strongly by considering the subspace
\begin{equation}
\fesz : = \{\hat{v}_h\in \fes\;|\; v_F=\gamma_F=0,\;\forall F\in\Fb\}.
\end{equation}
The discrete HHO problem for the 2D biharmonic problem is as follows: Find $\hat{u}_h\in \fesz$ such that
\begin{equation}\label{discrete problem}
a_h(\hat{u}_h,\hat{w}_h) = \ell(w_{\mesh}), \qquad \forall \hat{w}_h\in \fesz,
\end{equation}
where the global discrete bilinear form $a_h$ and the global linear form $\ell$ are assembled cellwise as
\begin{equation}\label{bilinear form}
a_h(\hat{v}_h, \hat{w}_h):= \su a_K(\hat{v}_K, \hat{w}_K), \qquad
\ell(w_{\mesh}):= (f,w_{\mesh})_\Omega = \su (f,{w}_K)_K.
\end{equation}
Notice that only the first component of the triple $\hat{w}_h$ is used to evaluate the right-hand side in \eqref{discrete problem}.
An important observation is that the discrete problem~\eqref{discrete problem} is amenable to static condensation. Indeed,
the cell unknowns can be eliminated locally in every mesh cell, leading to a global problem where the only remaining unknowns
are those attached to the mesh faces, i.e.,
those in $\mathbb{P}^{k+1}(\Fall)\times \mathbb{P}^k(\Fall)$.

\begin{remark}[Comparison with \cite{BoDPGK:18}] \label{rem:cost_HHO_2D}
The present HHO method is cheaper than
the one from \cite{BoDPGK:18} where the globally coupled unknowns are in
$\mathbb{P}^{k}(\Fall) \times \mathbb{P}^{k}(\Fall;\RR^{2})$ with $k\geq 1$.
Indeed, in the present method, there are $(2k+3)$, $k\ge0$, unknowns per mesh interface,
whereas this number is $(3k+3)$, $k\ge1$, in \cite{BoDPGK:18}. On the other hand, the present method is slightly more expensive regarding static condensation since the number of cell unknowns is $\frac12(k+3)(k+4)$ vs.~$\frac12(k+2)(k+3)$ in \cite{BoDPGK:18}.
However, the slight overhead incurred in the static condensation is compensated
by the simpler form of the stabilization; see Section~\ref{sec:Numerical example}
for more insight into the computational costs.
\end{remark}

\begin{remark}[Variant]\label{cheap HHO}
It is also possible to consider the slightly cheaper choice $\fesE: =\mathbb{P}^{k+1}(K) \times \mathbb{P}^{k+1}({\dK})
\times \mathbb{P}^{k}({\dK})$ with $k\geq0$. With this choice, the number of cell unknowns to be statically condensed is slightly reduced,
but the size of the global problem coupling all the face unknowns is unchanged. Notice that this choice requires to modify the stabilization
bilinear form by setting
\begin{align}
S_{\dK}(\hat{v}_K,\hat{w}_K)
:={} & h_K^{-4}\big(\Pi^{k+1}_{K} (v_K-  R_K(\hat{v}_K)), \Pi^{k+1}_{K} (w_K- R_K(\hat{w}_K))\big)_{K} \nonumber \\
& + h_K^{-3}
\big( J^{k+1}_{\dK}(v_{\dK}- R_K(\hat{v}_K)),J^{k+1}_{\dK}(w_{\dK}- R_K(\hat{w}_K))\big)_{\dK} \nonumber \\
& + h_K^{-1} \big( \Pi^{k}_{\dK}(\gamma_{\dK}- \partial_n R_K(\hat{v}_K)),
\Pi^{k}_{\dK}(\chi_{\dK}- \partial_n R_K(\hat{w}_K))\big)_{\dK}. \label{eq:cheap_HHO}
\end{align}
The analysis of this variant will not be detailed herein, but this variant will be included in the numerical investigations presented in Section \ref{sec:Numerical example}.
\end{remark}

\begin{remark}[\Rev{Boundary conditions}] \label{rem:non-homo_2D}
In the non-homogeneous case, the HHO solution is sought in the space $\fes$, whereas the test functions are still taken in the space $\fesz$. The value of the components of the HHO solution attached to the mesh boundary faces is then assigned by means of the projections $J^{k+1}_{\dK}({g_D}|_{\dK})$ and $\Pi^{k}_{\dK}({g_N}|_{\dK})$. The convergence analysis proceeds as for homogeneous boundary conditions, up to straightforward adaptations when bounding the consistency error. Details are omitted for brevity. \Rev{It is also possible to consider the boundary conditions $u=\partial_{nn}u=0$ on $\partial \Omega$. The discrete HHO bilinear form is still defined as above, but the discrete problem \eqref{discrete problem} now involves the subspace $\fesz : = \{\hat{v}_h\in \fes\;|\; v_F=0,\;\forall F\in\Fb\}$, i.e., the boundary condition $u=0$ is still strongly enforced, whereas the boundary condition $\partial_{nn}u=0$ is weakly enforced.}
\end{remark}

\section{Stability and error analysis}\label{sec:Stability and error analysis}

In this section, we perform the stability and error analysis of the HHO method devised in the previous section for the 2D biharmonic problem. We first establish a local stability property for the bilinear form $a_K$ together with the well-posedness of the discrete problem~\eqref{discrete problem}. Then, we introduce a suitable reduction operator leading to optimal approximation properties, we bound the corresponding consistency error, and finally we derive the error estimate.

In what follows, the symbol $C$ denotes a generic positive constant whose value can change at each occurrence, provided this value only depends on the mesh shape-regularity, the polynomial degree $k$, and the space dimension $d$.

\subsection{Stability and well-posedness}

We equip the local HHO space $\fesE$  with the $H^2$-like seminorm such that for all $\hat{v}_K:=(v_K, v_{\dK}, \gamma_{\dK}) \in \fesE$,
\begin{equation}\label{H2_seminorm_elem}
|\hat{v}_K|^2_{\fesE}: = \|\nabla^2 v_K\|_K^2
+ h_K^{-3} \| v_{\dK} - v_K\|_{\dK}^2
+ h_K^{-1}\|  \gamma_{\dK} - \partial_n v_K\|_{\dK}^2.
\end{equation}

\begin{lemma}[Local stability and boundedness] \label{lem:stab_bnd}
There is a real number $\alpha>0$, depending only on the mesh shape-regularity, the polynomial degree $k$, and the space dimension $d$, such that for all $\hat{v}_K\in \fesE$ and all $K\in\mesh$,
\begin{equation}\label{local equivalent}
\alpha|\hat{v}_K|^2_{\fesE}
\leq  \|\nabla^2 R_K(\hat{v}_K) \|_{ K}^2
+ S_{\dK}(\hat{v}_K,\hat{v}_K)
\leq \alpha^{-1} |\hat{v}_K|^2_{\fesE}.
\end{equation}
\end{lemma}

\begin{proof}
(1) Lower bound. Using the reconstruction defined in \eqref{reconstruction} with $w := v_K\in \mathbb{P}^{k+2}(K)$, we have
\begin{equation}
\begin{aligned}
(\nabla^2 R_K(\hat{v}_K),\nabla^2 v_K)_K= {} & \|\nabla^2  {v}_K\|^2_K
+ (v_K -v_{\dK} , \partial_n \Delta  v_K)_{\dK} \\
& - (\partial_n v_K - \gamma_{\dK},  \partial_{nn}  v_K)_{\dK}
- (\partial_t (v_K - v_{\dK}),  \partial_{nt}  v_K)_{\dK} \\
= {} & \|\nabla^2  {v}_K\|^2_K
+ (J_{\dK}^{k+1}(v_K -v_{\dK}), \partial_n \Delta  v_K)_{\dK} \\
& - (\Pi_{\dK}^k(\partial_n v_K - \gamma_{\dK}),  \partial_{nn}  v_K)_{\dK}
- (\partial_t J_{\dK}^{k+1}(v_K - v_{\dK}),  \partial_{nt}  v_K)_{\dK},
\end{aligned}
\end{equation}
where we used the two identities from \eqref{eq:prop_canon_dK} together with $\partial_n \Delta  v_K\in \mathbb{P}^{k-1}(\FK)$ (if $k\ge1$, otherwise this term vanishes) and $\partial_{nn} v_K,\partial_{nt} v_K \in \mathbb{P}^{k}(\FK)$.
Using the Cauchy--Schwarz inequality together with the inverse inequalities \eqref{H1_inv},  \eqref{trace_inv}, and \eqref{inv_tang}, we infer that
\begin{equation*}
\begin{aligned}
\|\nabla^2  {v}_K\|^2_K\leq {}&
\|\nabla^2  R_K(\hat{v}_K)\|_K  \|\nabla^2  {v}_K\|_K
+ \| J_{\dK}^{k+1}(v_K -v_{\dK})\|_{\dK} \| \partial_n \Delta  v_K\|_{\dK} \\
& + \| \Pi_{\dK}^k(\partial_n v_K - \gamma_{\dK}) \|_{\dK} \| \partial_{nn}  v_K \|_{\dK}
+ \| \partial_t J_{\dK}^{k+1}(v_K - v_{\dK}) \|_{\dK} \| \partial_{nt}  v_K \|_{\dK} \\
\leq {}& \|\nabla^2  R_K(\hat{v}_K)\|_K \|\nabla^2  {v}_K\|_K
+ C S_{\dK}(\hat{v}_K,\hat{v}_K)^{\frac12} \|\nabla^2  {v}_K\|_K,
\end{aligned}
\end{equation*}
which shows that
\begin{equation} \label{eq:bnd_up1}
\|\nabla^2  {v}_K\|_K \le \|\nabla^2  R_K(\hat{v}_K)\|_K + C S_{\dK}(\hat{v}_K,\hat{v}_K)^{\frac12}.
\end{equation}
Moreover, since $v_{\dK} - v_K=J_{\dK}^{k+1}(v_{\dK} - v_K)-(v_K-J_{\dK}^{k+1}(v_K))$, the triangle inequality implies that
\begin{align*}
h_K^{-\frac32} \| v_{\dK} - v_K\|_{\dK} \le {}& h_K^{-\frac32} \|J_{\dK}^{k+1}(v_{\dK} - v_K)\|_{\dK}
+ h_K^{-\frac32}\|v_K-J_{\dK}^{k+1}(v_K)\|_{\dK} \\
\le {}& S_{\dK}(\hat{v}_K,\hat{v}_K)^{\frac12} + Ch_K^{\frac12}\|\partial_{tt}v_K\|_{\dK} \\
\le {}& S_{\dK}(\hat{v}_K,\hat{v}_K)^{\frac12} + C\|\nabla^2 v_K\|_K,
\end{align*}
where we used the approximation property \eqref{eq:app_canon_dK} and the discrete trace inequality~\eqref{trace_inv}. Combining this bound with \eqref{eq:bnd_up1} proves that
\begin{equation} \label{eq:bnd_up2}
h_K^{-\frac32} \| v_{\dK} - v_K\|_{\dK} \le C\big( \|\nabla^2  R_K(\hat{v}_K)\|_K + S_{\dK}(\hat{v}_K,\hat{v}_K)^{\frac12}\big).
\end{equation}
Proceeding similarly shows that
\begin{align*}
h_K^{-\frac12} \| \gamma_{\dK} - \partial_n v_K\|_{\dK} \le {}& h_K^{-\frac12} \|\Pi_{\dK}^{k}(\gamma_{\dK} - \partial_n v_K)\|_{\dK}
+ h_K^{-\frac12}\|\partial_nv_K-\Pi_{\dK}^{k}(\partial_n v_K)\|_{\dK} \\
\le {}& S_{\dK}(\hat{v}_K,\hat{v}_K)^{\frac12} + C h_K^{\frac12} \|\partial_{nt}v_K\|_{\dK} \\
\le {}& S_{\dK}(\hat{v}_K,\hat{v}_K)^{\frac12} + C\|\nabla^2 v_K\|_K,
\end{align*}
where we used the approximation properties of $\Pi_{\dK}^k$ and the discrete trace inequality~\eqref{trace_inv}. Combining this bound with \eqref{eq:bnd_up1} proves that
\begin{equation} \label{eq:bnd_up3}
h_K^{-\frac12} \| \gamma_{\dK} - \partial_n v_K\|_{\dK} \le C\big( \|\nabla^2  R_K(\hat{v}_K)\|_K + S_{\dK}(\hat{v}_K,\hat{v}_K)^{\frac12}\big).
\end{equation}
Finally, combining the bounds  \eqref{eq:bnd_up1},  \eqref{eq:bnd_up2}, and  \eqref{eq:bnd_up3} proves the lower bound in~\eqref{local equivalent}.
\\
(2) Upper bound. Using this time $w := R_K(\hat{v}_K)\in\mathbb{P}^{k+2}(K)$ in the reconstruction defined in \eqref{reconstruction} and proceeding as above shows that
\begin{equation} \label{eq:bnd_up4}
\|\nabla^2  R_K(\hat{v}_K)\|_K \le \|\nabla^2  {v}_K\|_K
+ C S_{\dK}(\hat{v}_K,\hat{v}_K)^{\frac12}.
\end{equation}
Moreover, still proceeding as above, we infer that
\begin{align*}
h_K^{-\frac32} \|J_{\dK}^{k+1}(v_{\dK} - v_K)\|_{\dK} \le {}&
h_K^{-\frac32} \| v_{\dK} - v_K\|_{\dK} + h_K^{-\frac32}\|v_K-J_{\dK}^{k+1}(v_K)\|_{\dK} \\
\le {}& h_K^{-\frac32} \| v_{\dK} - v_K\|_{\dK} + C\|\nabla^2 v_K\|_K,
\end{align*}
and
\begin{equation*}
h_K^{-\frac12} \|\Pi_{\dK}^{k}(\gamma_{\dK} - \partial_n v_K)\|_{\dK} \le
h_K^{-\frac12} \| \gamma_{\dK} - \partial_n v_K\|_{\dK} + C\|\nabla^2 v_K\|_K.
\end{equation*}
Putting the above two bounds together shows that
\begin{equation} \label{eq:bnd_up5}
S_{\dK}(\hat{v}_K,\hat{v}_K)^{\frac12} \le C |\hat{v}_K|_{\fesE}.
\end{equation}
Finally, the combination of~\eqref{eq:bnd_up4} and~\eqref{eq:bnd_up5} proves the upper bound in~\eqref{local equivalent}.
\end{proof}

We equip the space $\fesz$ with the norm
\begin{equation}\label{H1_seminorm}
\|\hat{v}_h\|_{\fesz}:= \su |\hat{v}_K|^2_{\fesE},\qquad \forall \hat{v}_h \in \fesz.
\end{equation}
To show that this indeed defines a norm, consider $\hat{v}_h \in \fesz$ such that
$\|\hat{v}_h\|_{\fesz}=0$. Then, for all $K\in\mesh$, $v_K\in\mathbb{P}^1(K)$ and
$v_{\dK}=v_K$ and $\gamma_{\dK}=\partial_n v_K$ on $\dK$. For any cell $K\in\mesh$
having at least one boundary face, say $F\in\FK\cap \Fb$, we have $v_F=\gamma_F=0$
by definition of $\fesz$. Since $v_K$ is affine and its gradient vanishes
identically on $F$, $\nabla v_K$ vanishes in $K$, and since $v_K$ vanishes on $F$,
we infer that $v_K$ vanishes identically in $K$. This
implies that $v_F=\gamma_F=0$ for all
$F\in\FK$. We can then propagate the reasoning one layer of cells further inside the domain, and by repeating the process, we reach all the cells composing the mesh. Thus, the three components of the triple $\hat{v}_h$ vanish identically everywhere.

\begin{corollary}[Coercivity and well-posedness]\label{lemma: coercivit}
The discrete bilinear form $a_h$ is coercive on $\fesz$, and the discrete problem \eqref{discrete problem} is well-posed.
\end{corollary}

\begin{proof}
Summing the lower bound in~\eqref{local equivalent} over all the mesh cells shows the following coercivity property:
\begin{equation}\label{coercivity}
a_h(\hat{v}_h,\hat{v}_h)  \geq \alpha\|\hat{v}_h\|_{\fesz}^2, \qquad \forall \hat{v}_h\in\fesz.
\end{equation}
The well-posedness of \eqref{discrete problem} then follows from the Lax--Milgram lemma.
\end{proof}

\subsection{Local reduction operator and polynomial approximation}

For all $K\in \mesh$, we define the local reduction operator
$\mathcal{\hat{I}}^k_K : H^2(K) \rightarrow \fesE$ such that for all $v\in H^2(K)$,
\begin{equation}\label{interpoltation}
\mathcal{\hat{I}}^k_K(v):= \big(\Pi_{K}^{k+2}(v),J_{\dK}^{k+1} (v), \Pi_{\dK}^k (\n_K{\cdot}\nabla v)\big) \in \fesE.
\end{equation}
Moreover, the $H^2$-elliptic projection $\mathcal{E}_K(v):H^2(K)\rightarrow \mathbb{P}^{k+2}(K)$ is defined such that
\begin{equation}\label{H2_ell_proj}
\begin{alignedat}{2}
(\nabla^2 (\mathcal{E}_K(v) - v), \nabla^2 w)_K &= 0,&\qquad&\forall w \in  \mathbb{P}^{k+2}(K), \\
(\mathcal{E}_K(v)- v, \xi)_K &=0,&\qquad&\forall \xi\in \mathbb{P}^{1}(K).
\end{alignedat}
\end{equation}
The following lemma states the two key properties of the local reduction operator defined in~\eqref{interpoltation}.

\begin{lemma}[Local reduction operator] \label{lem:local_red}
We have $R_K\circ \mathcal{\hat{I}}^k_K = \mathcal{E}_K$ for all $K\in\mesh$. Moreover,
for all $K\in\mesh$ and all $v\in H^2(K)$, we have
\begin{equation}\label{stabilization bound}
S_{\dK}(\mathcal{\hat{I}}^k_K(v),\mathcal{\hat{I}}^k_K(v))^{\frac12}
\leq
C \| \nabla^2(v - \Pi_K^{k+2}(v))\|_{K}.
\end{equation}
\end{lemma}

\begin{proof}
Let $K\in\mesh$ and let $v\in H^2(K)$. \\
(1) Using the definition \eqref{eq:rec_ipp} of the reconstruction operator, we infer that
for all $w \in  \mathbb{P}^{k+2}(K)$, we have
\begin{equation*}
\begin{aligned}
(\nabla^2 R_K(\mathcal{\hat{I}}^k_K (v)), \nabla^2 w)_K  = {}
& (\Pi_{K}^{k+2}(v) , \Delta^2 w)_K \\
& - ( J_{\dK}^{k+1} (v) , \partial_n \Delta  w)_{\dK}
+ ( \Pi_{\dK}^k (\partial_n v),  \partial_{nn}  w)_{\dK}
+  (\partial_t( J_{\dK}^{k+1} (v)),  \partial_{nt}  w)_{\dK}.
\end{aligned}
\end{equation*}
Since $\Delta^2 w \in  \mathbb{P}^{k-2}(K)$ for $k\ge2$ (and vanishes otherwise), $\partial_n \Delta w \in  \mathbb{P}^{k-1}(K)$ for $k\ge1$ (and vanishes otherwise), $\partial_{nn} w \in  \mathbb{P}^{k}(K)$, and $\partial_{nt} w \in  \mathbb{P}^{k}(K)$, the $L^2$-orthogonality properties of $\Pi_{K}^{k+2}$ and $\Pi_{\dK}^k$, together with the identities \eqref{eq:prop_canon_dK} satisfied by $J_{\dK}^{k+1}$ imply that
\begin{equation*}
\begin{aligned}
&(\nabla^2 R_K(\mathcal{\hat{I}}^k_K (v)), \nabla^2 w)_K  =
( v , \Delta^2 w)_K
- (  v , \partial_n \Delta  w)_{\dK}
+ ( \partial_{n} v,  \partial_{nn}  w)_{\dK}
+  (\partial_tv,  \partial_{nt}  w)_{\dK} = (\nabla^2v,\nabla^2w)_K.
\end{aligned}
\end{equation*}
Moreover, for all $\xi \in \mathbb{P}^{1}(K)$, we have $(R_K(\mathcal{\hat{I}}^k_K (v)),  \xi)_K = (\Pi_K^{k+2} (v),  \xi)_K = (v,  \xi)_K $ for all $\xi\in\mathbb{P}_1(K)$. The above two identities prove that $R_K(\mathcal{\hat{I}}^k_K (v))=\mathcal{E}_K(v)$ for all $v\in H^2(K)$. Thus, $R_K\circ \mathcal{\hat{I}}^k_K = \mathcal{E}_K$.
\\
(2) Let us now prove \eqref{stabilization bound}.
Recalling the definitions \eqref{def: stabilization} and \eqref{interpoltation}, we have
\begin{equation} \label{relation for stabilization}
S_{\dK}(\mathcal{\hat{I}}^k_K(v),\mathcal{\hat{I}}^k_K(v))^{\frac12}
\le
h_K^{-\frac32} \| J^{k+1}_{\dK}(v- \Pi^{k+2}_K (v) ) \|_{\dK}
+ h_K^{-\frac12}
\| \Pi^{k}_{\dK}(\partial_{n}v- \partial_{n}\Pi^{k+2}_K (v)) \|_{\dK},
\end{equation}
where we used that $J^{k+1}_{\dK}\circ J^{k+1}_{\dK}=J^{k+1}_{\dK}$ and $\Pi^{k}_{\dK}
\circ \Pi^{k}_{\dK}=\Pi^{k}_{\dK}$.
We start with the first term in \eqref{relation for stabilization} where we set $\phi:=
v- \Pi^{k+2}_K (v)$. Notice that $\phi\in H^2(K)^\perp$ and that $\phi|_{\dK}\in H^1(\FK)$.
Invoking the triangle inequality, the approximation property \eqref{eq:app_canon_dK},
and the trace inequality \eqref{eq:approx_faces} shows that
\begin{align*}
h_K^{-\frac32} \| J^{k+1}_{\dK}(\phi)\|_{\dK} &\le h_K^{-\frac32} \| \phi\|_{\dK}
+ h_K^{-\frac32} \| \phi - J^{k+1}_{\dK}(\phi)\|_{\dK} \\
&\le h_K^{-\frac32} \| \phi\|_{\dK}
+ Ch_K^{-\frac12} \|\partial_t\phi\|_{\dK} \le C \|\nabla^2\phi\|_K
= C\|\nabla^2(v- \Pi^{k+2}_K (v))\|_K.
\end{align*}
Moreover, for the second term in \eqref{relation for stabilization}, we invoke the $L^2(\dK)$-stability of $\Pi_{\dK}^k$ and the trace inequality \eqref{eq:approx_faces} to show that
\begin{align*}
h_K^{-\frac12}
\| \Pi^{k}_{\dK}(\partial_{n}v- \partial_{n} (\Pi^{k+2}_K (v)) \|_{\dK}
&\le h_K^{-\frac12}
\| \partial_{n}(v- \Pi^{k+2}_K (v)) \|_{\dK} \le C \|\nabla^2(v- \Pi^{k+2}_K (v))\|_K.
\end{align*}
Combining the above two bounds with \eqref{relation for stabilization} proves \eqref{stabilization bound}.
\end{proof}

To bound the consistency error in the next section, we will consider a norm that is stronger than the $H^2$-norm. For all $K\in \mesh$ and all $v\in H^{2+s}(K)$, $s>\frac{3}{2}$, we consider the following norm:
\begin{equation}\label{Def: specical norm}
\begin{aligned}
\|v\|^2_{\sharp, K}:= \|\nabla^2 v\|_K^2
+
h^{3}_K \|\partial_{n} \Delta v \|_{\dK}^2
+
h_K \|\partial_{nn}  v \|_{\dK}^2
+
h_K \|\partial_{nt}  v \|_{\dK}^2.
\end{aligned}
\end{equation}

\begin{lemma}[Approximation]\label{lem: approxmation}
The following holds true for all $K\in \mesh$ and all $v\in H^{2+s}(K)$, $s>\frac{3}{2}$:
\begin{equation}\label{elliptic projection}
\|v - \mathcal{E}_K (v)\|_{\sharp, K} \leq C \|v - \Pi_K^{k+2}(v)\|_{\sharp, K}.
\end{equation}
\end{lemma}

\begin{proof}
Using the triangle inequality, we have
\begin{equation*}
\|v - \mathcal{E}_K (v)\|_{\sharp, K} \leq \|v - \Pi_K^{k+2}(v)\|_{\sharp, K} + \|\mathcal{E}_K(v) - \Pi_K^{k+2}(v)\|_{\sharp, K},
\end{equation*}
so that we only need to bound the second term on the right-hand side. Owing to the discrete inverse and trace inequalities \eqref{trace_inv} and \eqref{H1_inv}, we readily infer that
\begin{equation*}
\|\mathcal{E}_K(v) - \Pi_K^{k+2}(v)\|_{\sharp, K} \le C\|\nabla^2(\mathcal{E}_K(v) - \Pi_K^{k+2}(v))\|_K,
\end{equation*}
so that it remains to bound $\|\nabla^2(\mathcal{E}_K(v) - \Pi_K^{k+2}(v))\|_K$.
Recalling that $\mathcal{E}_K = R_K\circ \mathcal{\hat{I}}^k_K$, using the definition \eqref{reconstruction} of the reconstruction operator, and reasoning as in the proof of Lemma~\ref{lem:local_red} to remove the various projection operators, we infer that
\begin{equation*}
\begin{aligned}
(\nabla^2 \mathcal{E}_K(v), \nabla^2 w)_K
= {}& (\nabla^2 \Pi_{K}^{k+2}(v) , \nabla^2 w)_K
+ (\Pi_{K}^{k+2}(v)
- v , \partial_n \Delta  w)_{\dK} \\
&- (\partial_n (\Pi_{K}^{k+2}(v) - v),  \partial_{nn} w)_{\dK}
- (\partial_t (\Pi_{K}^{k+2}(v) - v),  \partial_{nt} w)_{\dK},
\end{aligned}
\end{equation*}
for all $w \in \mathbb{P}^{k+2}(K)$. Taking $w:=\mathcal{E}_K(v)-\Pi_{K}^{k+2}(v)$,
and invoking the Cauchy--Schwarz inequality together with the discrete trace and inverse inequalities \eqref{trace_inv} and \eqref{H1_inv}, we infer that
\begin{align*}
\|\nabla^2(\mathcal{E}_K(v) - \Pi_K^{k+2}(v))\|_K
&\le C\big( h^{-\frac32}_K \|v-\Pi_{K}^{k+2}(v) \|_{\dK}
+ h^{-\frac12}_K	\|\nabla(v-\Pi_{K}^{k+2}(v)) \|_{\dK}\big) \\
&\le C\|\nabla^2(v-\Pi_K^{k+2}(v))\|_K,
\end{align*}
where the last bound follows from the trace inequality~\eqref{eq:approx_faces}. This completes the proof.
\end{proof}

\subsection{Bound on consistency error}

The global reduction operator $\Ihk:H^2(\Omega)\rightarrow \fes$ is defined
such that for all $v\in H^2(\Omega)$,
\begin{equation} \label{def:Ihk}
\Ihk(v):=\big( (\Pi_K^{k+2}(v))_{K\in\mesh},(J^{k+1}_F(v))_{F\in\Fall},(\Pi_F^{k}(\n_F{\cdot}\nabla v))_{F\in\Fall} \big)\in \fes,
\end{equation}
recalling that $v$ and $\nabla v$ are single-valued on every $F\in\Fint$ for
all $v\in H^2(\Omega)$. Importantly, we notice that for all $K\in\mesh$, the local components of $\Ihk(v)$ attached to $K$ and the faces composing its boundary are
$\mathcal{\hat{I}}^k_K(v|_K)$. Moreover, for the exact solution $u$ of~\eqref{weak form}, we have $\Ihk(u)\in \fesz$.
We define the consistency error $\delta_h \in (\fesz)^\prime$ such that
\begin{equation} \label{eq:def_delta}
\langle \delta_h,\hat{w}_h \rangle:= \ell(w_{\mesh}) - a_h(\Ihk(u), \hat{w}_h),
\qquad \forall \hat{w}_h\in \fesz,
\end{equation}
where $\langle\cdot,\cdot\rangle$ denotes the duality pairing between $(\fesz)'$ and $\fesz$.

\begin{lemma}[Consistency]\label{lemma:consistency}
Assume that $u\in H^{2+s}(\Omega)$ with $s>\frac32$. The following holds true:
\begin{equation} \label{consistency}
  \|\delta_h\|_{(\fesz)'}:= \sup_{\hat{w}_h\in\fesz}
\frac{|\langle \delta_h,\hat{w}_h \rangle|}{\|\hat{w}_h\|_{\fesz}}
\leq C \left( \su \|u-\Pi_K^{k+2} (u)\|^2_{\sharp,K}\right)^{\frac12}.
\end{equation}
\end{lemma}

\begin{proof}
Let $\hat{w}_h\in \fesz$. Using the definition of $\ell$ in \eqref{bilinear form}, the PDE and the boundary conditions satisfied by the exact solution $u$, and integrating by parts cellwise, we infer that
\begin{equation*}
\ell(w_{\mesh}) = \su \Big\{ (\nabla^2 u, \nabla^2 {w}_K)_K
+ (\partial_n \Delta u , w_K)_{\dK}
- (\partial_{nn}  u , \partial_{n} w_K)_{\dK}
- (\partial_{nt}  u , \partial_{t} w_K)_{\dK} \Big\}.
\end{equation*}
The assumption $u \in H^{2+s} (\Omega)$ with $s>\frac32$ implies that $(\partial_n \Delta u)|_{\dK}$, $(\partial_{nn}u)|_{\dK}$, and $(\partial_{nt} u)|_{\dK}$ are meaningful in $L^2(\dK)$ and single-valued at every mesh interface. Moreover, since $w_{\dK}$, $\partial_t w_{\dK}$, and $ \chi_{\dK}$ are single-valued at every mesh interface $\Fint$ and vanish at each mesh boundary face since $\hat{w}_h\in \fesz$, we have
\begin{align*}
\ell(w_{\mesh})
={} &
\su \Big\{ (\nabla^2 u, \nabla^2 {w}_K)_K
+ (\partial_n \Delta u , w_K - w_{\dK})_{\dK} \\
& - (\partial_{nn}  u , \partial_{n} w_K  - \chi_{\dK})_{\dK}
- (\partial_{nt}  u , \partial_{t} (w_K - w_{\dK}))_{\dK} \Big\}.
\end{align*}
Since $a_h$ is assembled cellwise (see \eqref{bilinear form}) and the local
components of $\Ihk(u)$ are $\mathcal{\hat{I}}^k_K(u|_K)$ for all $K\in\mesh$,
we infer that
$a_h(\Ihk(u), \hat{w}_h) = \su a_K(\mathcal{\hat{I}}^k_K(u|_K),\hat{w}_K)$.
Using the definition \eqref{eq:def_aK} of $a_K$, the definition \eqref{reconstruction} of
$R_K(\hat{w}_K)$, and the identity $R_K\circ \mathcal{\hat{I}}^k_K = \mathcal{E}_K$ from Lemma~\ref{lem:local_red} leads to
\begin{align*}
a_h(\Ihk(u), \hat{w}_h) ={} &
\su \Big\{ (\nabla^2 \mathcal{E}_K(u), \nabla^2 {w}_K)_K
+ (\partial_n \Delta \mathcal{E}_K(u), w_K - w_{\dK})_{\dK} \\
& - (\partial_{nn} \mathcal{E}_K(u), \partial_{n} w_K  - \chi_{\dK})_{\dK}
- (\partial_{nt} \mathcal{E}_K(u), \partial_{t} (w_K - w_{\dK}))_{\dK} + S_{\dK}(\mathcal{\hat{I}}^k_K(u),\hat{w}_K) \Big\}.
\end{align*}
Defining the function $\eta$ cellwise as $\eta|_K:=u|_K-\mathcal{E}_K(u|_{K})$ for all $K\in\mesh$,
we infer that
\begin{equation}\label{the error relation}
\begin{aligned}
\langle \delta_h,\hat{w}_h \rangle
= {}& \su \Big\{ (\nabla^2 \eta, \nabla^2 {w}_K)_K
+ (\partial_n \Delta \eta, w_K - w_{\dK})_{\dK} \\
& - (\partial_{nn} \eta, \partial_{n} w_K  - \chi_{\dK})_{\dK}
- (\partial_{nt} \eta, \partial_{t} (w_K - w_{\dK}))_{\dK} - S_{\dK}(\mathcal{\hat{I}}^k_K(u),\hat{w}_K)\Big\}.
\end{aligned}
\end{equation}
(Notice that $(\nabla^2 \eta, \nabla^2 {w}_K)_K=0$, but we keep this term since it can be bounded as the other ones.)
Let us denote by $\mathcal{T}_{1,K}$ the first four addends on the right-hand side and by
$\mathcal{T}_{2,K}$ the fifth addend. We bound $\mathcal{T}_{1,K}$ by the Cauchy--Schwarz inequality
and also invoke the inverse inequality~\eqref{inv_tang}. Recalling the definition~\eqref{Def: specical norm} of the
$\|{\cdot}\|_{\sharp,K}$-norm, this yields
\begin{equation*}
|\mathcal{T}_{1,K}|\le C\| \eta \|_{\sharp,K}|\hat{w}_K|_{\fesE}.
\end{equation*}
Moreover, owing to~\eqref{stabilization bound} and the upper bound in~\eqref{local equivalent}, we have
\begin{equation*}
|\mathcal{T}_{2,K}| \le S_{\dK}(\mathcal{\hat{I}}^k_K(u),\mathcal{\hat{I}}^k_K(u))^{\frac12}
S_{\dK}(\hat{w}_K,\hat{w}_K)^{\frac12} \le C\| \nabla^2(u - \Pi_K^{k+2}(u))\|_{K}|\hat{w}_K|_{\fesE}.
\end{equation*}
Altogether, this implies that
\begin{equation*}
|\langle \delta_h,\hat{w}_h \rangle| \le C\left( \su \| \eta \|^2_{\sharp,K} + \| \nabla^2(u - \Pi_K^{k+2}(u))\|_{K}^2
\right)^{\frac12} \|\hat{w}_h\|_{\fesz}.
\end{equation*}
Invoking Lemma~\ref{lem: approxmation}, this completes the proof.
\end{proof}

\subsection{Error estimate}

We are now ready to establish the main result concerning the error analysis.

\begin{theorem}[$H^2$-error estimate]\label{Theorem: main}
Assume that $u\in H^{2+s}(\Omega)$ with $s>\frac{3}{2}$. The following holds true:
\begin{equation} \label{eq:err1}
\su \|\nabla^2 (u-R_K(\hat{u}_K))\|_K^2 \leq C 	\su \|u-\Pi_K^{k+2}(u)\|_{\sharp,K}^2.
\end{equation}
Consequently, if $k\ge1$, assuming $u|_K\in H^{k+3}(K)$ for all $K\in\mesh$, we have
\begin{equation} \label{eq:err2}
\su \|\nabla^2 (u-R_K(\hat{u}_K))\|_K^2 \leq C 	\su \big(h_K^{k+1}|u|_{H^{k+3}(K)}\big)^2,
\end{equation}
and if $k=0$, letting $\sigma:=\min(s-1,1)\in (\frac12,1]$, we have
\begin{equation} \label{eq:err3}
\su \|\nabla^2 (u-R_K(\hat{u}_K))\|_K^2 \leq C 	\su \big(h_K(|u|_{H^{3}(K)}+h_K^{\sigma}|u|_{H^{3+\sigma}(K)})\big)^2.
\end{equation}
\end{theorem}

\begin{proof}
Set $\hat{e}_h : = \Ihk(u) - \hat{u}_h \in \fesz$, so that $a_h(\hat{e}_h,\hat{e}_h)= -\langle \delta_h,\hat{e}_h \rangle$. The coercivity property \eqref{coercivity} implies that
\begin{equation*}
\alpha \|\hat{e}_h\|_{\fesz}^2 \le a_h(\hat{e}_h,\hat{e}_h)= -\langle \delta_h,\hat{e}_h \rangle
\le \|\delta_h\|_{(\fesz)'} \|\hat{e}_h\|_{\fesz},
\end{equation*}
so that $\|\hat{e}_h\|_{\fesz}\le \frac{1}{\alpha}\|\delta_h\|_{(\fesz)'}$.
Since $\su \|\nabla^2 R_K(\hat{e}_K)\|_K^2\le C \|\hat{e}_h\|_{\fesz}^2$, we infer from Lemma \ref{lemma:consistency} that
\begin{equation*}
\su \|\nabla^2 R_K(\hat{e}_K)\|_K^2  \leq C \su \|u-\Pi_K^{k+2} (u)\|^2_{\sharp,K}.
\end{equation*}
Since $u-R_K(\hat{u}_K)=(u-\mathcal{E}_K(u))+R_K(\hat{e}_K)$,
the triangle inequality combined with Lemma~\ref{lem: approxmation} and the above bound
proves~\eqref{eq:err1}. Furthermore, \eqref{eq:err2} results from \eqref{eq:err1} and the approximation properties of $\Pi_K^{k+2}$
(using Lemma~\ref{lemma: Polynomial approximation} and the multiplicative trace inequality~\eqref{mult_tr}). Finally, \eqref{eq:err3} is proved similarly to \eqref{eq:err2}, but this time invoking the fractional multiplicative trace inequality~\eqref{f_mult_tr}, in particular to bound $h_K^{\frac32}\|\partial_n\Delta(u-\Pi_K^{2}(u))\|_{\dK}=h_K^{\frac32}\|\partial_n\Delta u\|_{\dK}$.
\end{proof}

\begin{remark}[Regularity gap] \label{rem:regularity}
The error estimates in Theorem~\ref{Theorem: main} require $u\in H^{2+s}(\Omega)$ with $s>\frac{3}{2}$. This global regularity requirement on the exact solution can be lowered to $s>1$ by using the techniques developed in \cite{ErnGuer2021} and \cite[Chap.~40\&41]{Ern_Guermond_FEs_II_2021} in the context of second-order elliptic PDEs. Indeed, the crucial point is to give a meaning to $\partial_n\Delta u$ on each mesh face, and this can be done by applying the tools from \cite{ErnGuer2021,Ern_Guermond_FEs_II_2021} to the field $\nabla\Delta u$. Notice that the requirement $u\in H^{2+s}(\Omega)$ with $s>\frac{3}{2}$ is, however, less stringent than the one resulting from achieving optimal decay rates as soon as $k\ge1$ (see \eqref{eq:err2}).
\end{remark}

\section{HHO method in arbitrary dimension}\label{sec:HHO 3D}

In this section, we adapt the material from the above two sections to devise and analyze an HHO method to
approximate the biharmonic problem in arbitrary dimension $d\ge2$. The main difference with the previous section is that the interpolation operator
$J_{\dK}^{k+1}$ is no longer available if $d\ge3$.
The idea in this section is to raise the degree of the face unknowns representing the solution trace to $(k+2)$, and to consider $L^2$-orthogonal projections to lead the analysis.
Thus, letting $k\geq 0$ be the polynomial degree,
the local HHO space considered in this section is such that for all $K\in\mesh$,
\begin{equation}\label{HHO space 3D}
\fesE: =\mathbb{P}^{k+2}(K) \times \mathbb{P}^{k+2}(\FK)\times\mathbb{P}^{k}(\FK).
\end{equation}

\begin{remark}[$d=3$]
In 3D, on tetrahedral meshes, one can also generalize the HHO method from the previous
section by considering the canonical hybrid finite element of degree $(k+2)$ on
the mesh faces.
\end{remark}

\subsection{Reconstruction, stabilization, discrete problem, and stability}

The local reconstruction operator is still defined by \eqref{reconstruction} (or, equivalently, \eqref{eq:rec_ipp}).
Instead, the local stabilization bilinear form $S_{\dK}$ has to be slightly modified and is now such that for all
$(\hat{v}_K, \hat{w}_K)\in \fesE \times \fesE$,
\begin{equation}\label{def: new stabilization}
S_{\dK}(\hat{v}_K,\hat{w}_K)
:= h_K^{-3} \big( v_{\dK}- v_K,w_{\dK}- {w}_K\big)_{\dK}
+ h_K^{-1} \big( \Pi^{k}_{\dK}(\gamma_{\dK}- \partial_n {v}_K),\Pi^{k}_{\dK}(\chi_{\dK}- \partial_n  {w}_K)\big)_{\dK}.
\end{equation}
Notice that only $L^2$-orthogonal projections are considered.
The local bilinear form $a_K$ is defined on $\fesE \times \fesE$ as in~\eqref{eq:def_aK}.

The global HHO space is now defined as
\begin{equation}
\fes : = \mathbb{P}^{k+2}(\mesh) \times \mathbb{P}^{k+2}(\Fall)\times \mathbb{P}^k(\Fall).
\end{equation}
Focusing for simplicity on homogeneous boundary conditions,
we consider the subspace $\fesz$ obtained by zeroing out all the components
attached to the mesh boundary faces.
The discrete HHO problem is as follows: Find $\hat{u}_h\in \fesz$ such that
\begin{equation}\label{new discrete problem}
a_h(\hat{u}_h,\hat{w}_h) = {\ell} (w_{\mesh}), \qquad \forall w_h\in \fesz,
\end{equation}
where $a_h$ and ${\ell}$ are still defined as in \eqref{bilinear form}.
Moreover, as in the 2D setting, the discrete problem~\eqref{new discrete problem} is amenable to static condensation, whereby
the cell unknowns are eliminated locally in every mesh cell, leading to a global problem where the only remaining unknowns
are those attached to the mesh faces, i.e., those in $\mathbb{P}^{k+2}(\Fall)\times \mathbb{P}^k(\Fall)$.

Finally, it is readily seen that the local stability and boundedness property stated in Lemma~\ref{lem:stab_bnd} still holds true.
Therefore, the discrete bilinear form $a_h$ is coercive on $\fesz$, so that the discrete problem~\eqref{new discrete problem} is well-posed owing to the Lax--Milgram lemma.

\begin{remark}[Comparison with \cite{BoDPGK:18}] \label{rem:cost_HHO_3D}
In the present HHO method, the global problem after static condensation
features $(2k+4)$, $k\ge0$, unknowns per mesh interface, whereas this number is
$(4k+4)$, $k\ge1$, for \cite{BoDPGK:18}.
\end{remark}

\begin{remark}[\Rev{Boundary conditions}]
\Rev{In the non-homogeneous case, similarly to Remark~\ref{rem:non-homo_2D}, the value of the components of the HHO solution attached to the mesh boundary faces is assigned by means of the projections $\Pi^{k+2}_{\dK}(g_D|_{\dK})$ and $\Pi^{k}_{\dK}(g_N|_{\dK})$. It is also possible to enforce the boundary conditions $u=\partial_{nn}u=0$ on $\partial \Omega$ by proceeding as in Remark~\ref{rem:non-homo_2D}. }
\end{remark}

\subsection{Polynomial approximation, consistency and error estimate}

For all $K\in \mesh$, the local reduction operator
$\mathcal{\hat{I}}^k_K : H^2(K) \rightarrow \fesE$
is now defined such that for all $v\in H^2(K)$,
\begin{equation}\label{new reduction}
\mathcal{\hat{I}}^k_K(v):= (\Pi_{K}^{k+2}(v),\Pi_{\dK}^{k+2} (v), \Pi_{\dK}^k
(\n_K{\cdot}\nabla v)) \in \fesE.
\end{equation}
We also define the operator $\tilde{\mathcal{E}}_K:= R_K \circ \mathcal{\hat{I}}^k_K: H^2(K) \rightarrow \mathbb{P}^{k+2}(K)$.
Although this operator is no longer the $H^2$-elliptic projection, we can show that it still enjoys the same approximation properties as those derived in Lemma~\ref{lem: approxmation}.
Recall that the $\|{\cdot}\|_{\sharp,K}$-norm is defined in~\eqref{Def: specical norm}.

\begin{lemma}[Polynomial approximation]
The following holds true for all $K\in \mesh$ and all $v\in H^{2+s}(K)$ with $s>\frac{3}{2}$:
\begin{equation}\label{new elliptic projection}
\|v - \tilde{\mathcal{E}}_K(v)\|_{\sharp, K} \leq C \|v - \Pi_K^{k+2}(v)\|_{\sharp, K}.
\end{equation}
Moreover, for all $K\in \mesh$ and all $v\in H^2(K)$, we have
\begin{equation}\label{new stabilization bound}
S_{\dK}(\mathcal{\hat{I}}^k_K(v),\mathcal{\hat{I}}^k_K(v))^{\frac12} \leq
C \| v - \Pi_K^{k+2}(v)\|_{\sharp,K}.
\end{equation}
\end{lemma}

\begin{proof}
(1) Using  the definition \eqref{reconstruction} of $R_K$, the definition \eqref{new reduction} of $\mathcal{\hat{I}}^k_K$, and the orthogonality property of the $L^2$-projections $\Pi_K^{k+2}$ and $\Pi_{\dK}^{k+2}$, we infer that for all $w \in \mathbb{P}^{k+2}(K)$,
\begin{equation*}
\begin{aligned}
(\nabla^2 \tilde{\mathcal{E}}_K(v), \nabla^2 w)_K = {}& (\nabla^2 \Pi_{K}^{k+2}(v) , \nabla^2 w)_K
- (\Pi_{\dK}^{k+2} (v)-\Pi_{K}^{k+2}(v), \partial_n \Delta  w)_{\dK} \\
&+ (\Pi_{\dK}^k (\partial_n v)
- \partial_n \Pi_{K}^{k+2}(v),  \partial_{nn}  w)_{\dK}
+ (\partial_t (\Pi_{\dK}^{k+2} (v)-\Pi_{K}^{k+2}(v)),  \partial_{nt}  w)_{\dK} \\
= {}& (\nabla^2 \Pi_{K}^{k+2}(v) , \nabla^2 w)_K
- (v-\Pi_{K}^{k+2}(v), \partial_n \Delta  w)_{\dK} \\
&+ (\partial_n (v-\Pi_{K}^{k+2}(v)), \partial_{nn}  w)_{\dK}
+ (\partial_t (\Pi_{\dK}^{k+2}(v)-\Pi_{K}^{k+2}(v)), \partial_{nt}  w)_{\dK},
\end{aligned}
\end{equation*}
Taking $w := \tilde{\mathcal{E}}_K(v) - \Pi_{K}^{k+2}(v)$, rearranging the terms,
and invoking the Cauchy--Schwarz inequality together with the inverse inequalities \eqref{trace_inv}, \eqref{H1_inv}, \eqref{inv_tang} leads to
\begin{multline*}
\|\nabla^2 (\tilde{\mathcal{E}}_K(v) - \Pi_{K}^{k+2}(v))\|_K \\ \le C\big( h_K^{-\frac32}
\|v-\Pi_{K}^{k+2}(v)\|_{\dK} + h_K^{-\frac12}
\|\partial_{n}(v-\Pi_{K}^{k+2}(v)) \|_{\dK}
+ h_K^{-\frac32} \|\Pi_{\dK}^{k+2}(v)-\Pi_{K}^{k+2}(v)\|_{\dK}\big).
\end{multline*}
Concerning the rightmost term, we observe that
$\Pi_{\dK}^{k+2}(v)-\Pi_{K}^{k+2}(v)=\Pi_{\dK}^{k+2}(v-\Pi_{K}^{k+2}(v))$, so that
using the $L^2$-stability of $\Pi^{k+2}_{\dK}$, we obtain
\begin{equation*}
\|\nabla^2 (\tilde{\mathcal{E}}_K(v) - \Pi_{K}^{k+2}(v))\|_K \le C\big( h_K^{-\frac32}
\|v-\Pi_{K}^{k+2}(v)\|_{\dK} + h_K^{-\frac12}
\|\partial_{n}(v-\Pi_{K}^{k+2}(v)) \|_{\dK}
\big).
\end{equation*}
The trace inequality \eqref{eq:approx_faces} then shows that
\begin{equation*}
\|\nabla^2 (\tilde{\mathcal{E}}_K(v) - \Pi_{K}^{k+2}(v))\|_K \le C \|\nabla^2(v-\Pi_{K}^{k+2}(v))\|_K.
\end{equation*}
The proof of~\eqref{new elliptic projection} can now be completed by invoking the
triangle inequality.
\\
(2) Let us now prove \eqref{new stabilization bound}. We have
\begin{equation*}
S_{\dK}(\mathcal{\hat{I}}^k_K(v),\mathcal{\hat{I}}^k_K(v))^{\frac12}
\le
h_K^{-\frac32} \| \Pi^{k+2}_{\dK}(v)- \Pi^{k+2}_K (v) \|_{\dK}
+ h_K^{-\frac12}
\| \Pi^{k}_{\dK}(\partial_{n}v- \partial_{n} \Pi^{k+2}_K (v)) \|_{\dK},
\end{equation*}
where we used that $\Pi^{k}_{\dK} \circ \Pi^{k}_{\dK}=\Pi^{k}_{\dK}$ for the second term
on the right-hand side.
To bound the first term on the right-hand side, we invoke the same arguments as in the
first step of this proof leading to
\begin{equation*}
h_K^{-\frac32} \| \Pi^{k+2}_{\dK}(v)- \Pi^{k+2}_K (v) \|_{\dK} \le C
\|\nabla^2(v-\Pi_{K}^{k+2}(v))\|_K.
\end{equation*}
Furthermore, the second term has already been bounded in the proof of Lemma~\ref{lem:local_red}. This completes the proof.
\end{proof}

The global reduction operator $\Ihk:H^2(\Omega)\rightarrow \fes$ is defined
such that for all $v\in H^2(\Omega)$,
\begin{equation} \label{def:Ihk_bis}
\Ihk(v):=\big( (\Pi_K^{k+2}(v))_{K\in\mesh},(\Pi^{k+2}_F(v))_{F\in\Fall},(\Pi_F^{k}(\n_F{\cdot}\nabla v))_{F\in\Fall} \big)\in \fes,
\end{equation}
so that the local components of $\Ihk(v)$ are $\mathcal{\hat{I}}^k_K(v|_K)$ for all $K\in\mesh$. The consistency error $\delta_h \in (\fesz)^\prime$ can now be defined as in~\eqref{eq:def_delta} and it can be bounded as in Lemma \ref{lemma:consistency}.
Finally, the error estimate and its proof are the same as those from Theorem \ref{Theorem: main} (and are not repeated for brevity).

\section{HHO method with Nitsche's boundary penalty}\label{sec:HHO-N}

In this section, we combine the HHO methods devised in the previous sections with Nitsche's boundary-penalty technique to enforce the boundary conditions in a weak manner. For brevity, we only discuss the HHO method presented in Section~\ref{sec:HHO 2D}, but the following developments can be readily applied to the HHO method from Section~\ref{sec:HHO 3D}. To allow for a bit more generality, we detail here the case of non-homogeneous boundary conditions. Thus, the model problem is as follows:
\begin{equation}
\Delta^2u=f\;\text{in $\Omega$}, \qquad u=g_D,\; \partial_nu=g_N\;\text{in $\partial\Omega$},
\end{equation}
where the assumptions on the boundary data $g_D$ and $g_N$ are given in Remark~\ref{rem:non_homo}. We set $\bG:=(\nabla u)|_{\partial\Omega}$ and notice that $\bG$
is explicitly known in terms of the
boundary data $g_D$ and $g_N$ since $\bG=g_N \n + (\partial_t g_D){\bm t}$.

Hinging on the ideas from \cite{BurEr:18,BCDE:21} for second-order elliptic PDEs, the
HHO-Nitsche (HHO-N) method devised in this section does not place
any discrete unknown on the
mesh boundary faces, but only in the mesh cells and the mesh interfaces.
Thus, for every mesh cell $K \in \mesh$, we define the subsets
\begin{equation}
\dKi := \overline{\dK\cap \Omega}, \qquad
\dKb:=\dK\cap \partial\Omega,
\end{equation}
as well as $\FKi: = \FK \cap \Fint$ and $\FKb: = \FK \cap \Fb$.
The mesh cells having at least one boundary face are collected in the subset
$\meshb:=\{K\in\mesh\;|\; \FKb\ne\emptyset\}$, and we set $\meshi:=\mesh\setminus\meshb$.

Letting $k\geq 0$ be the polynomial degree,
the local HHO-N space is such that for all $K\in\mesh$,
\begin{equation}\label{HHO-N space}
\fesE: =\mathbb{P}^{k+2}(K) \times \mathbb{P}^{k+1}(\FKi)\times\mathbb{P}^{k}(\FKi),
\end{equation}
and the corresponding global HHO-N space is now defined as
\begin{equation} \label{eq:global_HHO_N}
\fes : = \mathbb{P}^{k+2}(\mesh) \times \mathbb{P}^{k+1}(\Fint)\times \mathbb{P}^k(\Fint).
\end{equation}

\subsection{Reconstruction, stabilization, discrete problem, and stability}

The definition of the local reconstruction operator is slightly modified with respect to \eqref{reconstruction}. Indeed, $R_K\upi:\fesE\rightarrow \mathbb{P}^{k+2}(K)$ is now such that
for all $\hat{v}_K\in \fesE$,
\begin{equation}\label{eq:reconstruction HHO-N}
\begin{aligned}
(\nabla^2 R_K\upi (\hat{v}_K), \nabla^2 w)_{K}
={}&(\nabla ^2 {v}_K, \nabla^2 w)_{K} +
(v_K -v_{\partial K} , \partial_n \Delta  w)_{\dKi}
- (\partial_n v_K - \gamma_{\partial K},  \partial_{nn}  w)_{\dKi}
\\
&
- (\partial_t (v_K - v_{\partial K}),  \partial_{nt}  w)_{\dKi}
+ (v_K , \partial_n \Delta  w)_{\dKb}
- (\nabla  v_K, \nabla  \partial_{n}  w)_{\dKb},
\end{aligned}
\end{equation}
for all $w\in \mathbb{P}^{k+2}(K)^\perp$ together with the condition $(R_K\upi (\hat{v}_K),\xi)_K=(v_K,\xi)_K$ for all $\xi\in\mathbb{P}^1(K)$.
Equivalently, owing to the integration by parts formula~\eqref{eq:ipp2}, we have
\begin{equation}\label{eq:rec_ipp HHO-N}
(\nabla^2 R_K\upi(\hat{v}_K), \nabla^2 w)_K = (v_K, \Delta^2 w)_K
- (v_{\dK} , \partial_n \Delta  w)_{\dKi}
+ (\gamma_{\dK},  \partial_{nn}  w)_{\dKi}
+ (\partial_t v_{\dK},  \partial_{nt}  w)_{\dKi}.
\end{equation}
Dropping the integral over $\dKb$ for the three rightmost terms in~\eqref{eq:rec_ipp HHO-N} is, loosely speaking, a consistent operation in the case of homogeneous boundary conditions. In the general case, we need to lift the boundary data in every mesh cell $K\in \meshb$ by means of the lifting operator $\mathcal{L}_K:H^2(K)\rightarrow \mathbb{P}^{k+2}(K)$ such that for all $v\in H^2(K)$,
\begin{equation}\label{def: lifting}
(\nabla^2 \mathcal{L}_K(v), \nabla^2 w)_{K}
=- (v,  \partial_{n} \Delta w)_{\dKb}
+(\nabla v, \nabla {\partial_n  w} )_{\dKb},
\end{equation}
for all $w \in \mathbb{P}^{k+2}(K)^\perp$, together with the condition $(\mathcal{L}_K(v),\xi)_K =  0$ for all $\xi \in \mathbb{P}^{1}(K)$. Notice that $\mathcal{L}_K(u|_K)$
is fully computable from the boundary data $g_D$ and $g_N$.
For convenience, we set $\mathcal{L}_K(v):=0$ for all $K\in \meshi$.

The local stabilization bilinear form $S_{\dK}$ is also slightly modified and is now such that for all $(\hat{v}_K, \hat{w}_K)\in \fesE \times \fesE$,  we have
$S_{\dK}(\hat{v}_K,\hat{w}_K):= S\upi_{\dK}(\hat{v}_K,\hat{w}_K)+S\upb_{\dK}(v_K,w_K)$ with
\begin{align}
S\upi_{\dK}(\hat{v}_K,\hat{w}_K) :={}& h_K^{-3}
\big( J^{k+1}_{\dK}(v_{\dK}- v_K),J^{k+1}_{\dK}(w_{\dK}- {w}_K)\big)_{\dKi} \nonumber \\
& + h_K^{-1}\big( \Pi^{k}_{\dK}(\gamma_{\dK}- \partial_n {v}_K), \Pi^{k}_{\dK}(\chi_{\dK}- \partial_n  {w}_K) \big)_{\dKi}, \label{def: new stabilization int HHO-N} \\
S\upb_{\dK}(v_K,w_K) :={}&{}h_K^{-3}
\big( v_K, {w}_K\big)_{\dKb}
+h_K^{-1}\big(   \nabla {v}_K,  \nabla  {w}_K \big)_{\dKb},\label{def: new stabilization bd HHO-N}
\end{align}
where $S\upb_{\dK}$ represents the boundary-penalty contribution
and acts only on the cell components. We emphasize that
$S\upb_{\dK}$ does not need to be scaled by a weighting coefficient to be taken large
enough. Finally, the local bilinear form $a_K$ is defined on $\fesE \times \fesE$
as in~\eqref{eq:def_aK}.

The discrete HHO-N problem is as follows: Find $\hat{u}_h\in \fes$ such that
\begin{equation}\label{new discrete problem HHO-N}
	a_h(\hat{u}_h,\hat{w}_h) = {\ell}_h (\hat{w}_h), \qquad \forall \hat{w}_h\in \fes,
\end{equation}
where $a_h$ is still assembled cellwise as in \eqref{bilinear form} yielding
\begin{equation}
a_h(\hat{v}_h,\hat{w}_h) :=
\su (\nabla^2R_K\upi(\hat{v}_K),\nabla^2R_K\upi(\hat{w}_K))_K
+ \su S\upi_{\dK}(\hat{v}_K,\hat{w}_K)
+ \sum_{K\in\meshb} S\upb_{\dK}(v_K,w_K),
\end{equation}
whereas the linear
form ${\ell}_h$ now acts as follows:
\begin{equation}\label{linear form HHO-N}
\begin{aligned}
\ell_h(\hat{w}_h):= &\su (f,{w}_K)_K
+ \sum_{K\in\meshb} \Big\{ h_K^{-3}(g_D, w_{K})_{\dKb} + h_K^{-1} (\bG, \nabla w_{K})_{\dKb}\\
&+ \big(g_D,\partial_n \Delta R_K\upi(\hat{w}_{K}) \big)_{\dKb}
- \big(\bG, \nabla \partial_n R_K\upi(\hat{w}_{K}) \big)_{\dKb} \Big\}.
\end{aligned}
\end{equation}
Notice that
\begin{equation} \label{eq:rewriting_rhs}
\ell_h(\hat{w}_h)= \su (f,{w}_K)_K + \sum_{K\in\meshb} \Big\{
S\upb_{\dK}(u|_K,w_K) - (\nabla^2\mathcal{L}_K(u|_K),\nabla^2R_K\upi(\hat{w}_K))_K\Big\}.
\end{equation}
Notice also that $\ell_h(\hat{w}_h)=\su (f,{w}_K)_K =(f,w_{\mesh})_\Omega$ if the boundary
conditions are homogeneous (so that only the cell component of $\hat{w}_h$ is needed to
assemble $\ell_h$). As in the previous sections,
the discrete problem~\eqref{new discrete problem HHO-N} is amenable
to static condensation, whereby all the cell unknowns are eliminated
locally in every mesh cell, leading to a global problem where the only remaining unknowns
are those attached to the mesh interfaces.

It is easy to see that the local stability and boundedness property stated in Lemma~\ref{lem:stab_bnd} still holds true in the updated $H^2$-seminorm
\begin{equation}\label{H2_seminorm_elem HHO-N}
|\hat{v}_K|^2_{\fesE}: = \|\nabla^2 v_K\|_K^2
+ h_K^{-3} \| v_{\dK} - v_K\|_{\dKi}^2
+ h_K^{-1}\|  \gamma_{\dK} - \partial_n v_K\|_{\dKi}^2
+ S\upb_{\dK}(v_K,v_K),
\end{equation}
recalling that $S\upb_{\dK}(v_K,v_K)=h_K^{-3}  \|v_{K}\|^2_{\dKb}
+  h_K^{-1}   \|\nabla v_{K}\|^2_{\dKb}$.
Since $\|\hat{v}_h\|_{\fes}^2:=\su |\hat{v}_K|^2_{\fesE}$ defines a norm on the global
HHO space $\fes$ defined in~\eqref{eq:global_HHO_N},
the discrete bilinear form $a_h$ is coercive on $\fes$, and the discrete
problem~\eqref{new discrete problem HHO-N} is well-posed owing to the Lax--Milgram lemma.

\subsection{Polynomial approximation, consistency and error estimate}

For all $K\in \mesh$, we define the local reduction operator
$\mathcal{\hat{I}}^k_K : H^2(K) \rightarrow \fesE$ such that for all $v\in H^2(K)$,
\begin{equation}\label{interpoltation_N}
\mathcal{\hat{I}}^k_K(v):= \big(\Pi_{K}^{k+2}(v),J_{\dKi}^{k+1} (v), \Pi_{\dKi}^k (\n_K{\cdot}\nabla v)\big) \in \fesE,
\end{equation}
with obvious notation regarding the operators $J_{\dKi}^{k+1}$ and $\Pi_{\dKi}^k$.
Let us set
\begin{equation}
\mathcal{E}_K\upi := R_K\upi \circ \mathcal{\hat{I}}^k_K : H^2(K) \rightarrow \mathbb{P}^{k+2}(K).
\end{equation}
A straightforward verification (omitted for brevity) shows that the operator
$\mathcal{E}_K:= \mathcal{E}_K\upi + \mathcal{L}_K : H^2(K) \rightarrow
\mathbb{P}^{k+2}(K)$ coincides indeed with the $H^2$-elliptic projection
defined in~\eqref{H2_ell_proj}. Therefore, owing to Lemma~\ref{lem: approxmation},
there is $C$ such that for all $K\in \mesh$ and all
$v\in H^{2+s}(K)$, $s>\frac{3}{2}$,
\begin{equation}\label{elliptic projection N}
\|v - (\mathcal{E}\upi_K(v)+\mathcal{L}_K(v))\|_{\sharp, K} \leq
C \|v - \Pi_K^{k+2}(v)\|_{\sharp, K},
\end{equation}
where the $\|{\cdot}\|_{\sharp,K}$-norm is defined in~\eqref{Def: specical norm}.
Moreover, by restricting the arguments to the mesh interfaces in the proof of
Lemma~\ref{lem:local_red}, we infer that there is $C$ such that
for all $K\in\mesh$ and all $v\in H^2(K)$,
\begin{equation}\label{new stabilization bound HHO-N}
S\upi_{\dK}(\mathcal{\hat{I}}^k_K(v),\mathcal{\hat{I}}^k_K(v))^{\frac12} \leq
C \|\nabla^2 (v - \Pi_K^{k+2}(v))\|_K.
\end{equation}

The global reduction operator $\Ihk:H^2(\Omega)\rightarrow \fes$ is defined
such that for all $v\in H^2(\Omega)$,
\begin{equation}
\Ihk(v):=\big( (\Pi_K^{k+2}(v))_{K\in\mesh},(J^{k+1}_F(v))_{F\in\Fint},(\Pi_F^{k}(\n_F{\cdot}\nabla v))_{F\in\Fint} \big)\in \fes,
\end{equation}
recalling that $v$ and $\nabla v$ are single-valued on every $F\in\Fint$ for
all $v\in H^2(\Omega)$. As above, the local components of $\Ihk(v)$ attached to $K$
and its faces in $\FKi$ are $\mathcal{\hat{I}}^k_K(v|_K)$ for all $K\in\mesh$.
We define the consistency error $\delta_h \in (\fes)^\prime$ such that
$\langle \delta_h,\hat{w}_h \rangle:= \ell(w_{\mesh}) - a_h(\Ihk(u), \hat{w}_h)$,
for all $\hat{w}_h\in \fes$,
where $\langle\cdot,\cdot\rangle$ now denotes the duality pairing between $(\fes)'$ and $\fes$.

\begin{lemma}[Consistency]\label{lemma: new consistency HHO-N}
Assume that $u\in H^{2+s}(\Omega)$ with $s>\frac{3}{2}$. The following holds true:
\begin{equation} \label{new consistency HHO-N}
\|\delta_h\|_{(\fes)'} := \sup_{\hat{w}_h\in\fes}
\frac{|\langle \delta_h,\hat{w}_h \rangle|}{\|\hat{w}_h\|_{\fes}}
\leq C \left( \su \|u-\Pi_K^{k+2} (u)\|^2_{\sharp,K}\right)^{\frac12}.
\end{equation}
\end{lemma}

\begin{proof}
The proof is similar to that of Lemma~\ref{lemma:consistency},
so we only sketch it.
Let $\hat{w}_h\in \fes$ having local components $(w_K,w_{\dK},\chi_{\dK})$
for all $K\in\mesh$. On the one hand, we have
\begin{align*}
&\ell_h(\hat{w}_h) + \sum_{K\in\meshb} (\nabla^2 \mathcal{L}_K(u|_K),
\nabla^2 (R_K\upi(\hat{w}_{K})) )_{K} \\
&=
\su \Big\{ (\nabla^2 u, \nabla^2 {w}_K)_K
+ (\partial_n \Delta u , w_K )_{\dKb}
- (\nabla  \partial_n u , \nabla w_K )_{\dKb}
\\
&\quad + (\partial_n \Delta u , w_K - w_{\dK})_{\dKi}
- (\partial_{nn}  u , \partial_{n} w_K  - \chi_{\dK})_{\dKi}
- (\partial_{nt}  u , \partial_{t} (w_K - w_{\dK}))_{\dKi} \Big\} \\
&\quad + \sum_{K\in\meshb} \Big\{
h_K^{-3} \left( u, w_{K}\right)_{\dKb}
+h_K^{-1} \left( \nabla u, \nabla w_{K}\right)_{\dKb}
\Big\}.
\end{align*}
On the other hand, recalling that $\mathcal{E}_K= \mathcal{E}_K\upi + \mathcal{L}_K$,
we have
\begin{align*}
&a_h(\Ihk(u), \hat{w}_h) + \sum_{K\in\meshb} (\nabla^2 \mathcal{L}_K(u|_K),
\nabla^2 (R_K\upi(\hat{w}_{K})) )_{K} \\
&= \su \Big\{ (\nabla^2 \mathcal{E}_K(u),\nabla^2 R\upi_K(\hat{w}_K))_K
+ S\upi_{\dK}(\mathcal{\hat{I}}^k_K(u),\hat{w}_K) \Big\}
+ \sum_{K\in\meshb} S\upb_{\dK}(\Pi_K^{k+2}(u),w_K),
\end{align*}
so that we have
\begin{align*}
&a_h(\Ihk(u), \hat{w}_h) + \sum_{K\in\meshb} (\nabla^2 \mathcal{L}_K(u|_K),
\nabla^2 (R_K\upi(\hat{w}_{K})) )_{K} \\
&= \su \Big\{ (\nabla^2 \mathcal{E}_K(u), \nabla^2 {w}_K)_K
+ (\partial_n \Delta \mathcal{E}_K(u), w_K)_{\dKb }
- (\nabla \partial_n\mathcal{E}_K(u), \nabla w_K)_{\dKb }
\\
&\quad + (\partial_n \Delta \mathcal{E}_K(u), w_K - w_{\dK})_{\dKi}
- (\partial_{nn} \mathcal{E}_K(u), \partial_{n} w_K  - \chi_{\dK})_{\dKi}
- (\partial_{nt} \mathcal{E}_K(u), \partial_{t} (w_K - w_{\dK}))_{\dKi}
\\
&\quad + S\upi_{\dK}(\mathcal{\hat{I}}^k_K(u),\hat{w}_K) \Big\} + \sum_{K\in\meshb} \Big\{
h_K^{-3} ( \Pi_{K}^{k+2}(u), w_{K})_{\dKb}
+h_K^{-1} ( \nabla \Pi_{K}^{k+2}(u), \nabla w_{K})_{\dKb}\Big\}.
\end{align*}
Defining the function $\eta$ cellwise as $\eta|_K:=u|_K-\mathcal{E}_K(u|_{K})$
for all $K\in\mesh$, we infer that
\begin{align*}
&\langle \delta_h,\hat{w}_h \rangle
= \su \Big\{ (\nabla^2 \eta, \nabla^2 {w}_K)_K
+ (\partial_n \Delta \eta, w_K )_{\dKb}
- (\nabla \partial_n\eta, \nabla w_K )_{\dKb} \\
& + (\partial_n \Delta \eta, w_K - w_{\dK})_{\dKi}
- (\partial_{nn} \eta, \partial_{n} w_K  - \chi_{\dK})_{\dKi}
- (\partial_{nt} \eta, \partial_{t} (w_K - w_{\dK}))_{\dKi}
- S\upi_{\dK}(\mathcal{\hat{I}}^k_K(u),\hat{w}_K)\Big\} \\
&
+ \sum_{K\in\meshb} \Big\{ h_K^{-3} ( u-\Pi_{K}^{k+2}(u), w_{K})_{\dKb}
+h_K^{-1} ( \nabla  (u-\Pi_{K}^{k+2}(u)), \nabla w_{K})_{\dKb}\Big\}.
\end{align*}
All the terms on the right-hand side can now be bounded
by means of the Cauchy--Schwarz inequality.
For the first, fourth, fifth, and sixth terms, we use~\eqref{elliptic projection N},
for the seventh term (involving $S\upi_{\dK}$), we use~\eqref{new stabilization bound HHO-N},
and for the eighth and ninth terms, we invoke the trace inequality~\eqref{eq:approx_faces}.
\end{proof}

We are now ready to establish our main error estimate.

\begin{theorem}[$H^2$-error estimate]\label{Theorem: main N}
Assume that $u\in H^{2+s}(\Omega)$ with $s>\frac{3}{2}$. The following holds true:
\begin{equation} \label{eq:err1_N}
\su \|\nabla^2 (u-R\upi_K(\hat{u}_K)-\mathcal{L}_K(u))\|_K^2
\leq C 	\su \|u-\Pi_K^{k+2}(u)\|_{\sharp,K}^2.
\end{equation}
Consequently, if $k\ge1$, assuming $u|_K\in H^{k+3}(K)$ for all $K\in\mesh$, we have
\begin{equation} \label{eq:err2_N}
\su \|\nabla^2 (u-R\upi_K(\hat{u}_K)-\mathcal{L}_K(u))\|_K^2 \leq C
\su \big(h_K^{k+1}|u|_{H^{k+3}(K)}\big)^2,
\end{equation}
and if $k=0$, letting $\sigma:=\min(s-1,1)\in (\frac12,1]$, we have
\begin{equation} \label{eq:err3_N}
\su \|\nabla^2 (u-R\upi_K(\hat{u}_K)-\mathcal{L}_K(u))\|_K^2 \leq C \su \big(h_K(|u|_{H^{3}(K)}+h_K^{\sigma}|u|_{H^{3+\sigma}(K)})\big)^2.
\end{equation}
\end{theorem}

\begin{proof}
As in the proof of Theorem~\ref{Theorem: main}, one shows that
\Rev{$$\su \|\nabla^2 R\upi_K(\hat{e}_K)\|_K^2  \leq C \su \|u-\Pi_K^{k+2} (u)\|^2_{\sharp,K},$$}
where $\hat{e}_h : = \Ihk(u) - \hat{u}_h \in \fes$ is the discrete error.
Since $u-R\upi_K(\hat{u}_K)-\mathcal{L}_K(u)=(u-\mathcal{E}\upi_K(u)-\mathcal{L}_K(u))
+R\upi_K(\hat{e}_K)$ for all $K\in\mesh$,
the triangle inequality combined with~\eqref{elliptic projection N} and the above bound
on the discrete error proves~\eqref{eq:err1}.
Finally, \eqref{eq:err2_N} and \eqref{eq:err3_N} are established by invoking the same arguments as above.
\end{proof}

\section{Numerical examples} \label{sec:Numerical example}

In this section, we present numerical examples to illustrate the theoretical results on the present HHO methods and also to compare their numerical performance with respect to other methods from the literature.

\subsection{Convergence rates and computational performance of HHO methods} \label{sec:conv_rates}

We select $f$ on $\Omega:=(0,1)^2$ so that the exact solution to \eqref{pde} is
$u(x,y) = \sin(\pi x)^2  \sin(\pi y)^2$ with homogeneous
boundary conditions. We consider the two HHO methods analyzed above. For clarity, we
term ``HHO(A)'' the method  introduced in Section \ref{sec:HHO 2D} with discrete
unknowns in $\mathbb{P}^{k+2}(\mesh) \times \mathbb{P}^{k+1}(\Fint) \times
\mathbb{P}^{k}(\Fint)$ and ``HHO(B)'' the method introduced in Section \ref{sec:HHO 3D}
with discrete unknowns in $\mathbb{P}^{k+2}(\mesh) \times
\mathbb{P}^{k+2}(\Fint) \times \mathbb{P}^{k}(\Fint)$. Additionally, we
consider the method termed ``HHO(C)'' mentioned in Remark \ref{cheap HHO}
where the discrete unknowns are in $\mathbb{P}^{k+1}(\mesh)
\times \mathbb{P}^{k+1}(\Fint) \times \mathbb{P}^{k}(\Fint)$.
We employ polynomial degrees $k\in\{0,\ldots,5\}$.
Since we consider various polynomial degrees, and despite an $hp$-analysis falls beyond
the present scope, we implement the stabilization terms in \eqref{def: stabilization},
\eqref{def: new stabilization}, and \eqref{eq:cheap_HHO}
with $h_K^{-1}$ replaced by $(k+1)^2 h_K^{-1}$
for all $K\in\mesh$. All the computations were run with
Matlab R2018a on the NEF
platform at INRIA Sophia Antipolis M\'editerran\'ee using 12 cores,
and all the linear systems after static condensation are solved using the \verb*|backslash| function. \Rev{The  algorithm for solving the symmetric positive definite linear systems is the Cholesky factorization.}

\begin{figure}[htb]
\centering
\includegraphics[scale=0.2]{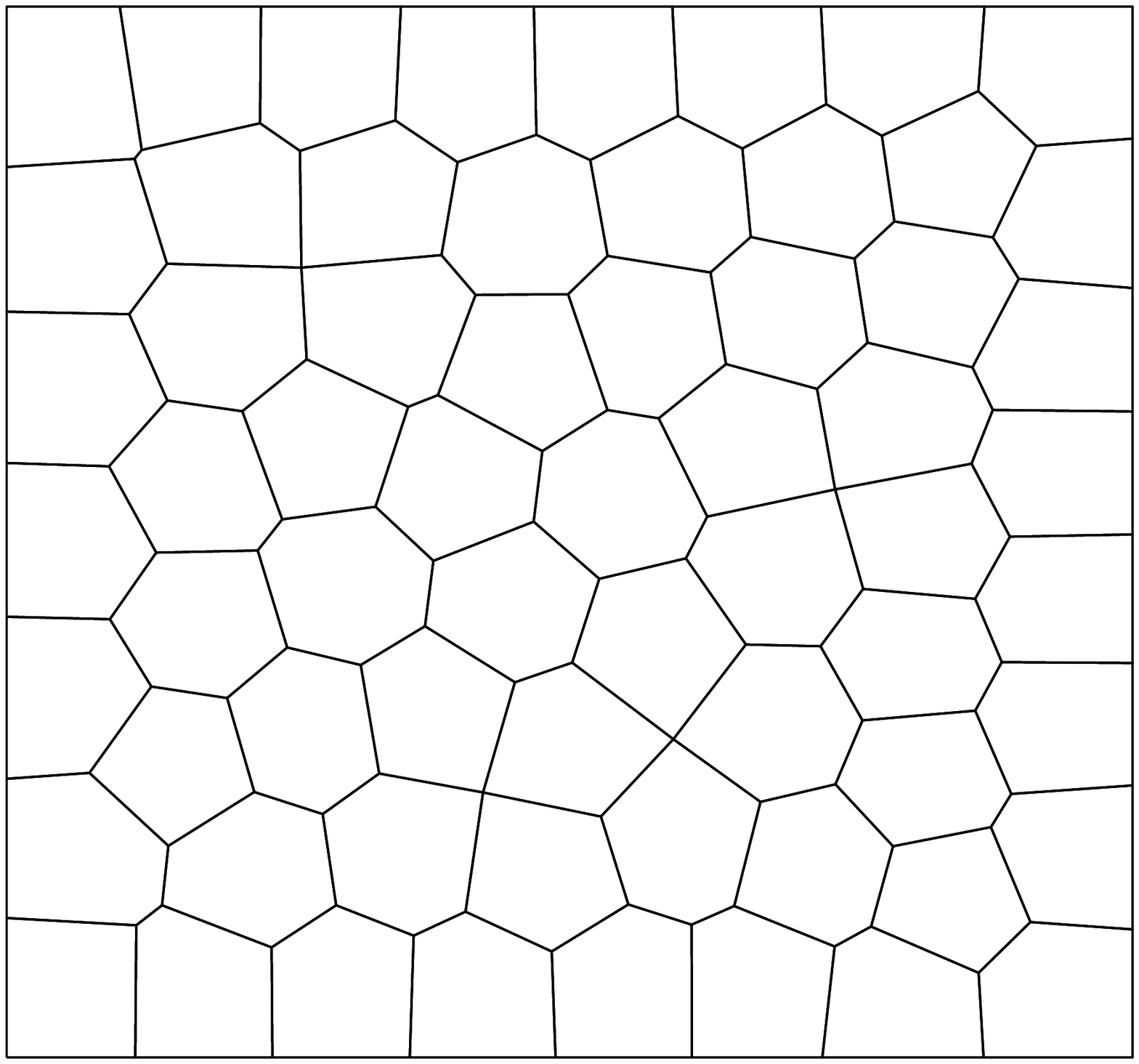}
\hspace{1cm}
\includegraphics[scale=0.2]{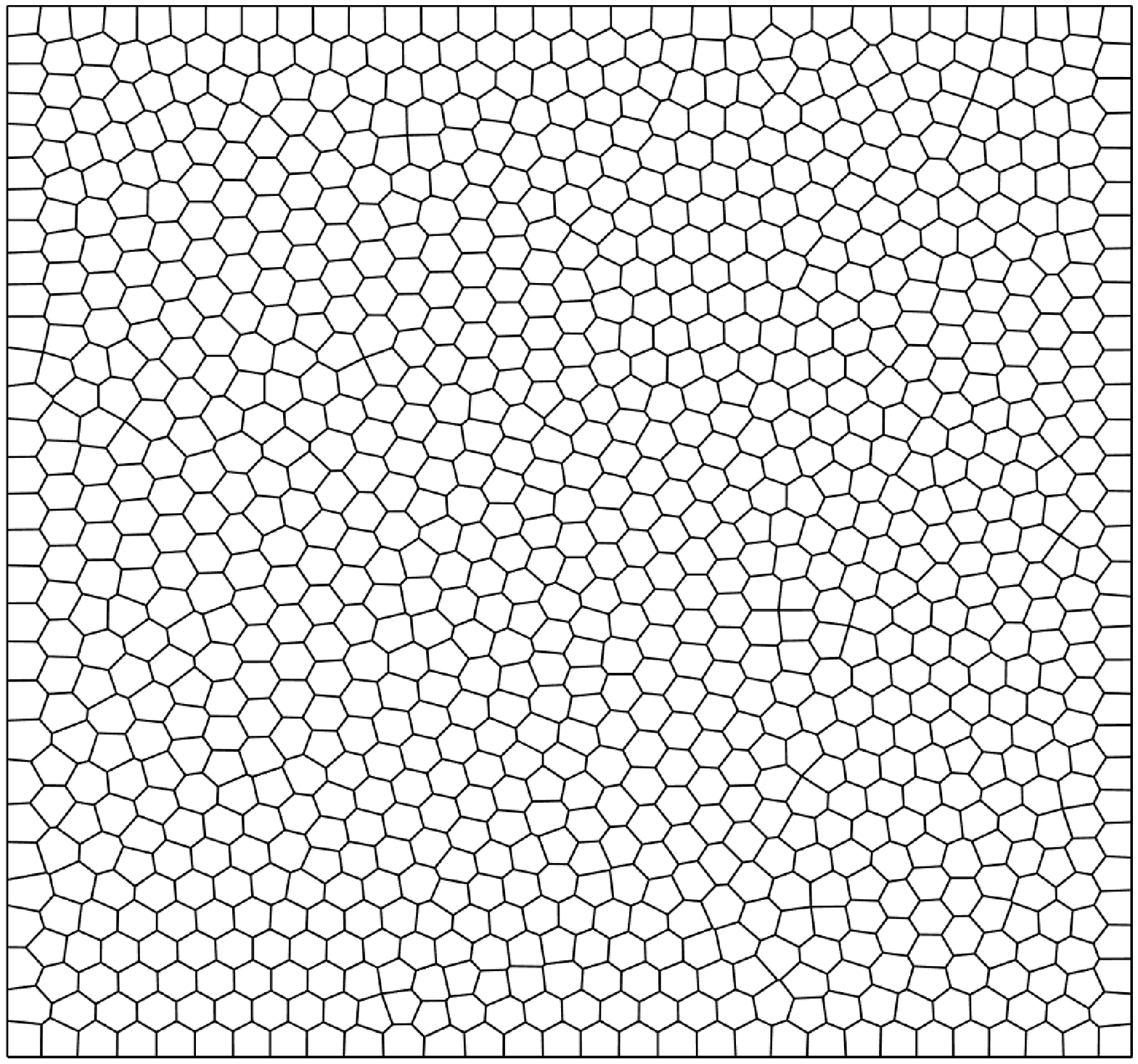}
\caption{Two examples of polygonal (Voronoi-like) meshes with $64$ (left) and $1{,}024$ polygons}.\label{ex1:mesh_figure}
\end{figure}

\begin{figure}[htb]
\centering
\includegraphics[scale=0.32]{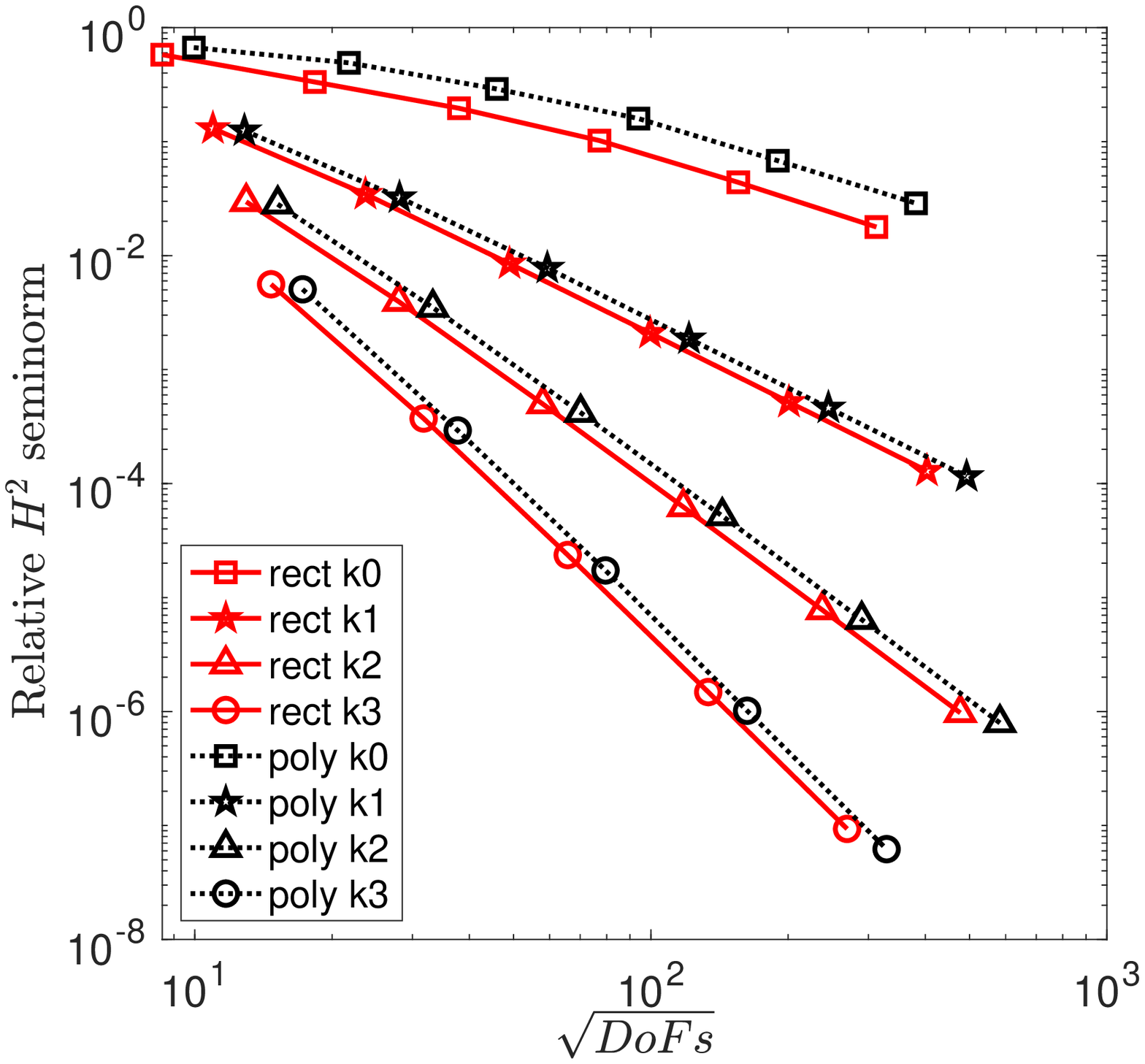}
\hspace{2cm}
\includegraphics[scale=0.32]{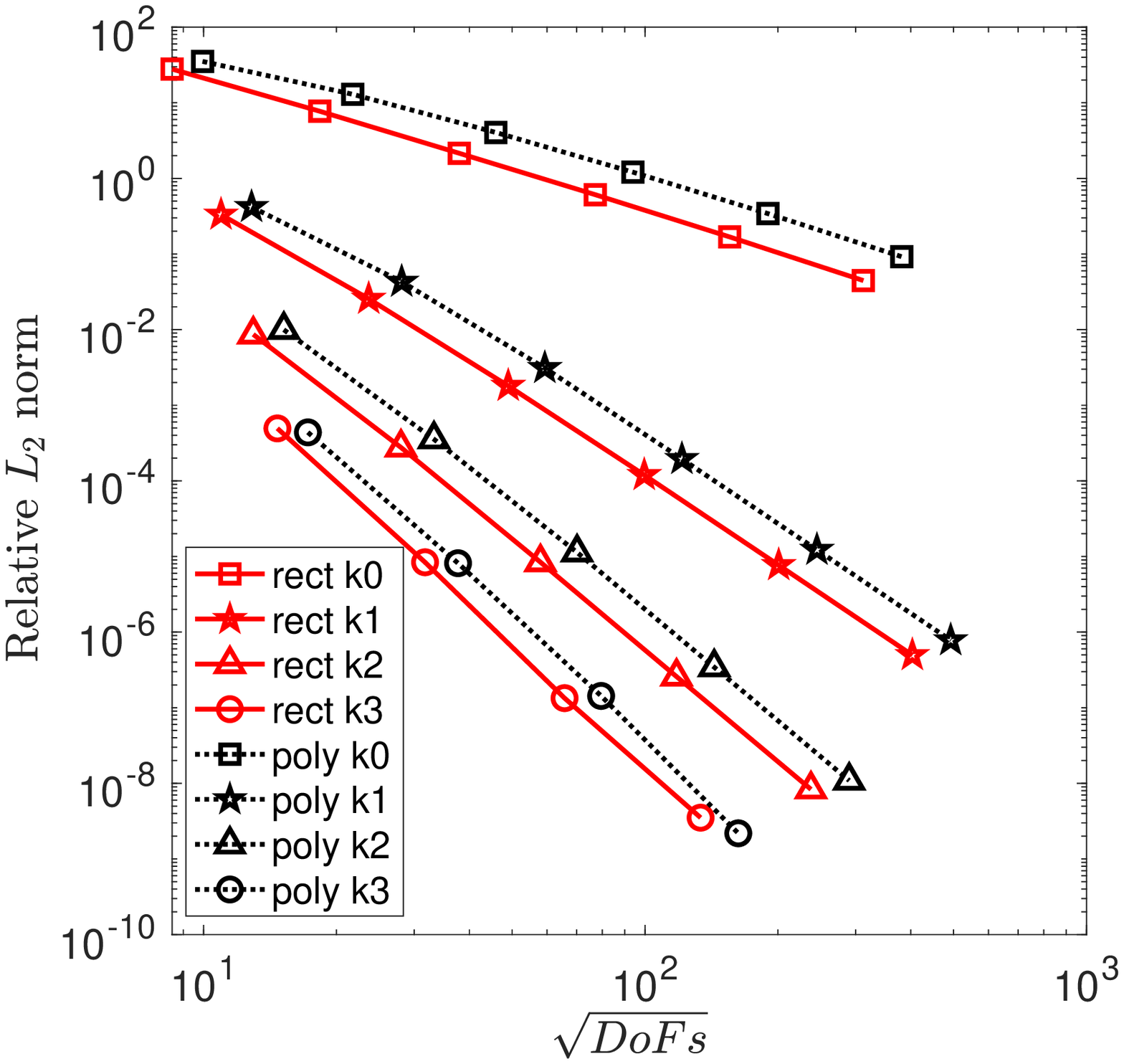}
\caption{
\label{ex1:h-refine}
Convergence of HHO(A) method in $H^2$- and $L^2$-(semi)norms on polygonal and rectangular meshes.}
\end{figure}

Let us first verify the convergence rates obtained with the HHO(A) method
with $k\in\{0,1,2,3\}$. We consider a sequence of successively refined rectangular meshes
and a sequence of successively refined polygonal (Voronoi-like) meshes
(generated through the PolyMesher Matlab library \cite{polymesher}).
Two examples of polygonal meshes are shown in Figure \ref{ex1:mesh_figure}
(in general, the cells do not contain more than $8$ edges).
We measure relative errors in the (broken) $H^2$-seminorm and in the $L^2$-norm,
both quantities being evaluated using the reconstruction of the HHO solution cellwise.
The errors are reported in Figure \ref{ex1:h-refine} as a function of
$\mathrm{DoFs}^{1/2}$, where $\mathrm{DoFs}$ denotes the total number of globally coupled
discrete unknowns (that is, the face unknowns). We observe that the $H^2$-error
converges at the optimal rate $O(h^{k+1})$, as predicted in Theorem \ref{Theorem: main}.
The $L^2$-error converges at the optimal rate $O(h^{k+3})$, except for $k=0$ where the rate
is only $O(h^2)$; all these rates are consistent with what can be expected
from a duality argument (not detailed herein for brevity; see \cite{BoDPGK:18,MuWaY:14}
for examples of this argument for HHO and WG methods).

\begin{figure}[!h]
\centering
\includegraphics[scale=0.3]{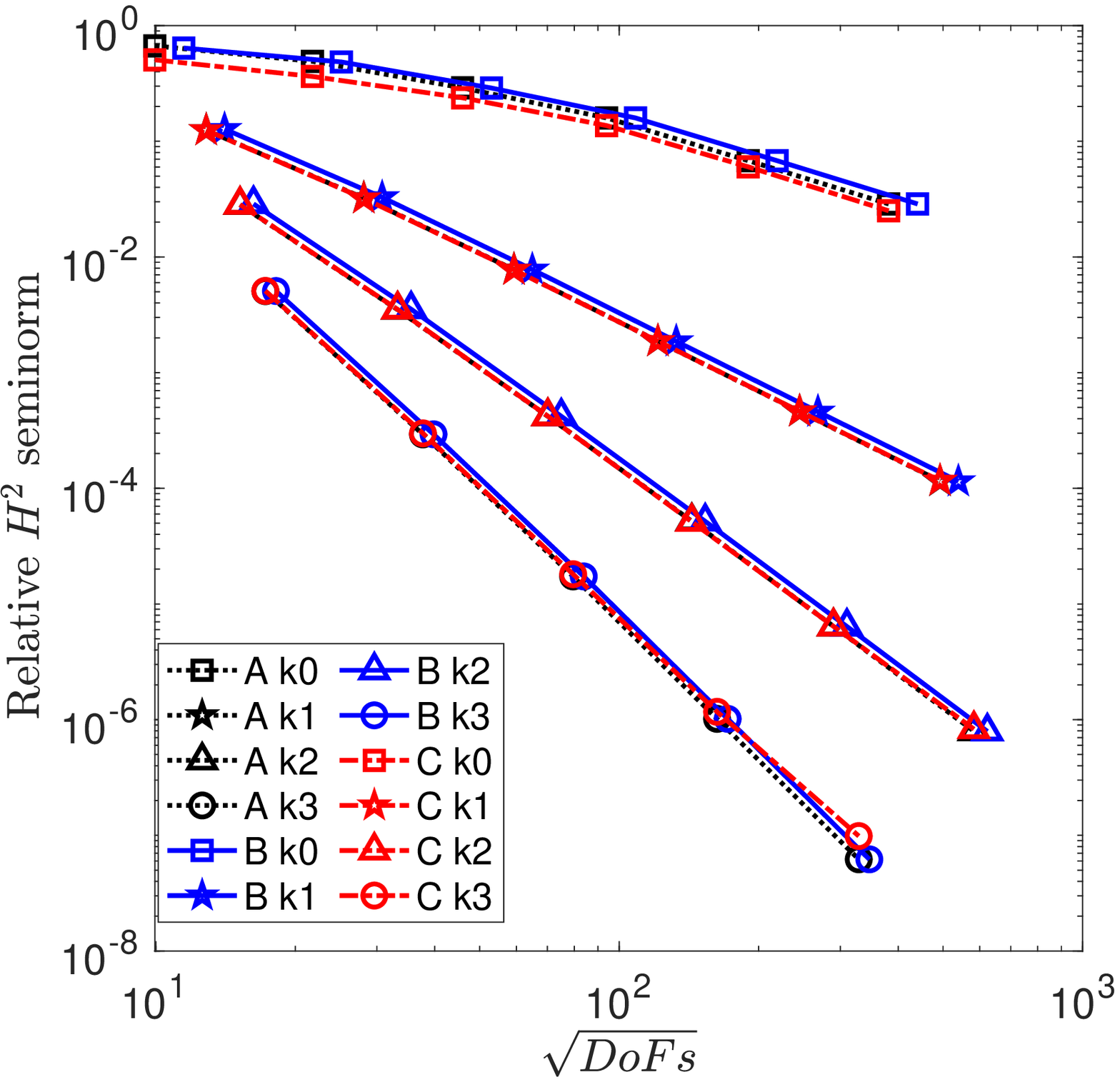}
\hspace{2cm}
\includegraphics[scale=0.3]{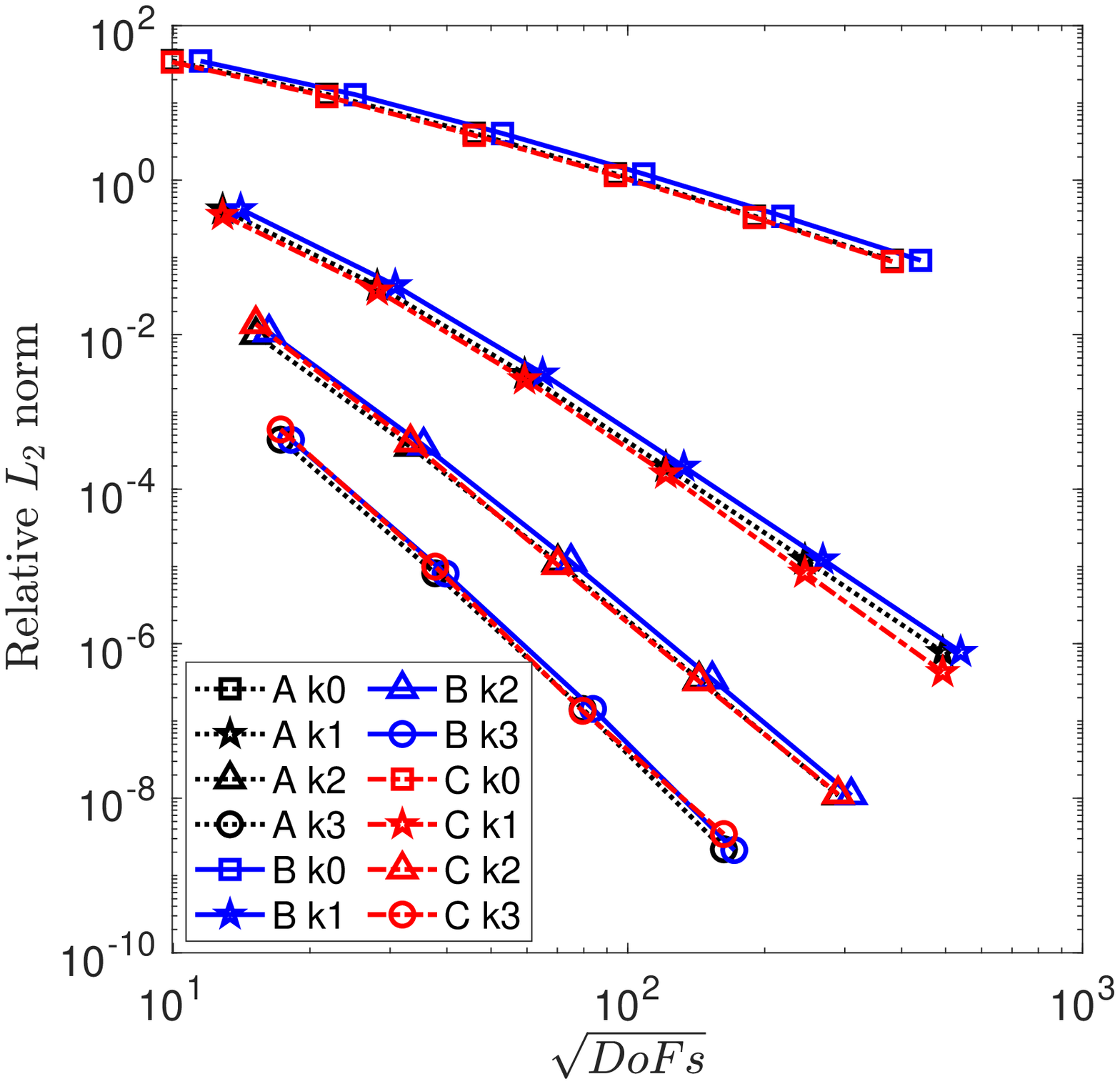}
\caption{
\label{ex1:Comparison error}
Convergence of HHO(A,B,C) methods in $H^2$- and $L^2$-(semi)norms on polygonal meshes. }
\end{figure}

\begin{figure}[!htb]
\centering
\includegraphics[scale=0.35]{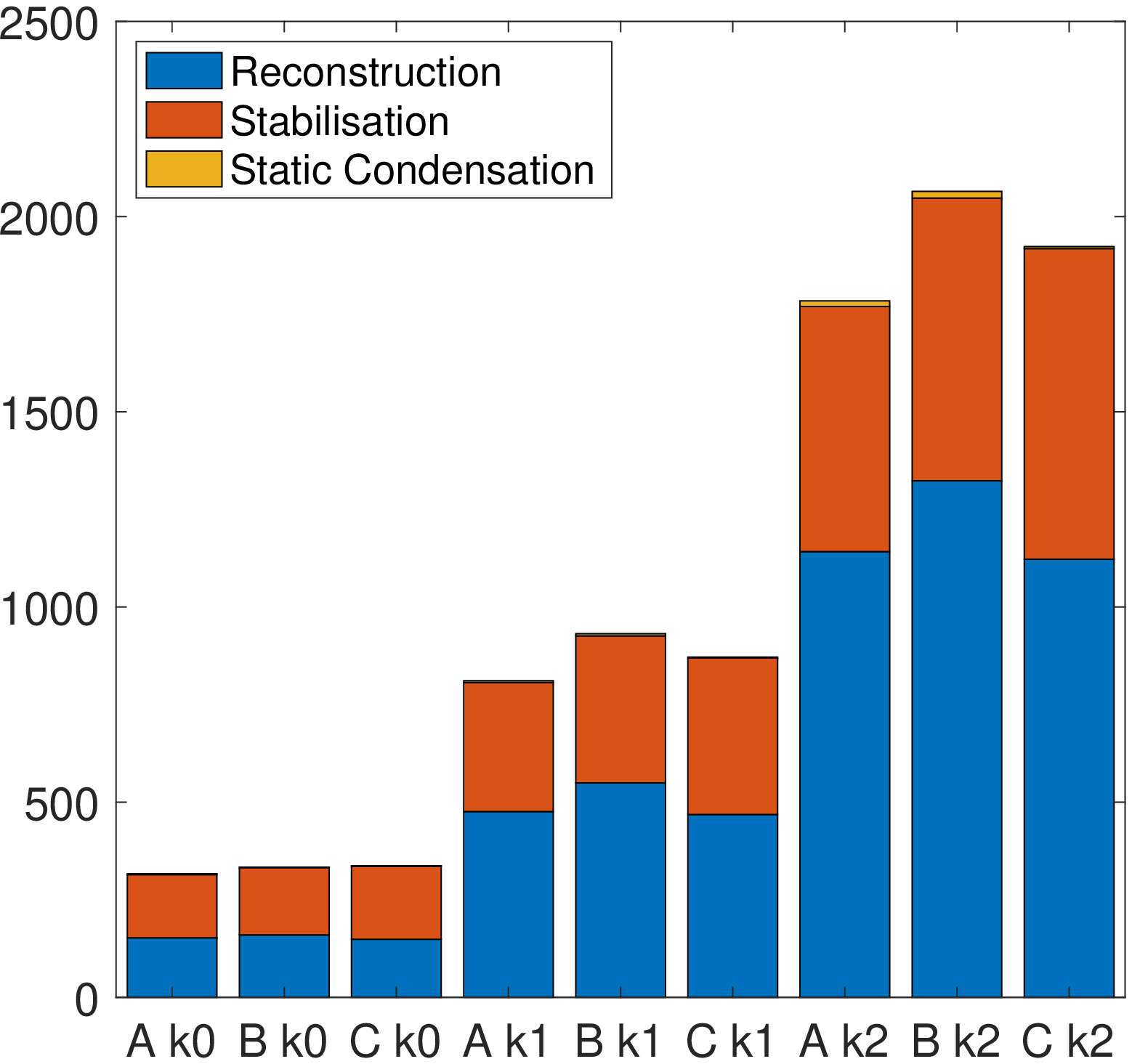}
\hspace{1cm}
\includegraphics[scale=0.35]{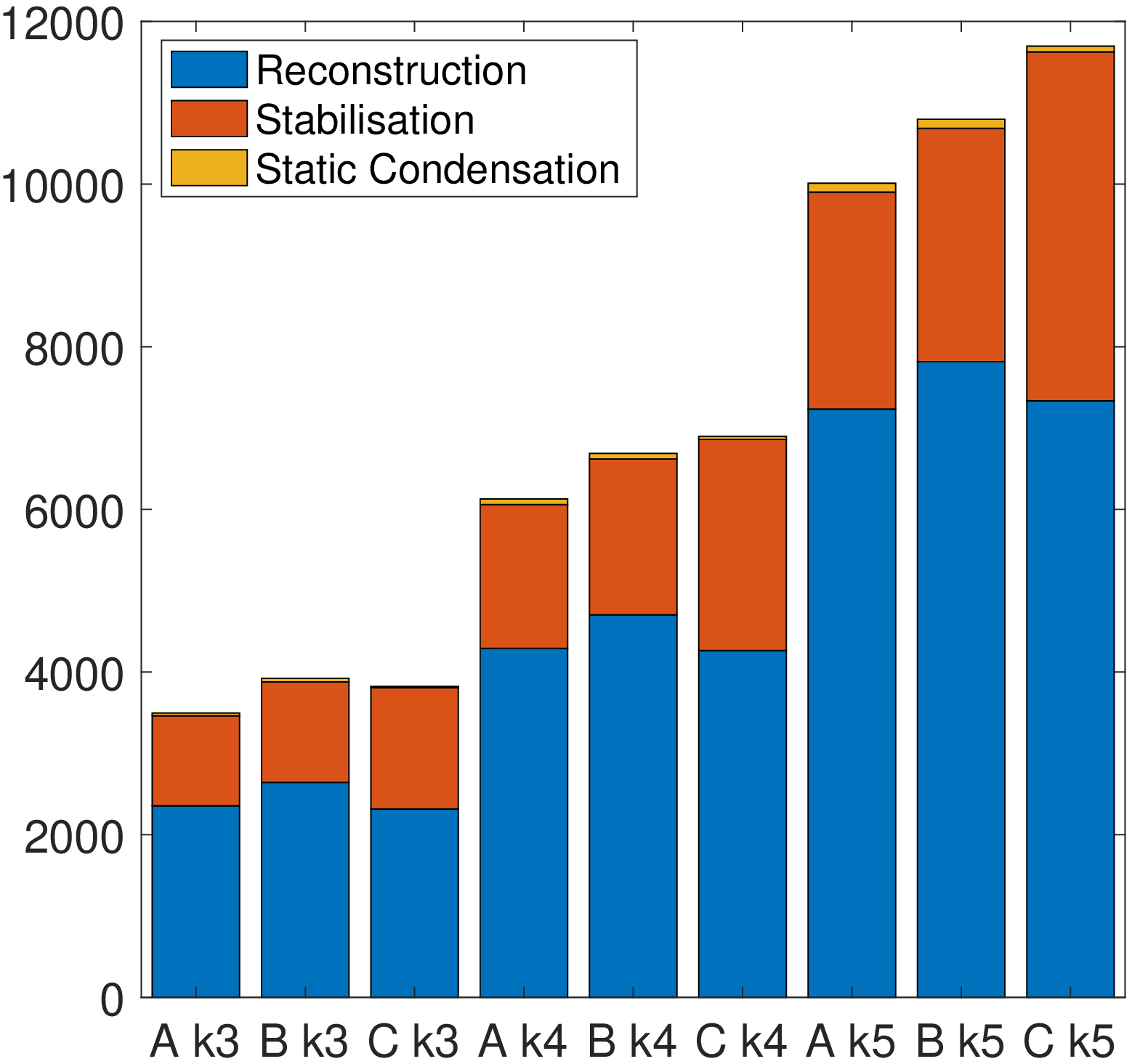}
\caption{
\label{ex1:Comparison time}
Comparison of computational times (in seconds) spent in reconstruction, stabilization, and static condensation for the three HHO(A,B,C) methods on a polygonal mesh with $16{,}384$ cells and polynomial degrees $k\in\{0,\ldots,5\}$.}
\end{figure}

Let us now compare the three HHO(A,B,C) methods. The same relative errors as in
Figure~\ref{ex1:h-refine} are reported in Figure \ref{ex1:Comparison error}.
The results show that the three HHO methods converge with the same rates,
and that the accuracy reached on a given mesh with a given polynomial degree is quite
close for the three methods. We mention that the three methods
are sensitive to conditioning issues
that arise for high polynomial degree when the error is already quite
low (typically below $10^{-8}$ in the $H^2$-seminorm),
and the HHO(C) method is somewhat more sensitive.
It is instructive to have a closer look at how the computational
costs related to the assembling of the system matrix
are spent between the tasks of reconstruction, stabilization, and
static condensation. The results are reported in Figure~\ref{ex1:Comparison time}
on a polygonal mesh with $16{,}384$ cells (and $49{,}014$ edges)
and polynomial degrees $k\in\{0,\ldots,5\}$.
Quite importantly, the local reconstruction operator is computed based on equation
\eqref{eq:rec_ipp}. Indeed, using \eqref{reconstruction} instead results in a more
intricate assembling of the right-hand side, increasing by a factor ranging
from 2.5 (for $k=0$) to 3.5 (for $k=5$) the time spent in reconstruction.
Figure \ref{ex1:Comparison time} shows that the time spent on
static condensation is always marginal. Moreover, we can see that the somewhat more
elaborate design of the stabilization in the HHO(C) method is reflected by a somewhat
larger computational cost than for the HHO(A,B) methods. The (perhaps a bit unexpected)
consequence is that the HHO(A,B) methods require altogether less assembly time than the
HHO(C) method although their number of discrete unknowns is larger.
Finally, we notice that the
reconstruction time is always larger than the stabilization time, and this trend gets
more pronounced for larger $k$. To sum up, the most computationally effective
method based on these results is HHO(A). In what follows, we only consider this method
and simply call it ``HHO'' method.

\subsection{Comparison with DG, $C^0$-IPDG, and FEM}

In this section, we compare the computational performance of HHO with
the fully nonconforming dG method on polygonal and simpler meshes,
and with the $C^0$-IPDG, Morley, and HCT methods on triangular meshes.

\begin{table}[!htb]
\begin{center}
\begin{tabular}{|c|c|c|c||c|c|c|c|}
\hline
\multicolumn{4}{|c||}{HHO, triangular mesh}
&
\multicolumn{4}{|c|}{dG, triangular mesh}  \\
\hline
\hline
order  & $\#$  DoFs  & assembling &  solving & order & $\#$  DoFs  & assembling  &  solving   \\
\hline
$k=0$ & $146{,}688$  & $275.5$ & $7.0$  &  $\ell=2$ & $196{,}608$ & $472.0$ & $16.0$  \\
\hline
$k=1$ & $244{,}480$  & $882.0 $ & $18.3$ & $\ell=3$& $327{,}680$ & $1300.9$ & $41.1$  \\
\hline
$k=2$ &  $342{,}272$ & $2076.3$ & $33.2$  & $\ell=4$ & $491{,}520$  & $2965.4$ &$96.0$ \\
\hline
$k=3$ & $440{,}064$ & $4062.0$ & $53.0$ & $\ell=5$ & $  688{,}128$ & $5940.6$ & $195.1$ \\
\hline
\end{tabular}
\\[.2cm]
\begin{tabular}{|c|c|c|c||c|c|c|c|}
\hline
\multicolumn{4}{|c||}{HHO, polygonal  mesh}
&
\multicolumn{4}{|c|}{dG,  polygonal  mesh}  \\
\hline
\hline
order  & $\#$  DoFs  & assembling &  solving & order & $\#$  DoFs  & assembling  &  solving   \\
\hline
$k=0$ & $145{,}554$   & $251.1$ & $17.3$ &  $\ell=2$ & $98{,}304$   &$420.5$ &$12.7$ \\
\hline
$k=1$ & $242{,}590$ & $770.2$ &$44.7$ & $\ell=3$& $163{,}840$ & $1160.9$ &  $33.1$ \\
\hline
$k=2$ &  $339{,}626$  & $1784.3$ & $86.9$  & $\ell=4$ & $245{,}760$   & $2647.3$ & $78.4$ \\
\hline
$k=3$ & $436{,}662$ & $3496.7$ & $149.9$ & $\ell=5$ & $ 344{,}064$ & $5304.7$ & $155.8$ \\
\hline
\end{tabular}
\end{center}
\caption{Comparison of total DoFs, assembling time, and solving time for the HHO and dG methods. The polynomial degree is chosen so that both methods deliver the same decay rates on the $H^2$-error. Upper table: triangular mesh composed of $32{,}768$ cells; lower table:
polygonal mesh composed of $16{,}384$ cells.}
\label{ex1:table Comparison for Time DG VS HHO}
\end{table}

Let us consider first the dG method. To put HHO and dG on a fair comparison
basis, we compare the HHO method with face polynomial degree $k\ge0$ to the dG method
with cell polynomial degree $\ell=k+2$, so that both methods deliver the same decay rates
on the $H^2$-error. A comparison of total DoFs, assembling time (including
static condensation if applicable), and solving time
for both methods is provided in Table~\ref{ex1:table Comparison for Time DG VS HHO}.
We consider a triangular mesh and a polygonal mesh (with $32{,}768$ and
$16{,}384$ cells, respectively). The first observation is that HHO always
leads to less DoFs, and to smaller times spent on assembling. The main reason
is that the HHO DoFs are attached to the mesh faces rather than the mesh cells.
Although there are more faces than cells in a given mesh (the more so as the cells
are polygons with many faces), the polynomial spaces in cells are richer than those
on faces. Moreover, the degree of the cell polynomials in the dG method is larger than
the degree of the face polynomials in the HHO method ($(k+2)$ vs.~$\{k,k+1\}$).
Another reason for the lower assembling times with HHO is that the evaluation of
numerical fluxes in dG methods actually leads to a more expensive evaluation of
face-related quantities.
The conclusions are, however, slightly different if one considers the solving time
(since the assembling stage can be fully parallelized, the solving time
becomes dominant in highly parallel architectures). The results
in Table~\ref{ex1:table Comparison for Time DG VS HHO} show that
on triangular meshes (where cells have a moderate number of faces), the
solving time for HHO is always smaller than that for dG. Instead, on
polygonal (Voronoi-like) meshes, the solving time for dG is smaller for
low polynomial degrees (up to $2$), whereas the solving time for HHO becomes again
smaller for higher polynomial degrees. The observation on polygonal meshes and low
polynomial degrees indicates that although the stencil of HHO methods is quite compact, it
is still less compact than that of dG methods. In particular, all the discrete
unknowns attached to the faces sharing a given mesh cell are coupled.
Figure~\ref{ex1:Comparison HHO vs DG in time} provides a more thorough
viewpoint on the above results by highlighting the relative efficiency of both
methods measured as the time needed to reach a certain error threshold in
the $H^2$-seminorm. The time is either the assembling time (which is more representative
of a serial implementation) or the solving time (which is more representative
of a parallel implementation). We can see that on triangular and rectangular meshes,
for all polynomial orders, the HHO method reaches an error threshold with less
assembling or solving time than the dG method. The same conclusion is reached
on polygonal meshes for the polynomial degree $k=3$ and both times as well as for
$k\in\{0,1,2\}$ and assembling time, whereas for $k\in\{0,1,2\}$ and solving time,
the efficiency of both methods is comparable.

\begin{figure}[htb]
\centering
\includegraphics[scale=0.3]{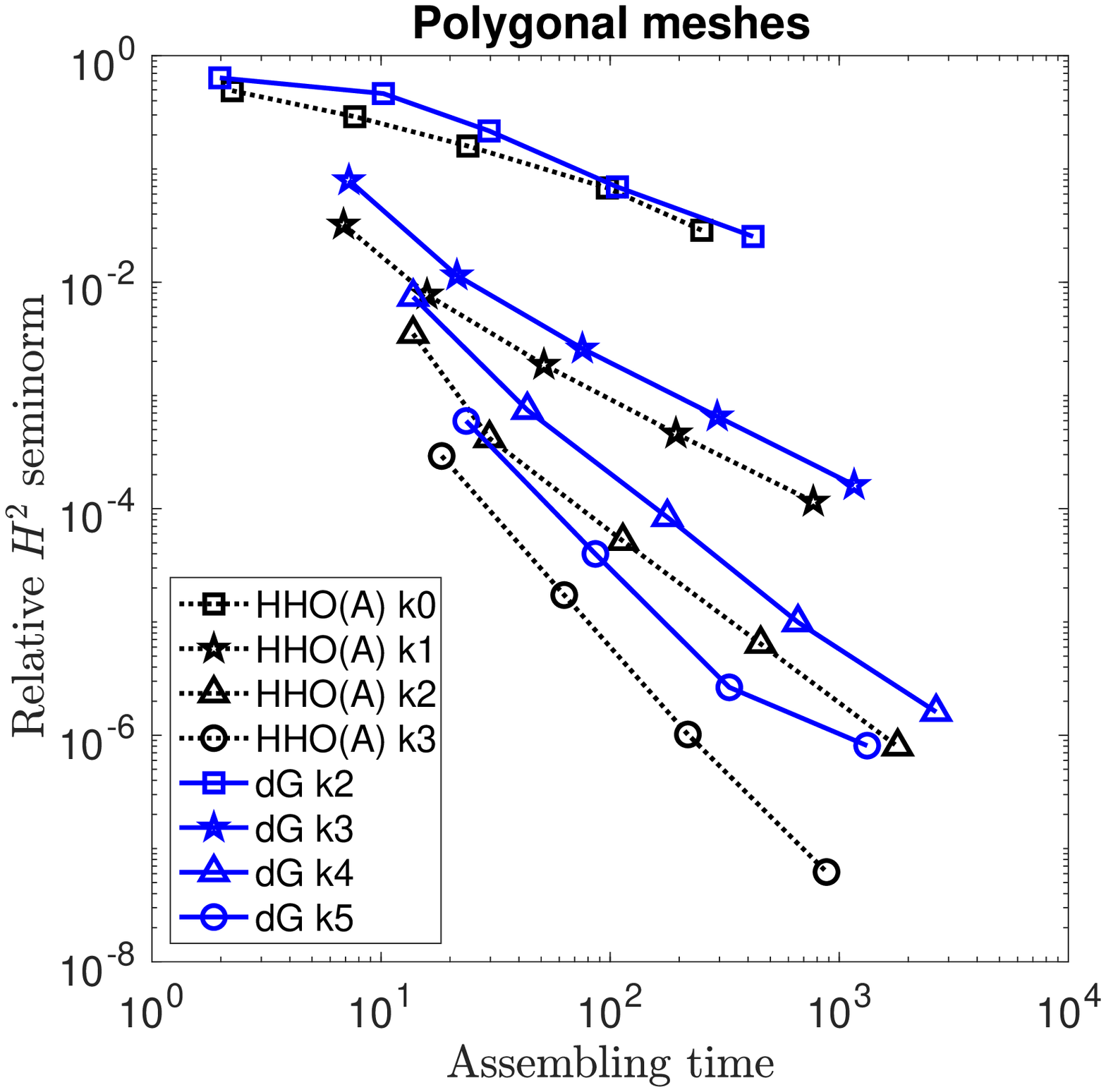}
\includegraphics[scale=0.3]{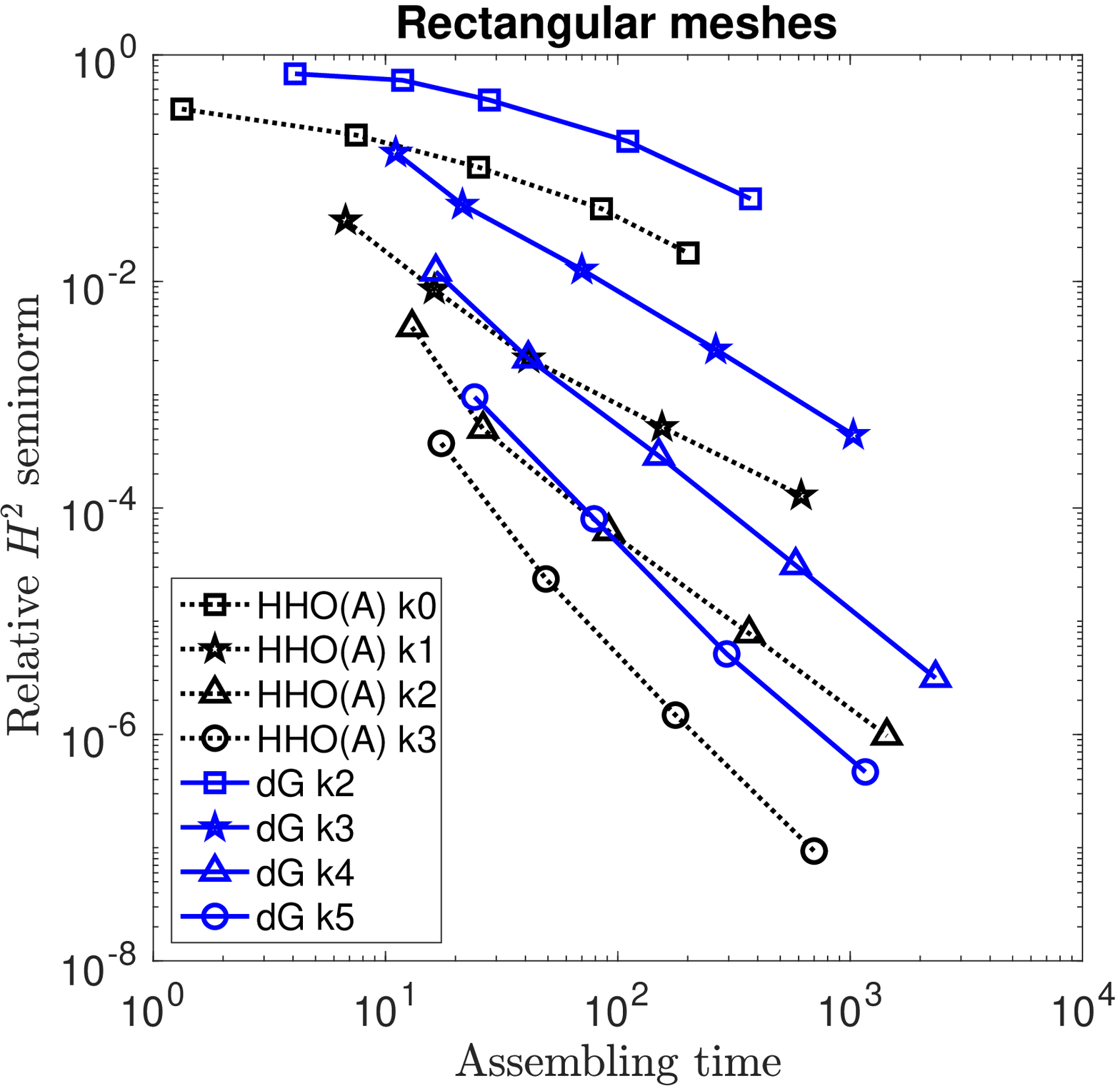}
\includegraphics[scale=0.3]{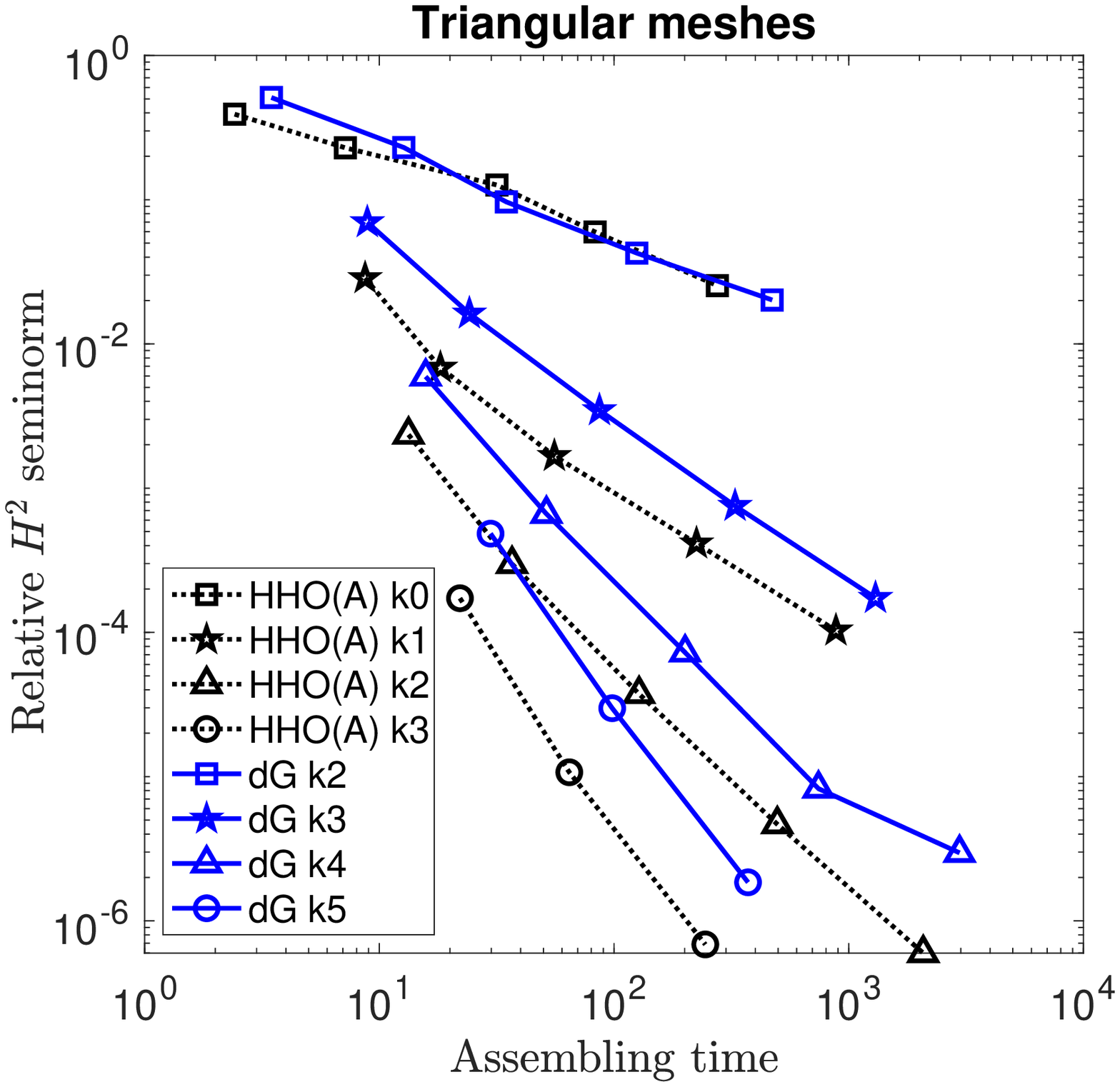} \\
\includegraphics[scale=0.3]{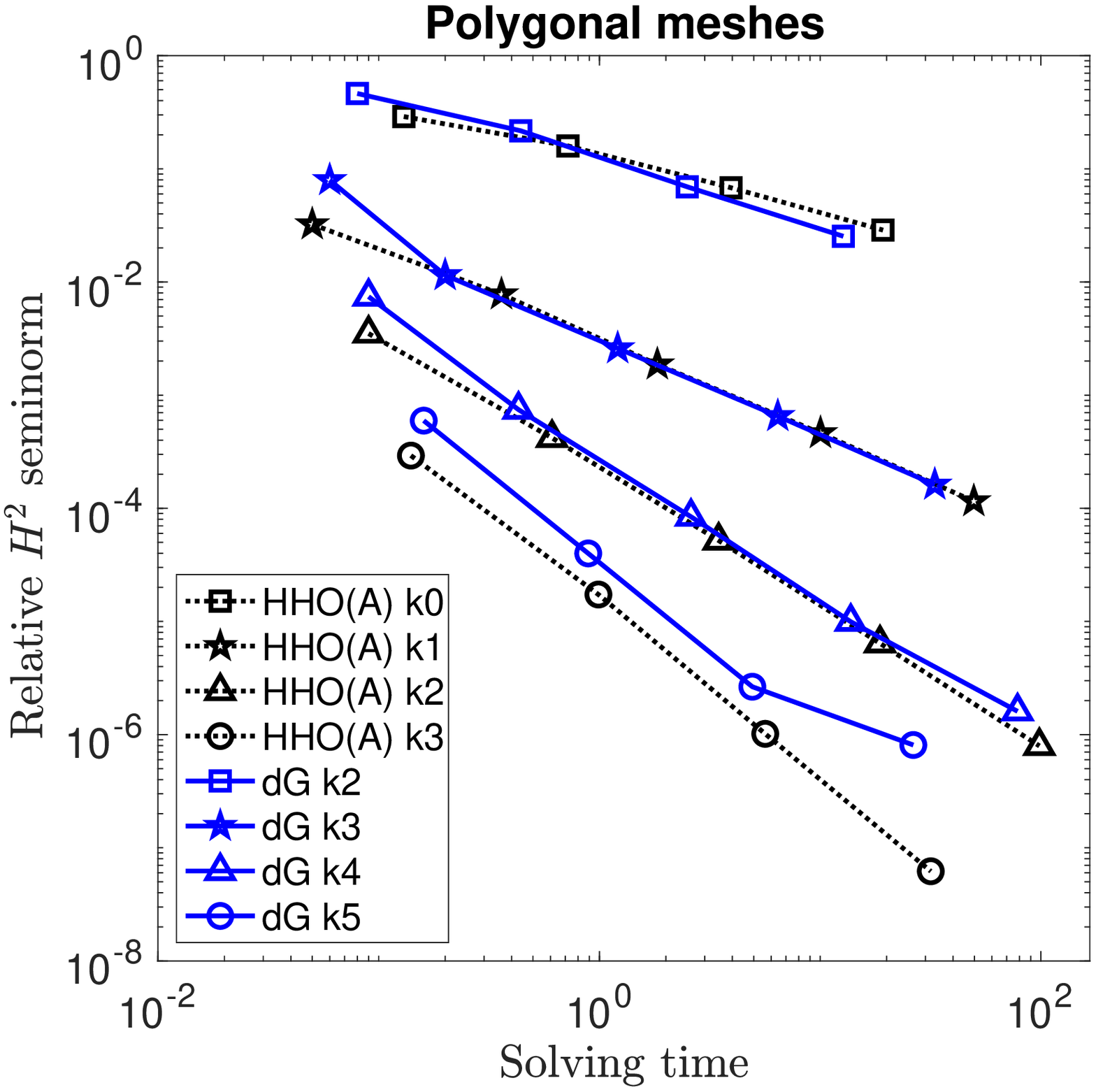}
\includegraphics[scale=0.3]{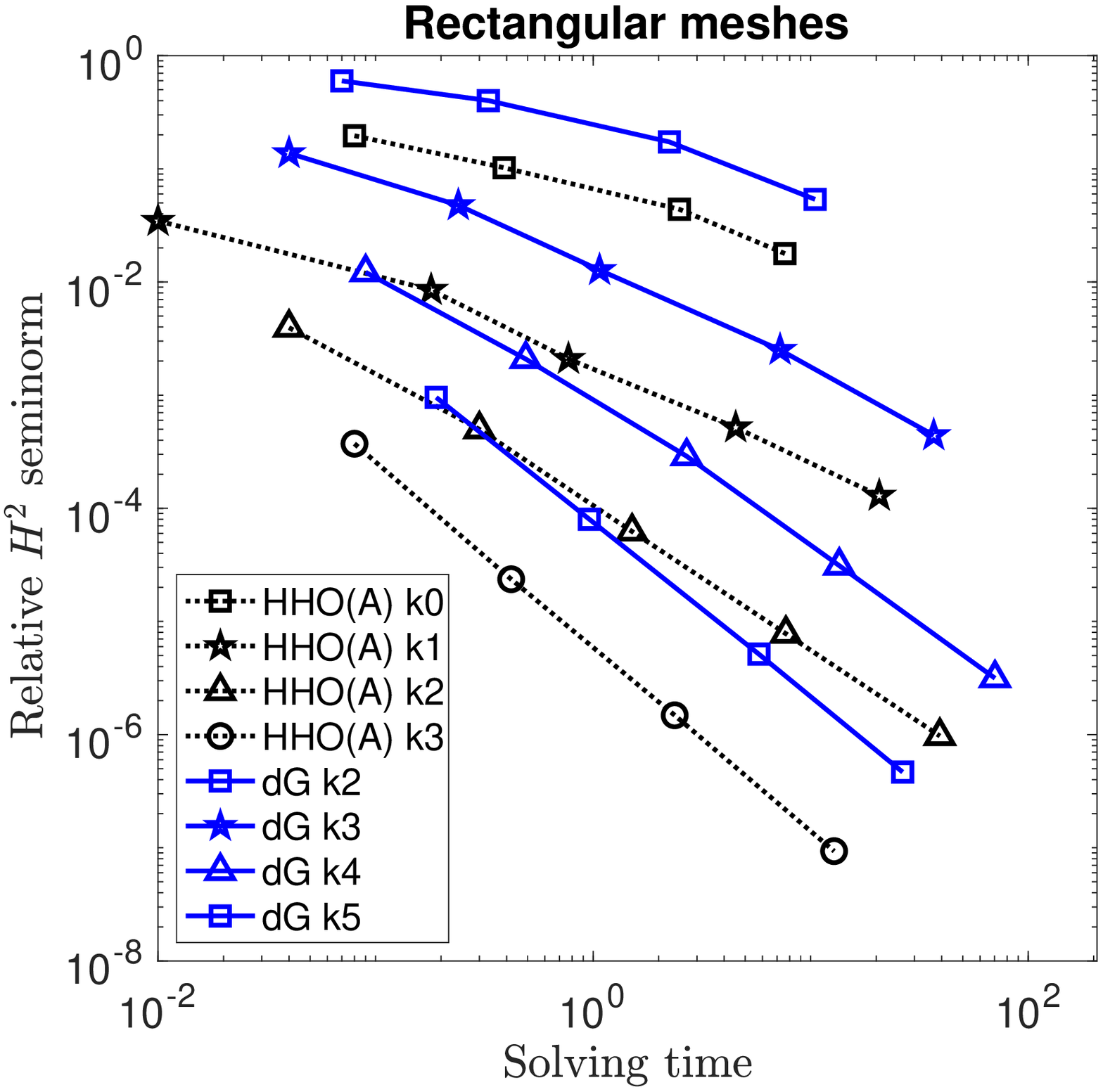}
\includegraphics[scale=0.3]{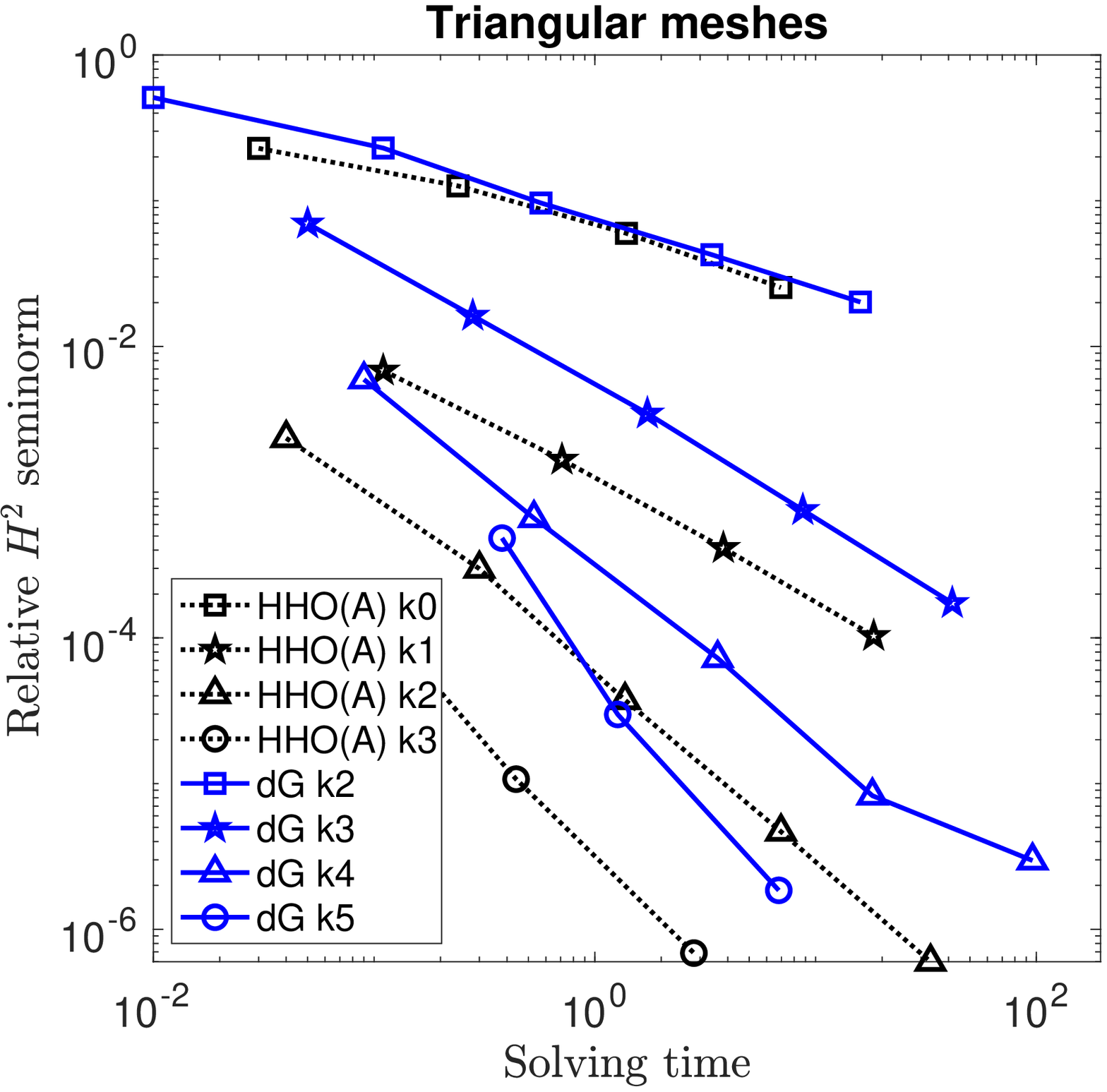}
\caption{
\label{ex1:Comparison HHO vs DG in time}
Comparison of HHO and dG methods: relative $H^2$-seminorm error
as a function of assembling time (upper row) and solving time (lower row)
on a sequence of polygonal (left), rectangular (center), and triangular (right) meshes
and polynomial order $k\in\{0,1,2,3\}$ for HHO and $\ell=k+2$ for dG.}
\end{figure}

\begin{table}[htb]
\begin{center}
\begin{tabular}{||c|c|c|c||c|c|c|c||}
\hline
$k=0$& $\#$  DoFs & assembling &  solving
& $k=1$ & $\#$  DoFs & assembling &  solving \\
\hline
Morley & $65{,}025$   & 22.9 &4.3& HCT & 97{,}283 &169.3&19.8 \\
\hline
HHO & $146{,}688$   &275.5&7.0&  HHO & 244{,}480 &882.0 &18.3\\
\hline
$C^0$-IPDG& $65{,}025$   &369.5&9.3&  $C^0$-IPDG& 130{,}560 &1318.8& 27.0	\\
%	\hline
%DG& $196608$   &316.77&17.08&  DG &$327680$ &316.77& 41.6\	\\
\hline
\end{tabular}
\end{center}
\caption{Comparison of total DoFs, assembling time, and solving time for the HHO, $C^0$-IPDG, Morley, and HCT methods. The polynomial degree is chosen so that all the methods in the same column deliver the same decay rates on the $H^2$-error.
Triangular mesh composed of $32{,}768$ cells, $49{,}408$ edges, and $16{,}641$ vertices.}
\label{ex2:Comparison for DoFs}
\end{table}

Let us now compare the efficiency of the HHO method to the $C^0$-IPDG, Morley, and HCT
methods on a sequence of successively refined triangulations with
$32$, $128$, $512$, $2{,}048$, $8{,}192$, and $32{,}768$ cells.
As above, the comparison is made between methods delivering the same decay rates
on the $H^2$-error. This means that the HHO method with polynomial degree $k\ge0$
is compared with the $C^0$-IPDG with degree $\ell=k+2$. Moreover, the HHO($k=0$) and
the $C^0$-IPDG($\ell=2$) methods are compared with the Morley element, and
the HHO($k=1$) and the $C^0$-IPDG($\ell=3$) methods are compared with the
HCT element. Table \ref{ex2:Comparison for DoFs} reports the total number of DoFs,
the assembling time, and the solving time for all the methods on the finest triangular
mesh. We can see that in the lowest-order case, both the assembling and solving times
for the Morley element are (much) smaller than those for the HHO($k=0$)
method, which are, in turn, smaller than those for the $C^0$-IPDG($\ell=2$) method.
The conclusion for the higher-order case is the same concerning the lower times
for HHO($k=1$) with respect to $C^0$-IPDG($\ell=3$), whereas only the assembling time
for HCT is (much) smaller than that for HHO($k=1$), the solving time being
instead comparable.
One reason for this good performance of HHO compared with HCT can be that the
stencil of HCT leads to a more dense system matrix, as a result of the method
attaching DoFs to the mesh vertices.
Figure \ref{ex2:Comparison time triangular} reports the error measured in the
$H^2$-seminorm as a function of assembling and solving time, thereby providing a
comparison of the efficiency of the various methods on all the considered triangulations.
We notice that the Morley element is the most efficient among the lowest-order methods,
whereas the efficiency of the HHO method is better than that of $C^0$-IPDG,
and it is better than that of the HCT element if the solving time is considered,
whereas the conclusion is reverted if the assembling time is considered.

\begin{figure}[htb]
\centering
\includegraphics[scale=0.3]{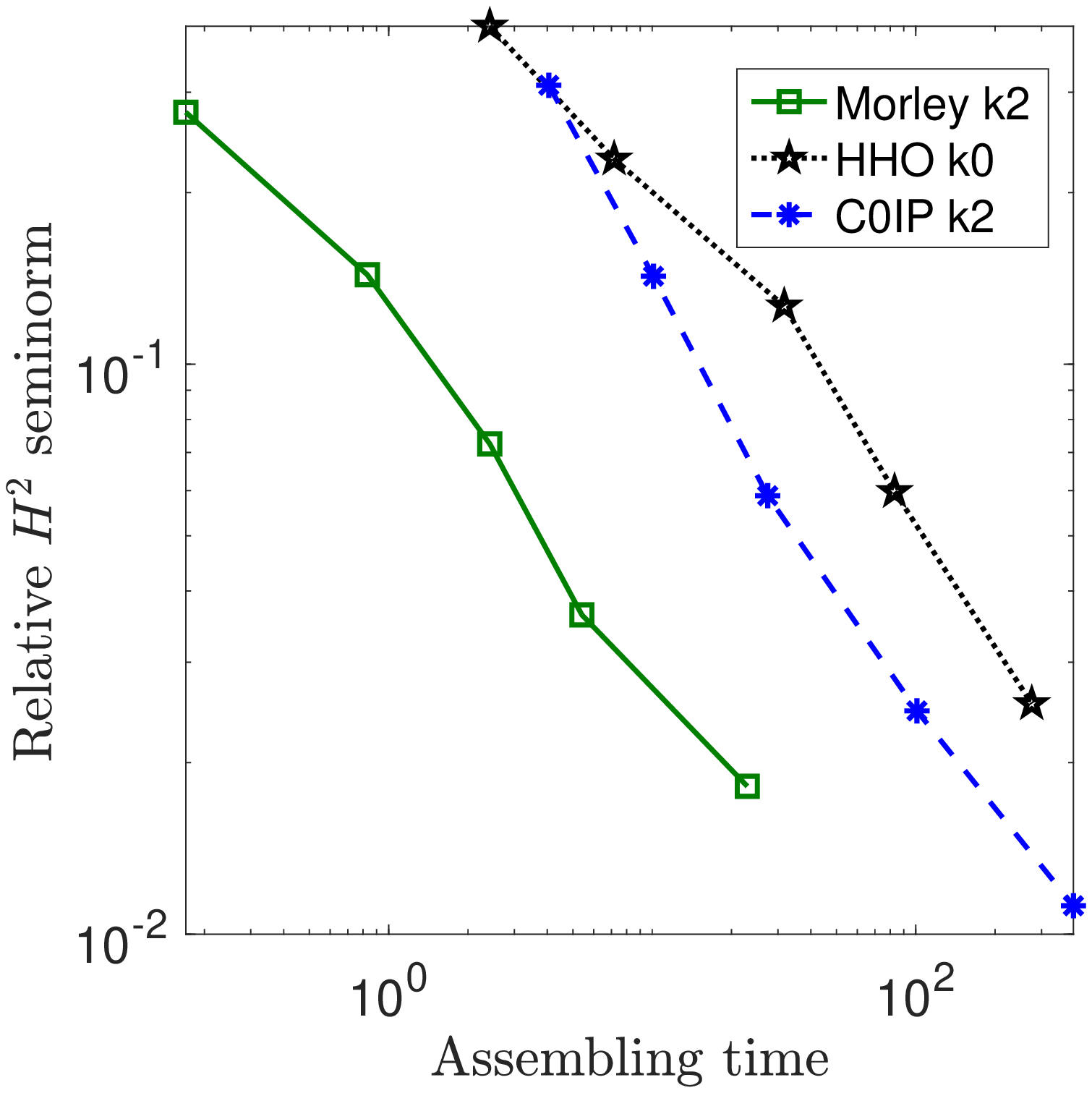}
\includegraphics[scale=0.3]{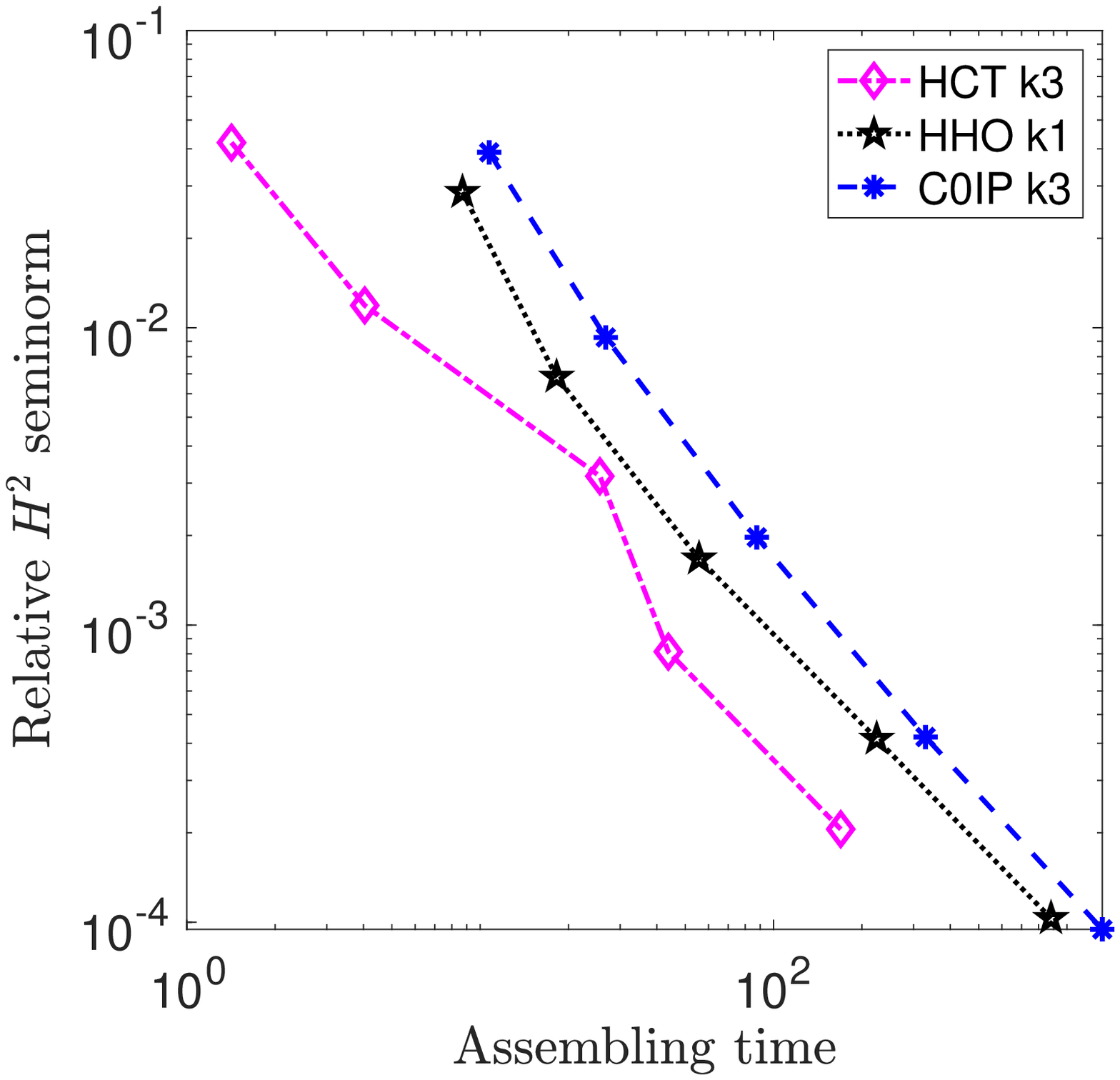}
\includegraphics[scale=0.3]{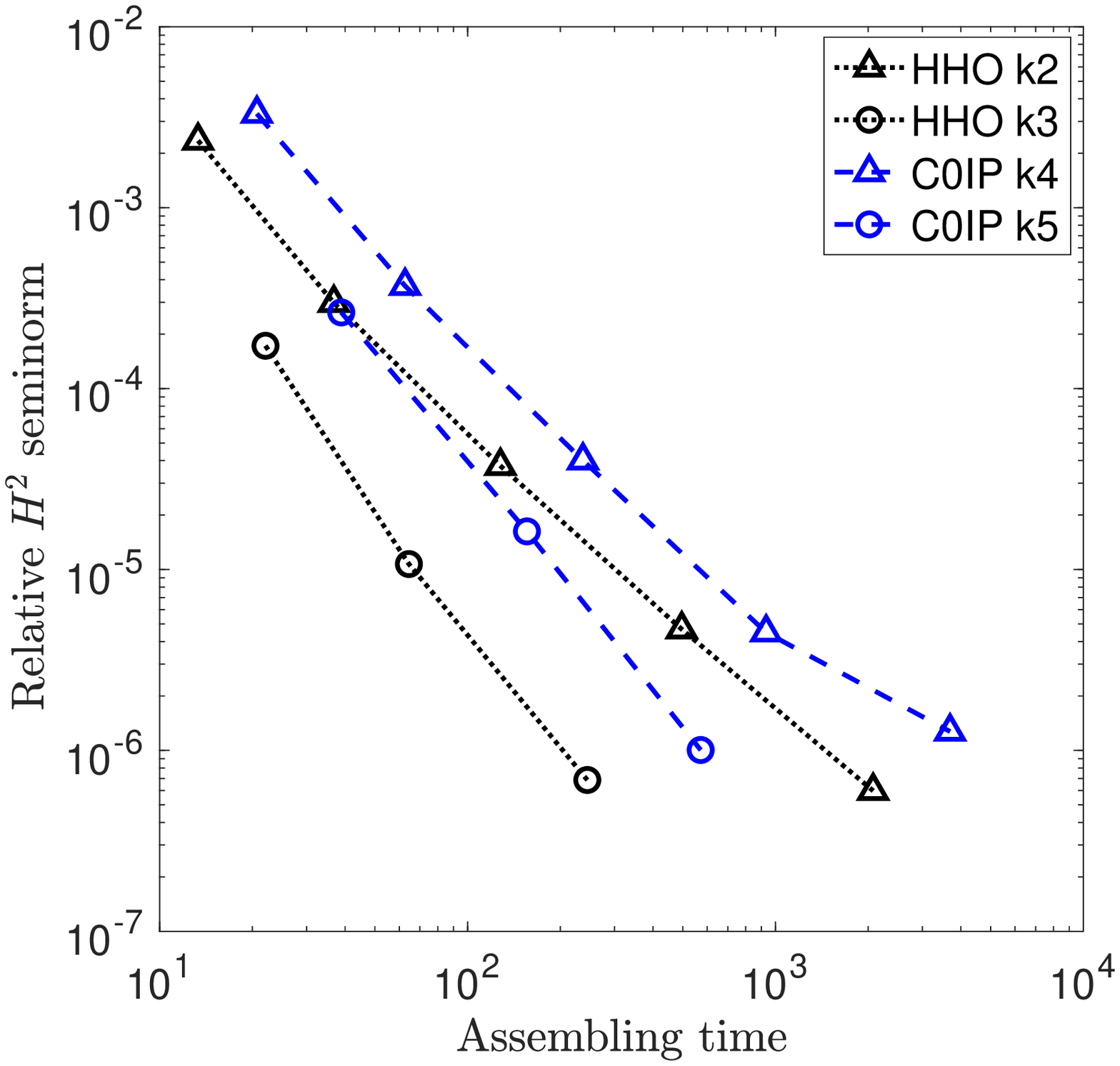} \\
\includegraphics[scale=0.3]{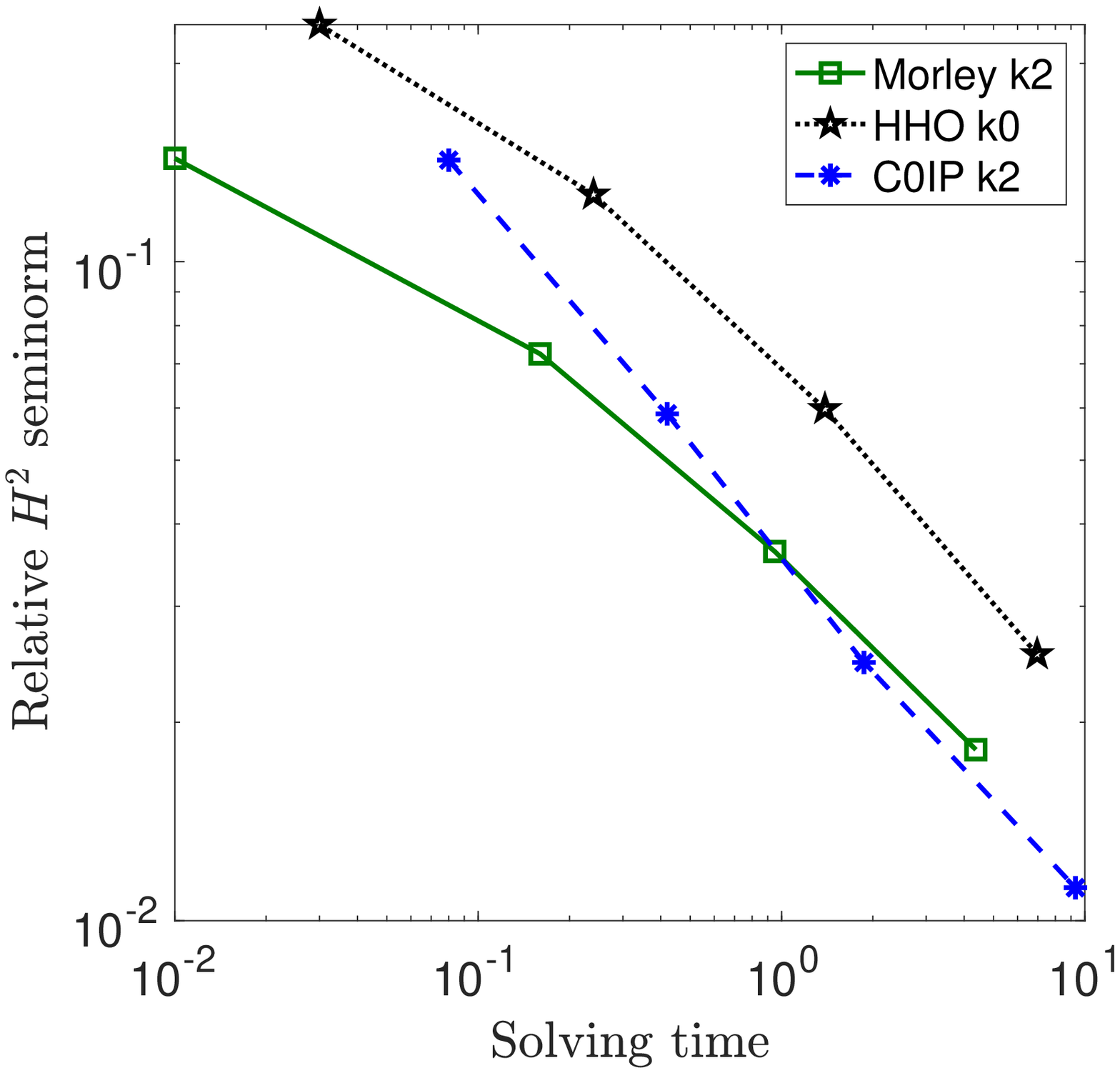}
\includegraphics[scale=0.3]{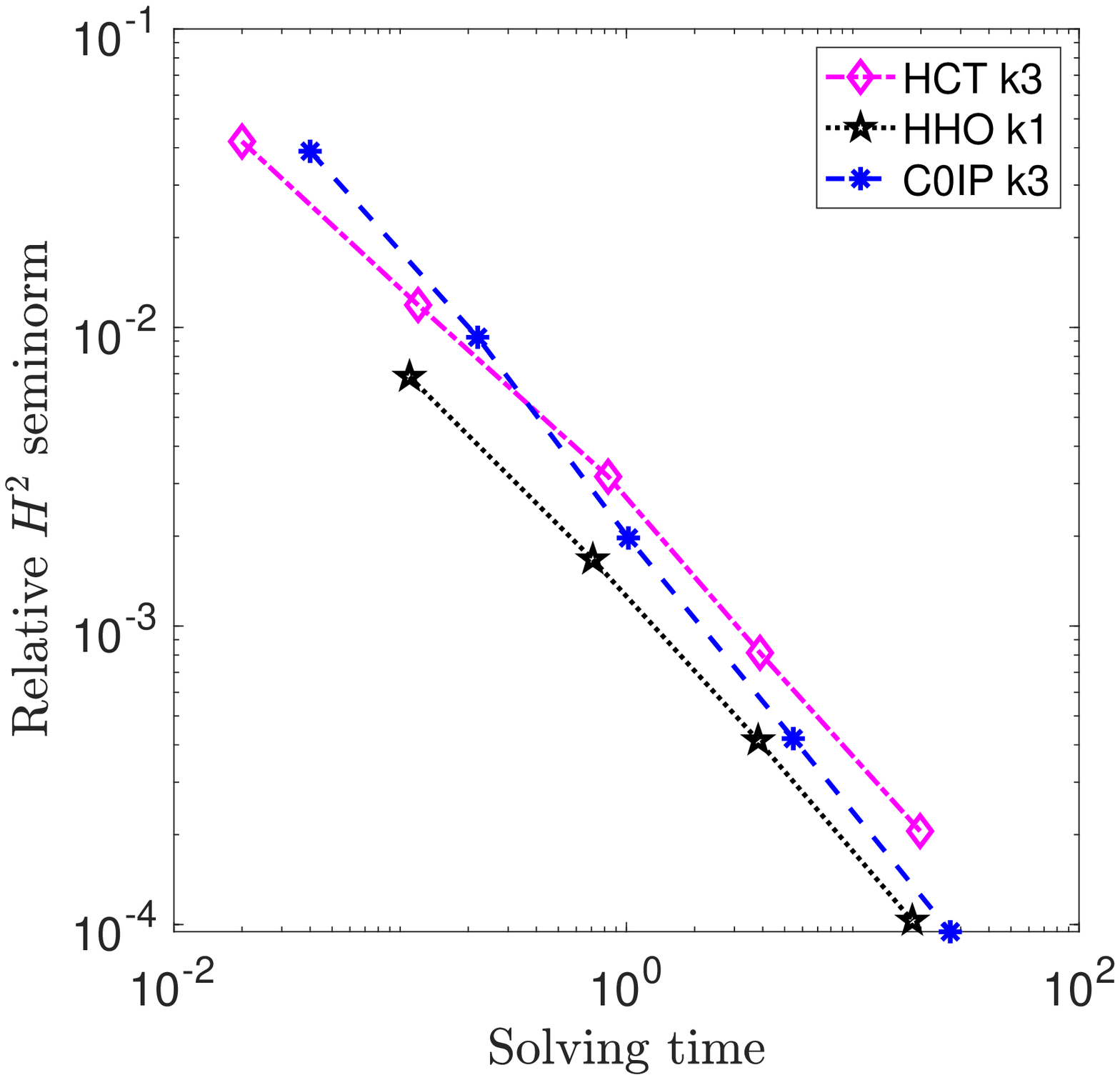}
\includegraphics[scale=0.3]{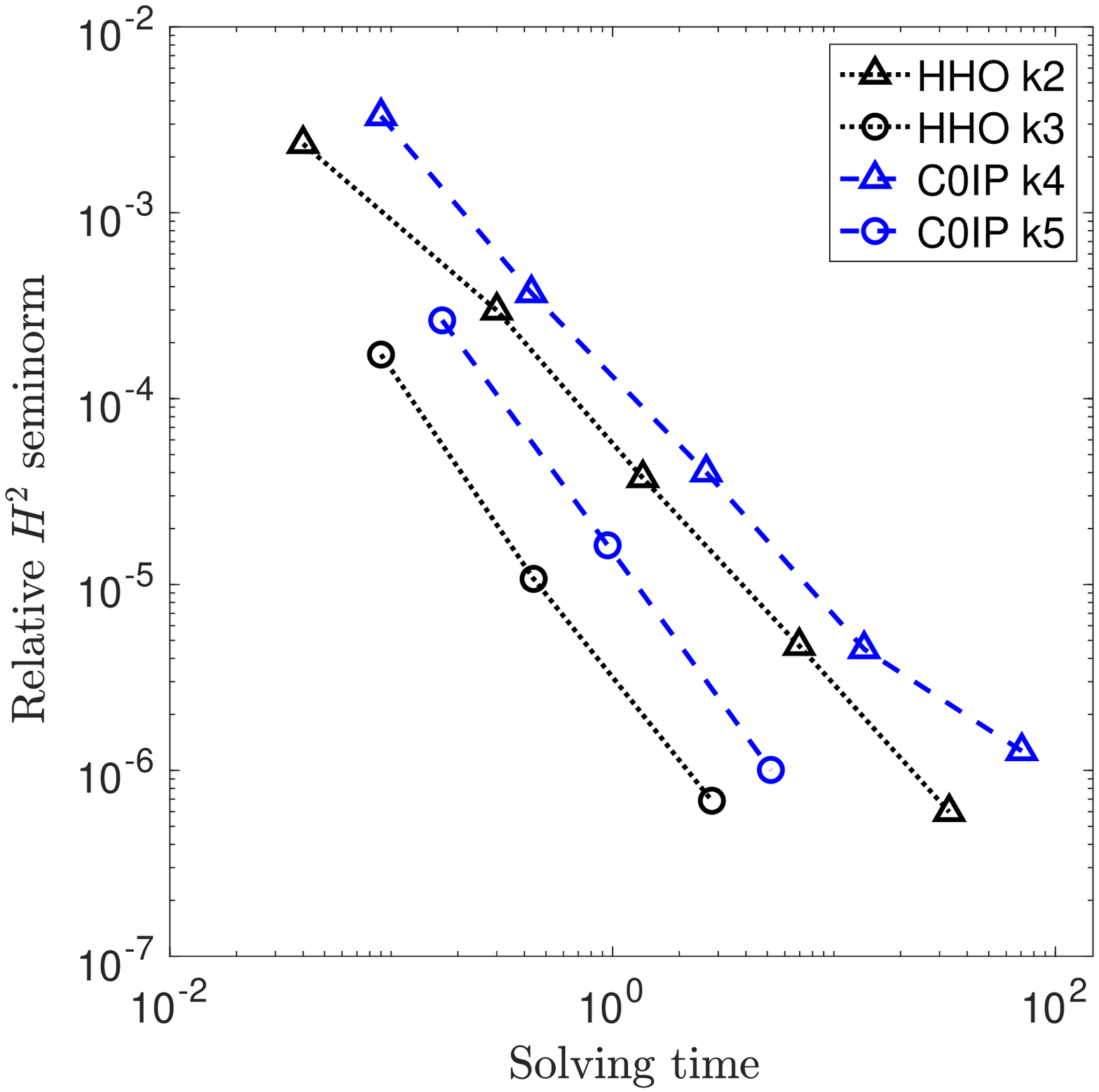}
\caption{
\label{ex2:Comparison time triangular}
Comparison of HHO, $C^0$-IPDG, Morley, and HCT methods: relative $H^2$-seminorm error
as a function of assembling time (upper row) and solving time (lower row)
on a sequence of triangular meshes.}
\end{figure}

\subsection{Tests on the HHO-Nitsche method}

To conclude, let us briefly illustrate that the proposed HHO-Nitsche (HHO-N) method with a
weak enforcement of the boundary conditions performs as
well as the HHO method with a strong enforcement of the boundary conditions.
We select $f$ and the non-homogeneous boundary data $g_{D}$ and $g_{N}$ such that
on $\Omega:=(0,1)^2$, the exact solution is $u(x,y) = \sin(\pi x)^2  \sin(\pi y)^2  + \exp{(-(x-0.5)^2-(y-0.5)^2)}$.
We consider the same sequence of polygonal meshes and the same polynomial degrees as in
Section~\ref{sec:conv_rates}. Figure \ref{ex3:h-refine} presents the relative errors
measured in the $H^2$-seminorm and the $L^2$-norm using cellwise the reconstruction operator
 for their evaluation. We compare the HHO and HHO-N methods. \Rev{Both} methods
employ the same number of globally coupled DoFs. We can see from Figure \ref{ex3:h-refine}
that the errors produced by both methods are quite close in all cases.

\begin{figure}[!hbt]
\centering
\includegraphics[scale=0.3]{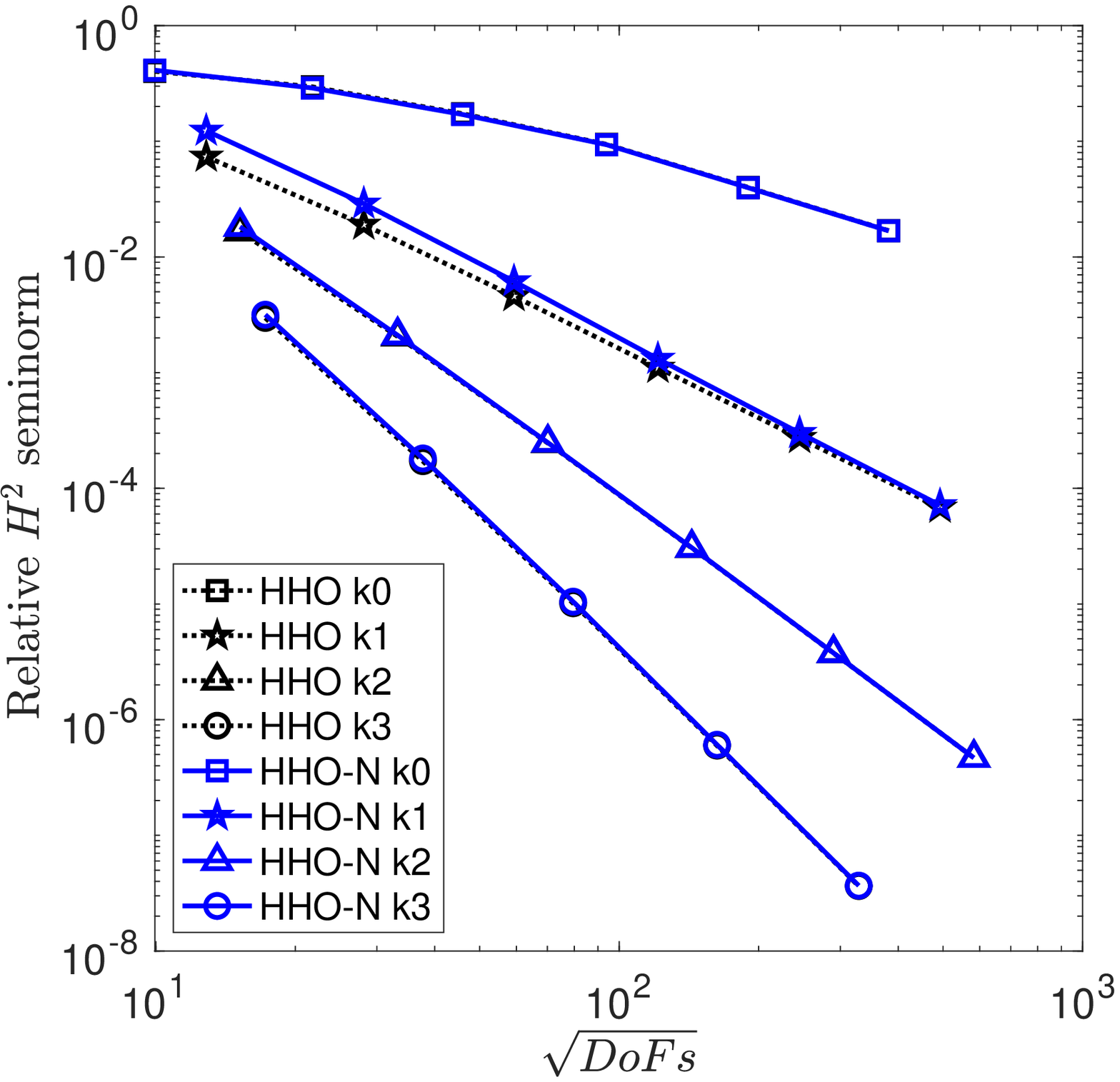}
\hspace{2cm}
\includegraphics[scale=0.3]{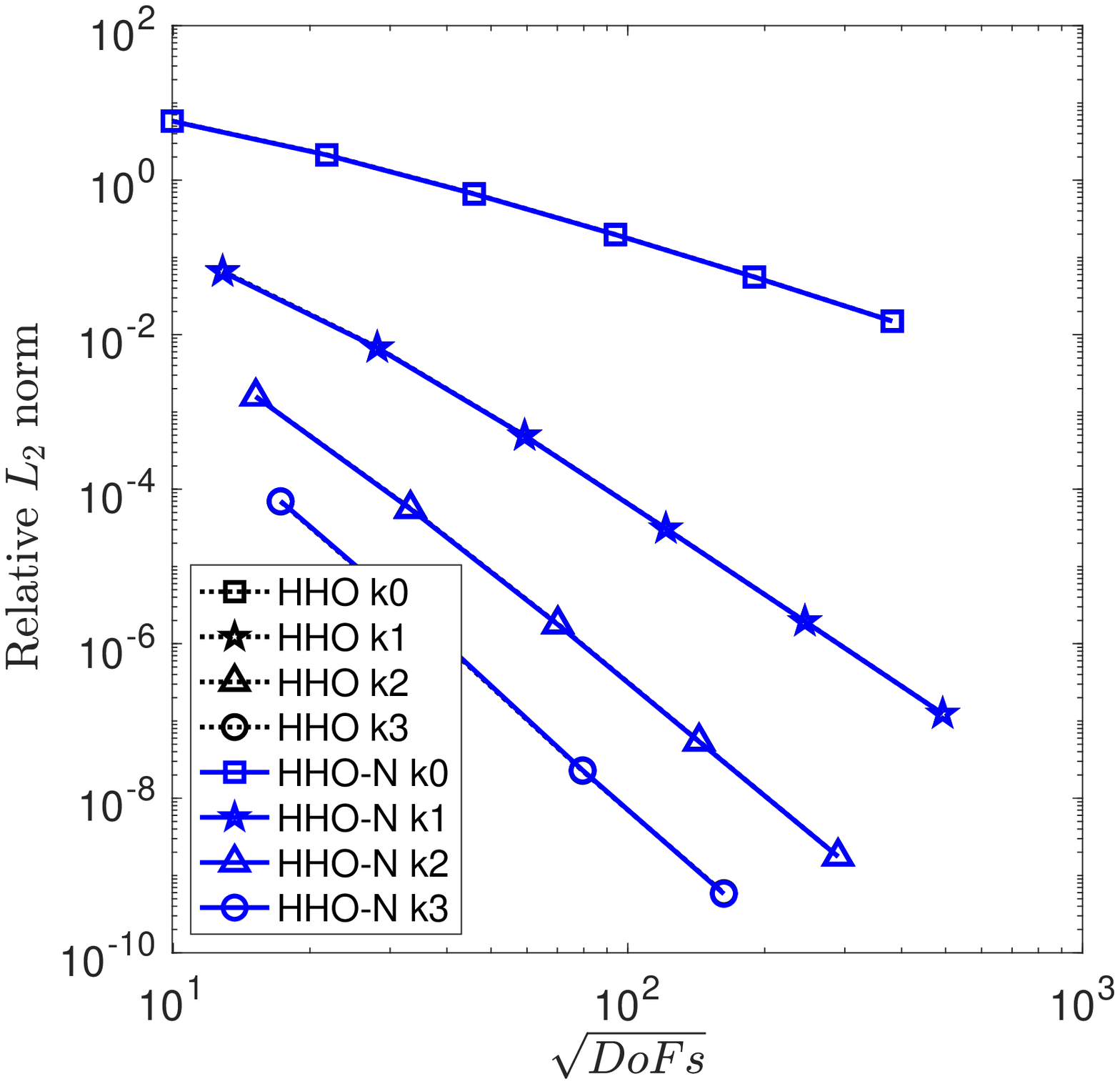}
\caption{
\label{ex3:h-refine}
Convergence of HHO and HHO-N methods in $H^2$- and $L^2$-(semi)norms on polygonal meshes.}
\end{figure}

\subsection*{Acknowledgment}
The use of the NEF computing platform at Inria Sophia Antipolis M\'editerran\'ee is gratefully acknowledged.

\bibliographystyle{siam}
\bibliography{HHO_N}

\end{document}